%% file: plan.tex
\documentclass[11pt,a4paper]{article}
\input{header}

\title{Skeleta in non-Archimedean and tropical geometry}
\author{Andrew W. Macpherson}
\begin{document}

\maketitle

\begin{abstract}
I describe an algebro-geometric theory of \emph{skeleta}, which provides a unified setting for the study of tropical varieties, skeleta of non-Archimedean analytic spaces, and affine manifolds with singularities. Skeleta are spaces equipped with a structure sheaf of topological semirings, and are locally modelled on the \emph{spectra} of the same. The primary result of this paper is that the topological space $X$ underlying a non-Archimedean analytic space may locally be recovered from the sheaf $|\sh O_X|$ of \emph{pointwise valuations} of its analytic functions; in other words, $(X,|\sh O_X|)$ is a skeleton.
\end{abstract}
\setcounter{tocdepth}{2}
\tableofcontents

\section{Introduction}

There are several areas in modern geometry in which one is led to consider spaces with \emph{affine} or \emph{piecewise affine} structure. The three with which I am in particular concerned are, in order of increasing subtlety: \begin{itemize}\item skeleta of non-Archimedean analytic spaces (\cite{Berkpadic}); \item tropical geometry (\cite{Grisha},\cite{BPR}); \item affine manifolds with singularities (\cite{Grossbook},\cite{KoSo2}).\end{itemize}
These cases share the following features: 
\begin{itemize}\item they are piecewise manifolds; \item it makes sense to ask which continuous, real-valued functions are \emph{piecewise affine}; \item they admit a stratification on which it makes sense to ask which of these are \emph{convex}.\end{itemize}

Moreover, in each case the structure is determined entirely by an underlying space $B$, together with a sheaf \[|\sh O_B|\subseteq C^0\left(B;\R\sqcup\{-\infty\}\right)\] of piecewise-affine, convex (where this is defined) functions.

The sheaf $|\sh O_B|$ is naturally a sheaf of \emph{idempotent semirings} under the operations of pointwise maximum and addition. It has long been understood, at least in the tropical geometry community (cf. e.g. \cite{Grisha}) that such semirings are the correct algebraic structures to associate to piecewise-affine geometries like $B$.

A natural question to ask is whether this sheaf-theoretic language can be pushed further in this setting and, as in algebraic geometry, the underlying space $B$ recovered from the semirings of local sections of $|\sh O_B|$. In this paper, I provide an affirmative answer to this question, though, as for the passage from classical algebraic geometry to scheme theory, it will require us to alter our expectations of what type of space underlies a piecewise-affine geometry. The resulting theory is what I call the theory of \emph{skeleta}.

The relationship between the theories of schemes and of skeleta goes beyond mere analogy: they can in fact be couched within the \emph{same} theoretical framework (appendix \ref{TOPOS}), \`a la Grothendieck (cf. also \cite{Toen}, \cite{Durov}). Within this framework, one need only specify which semiring homomorphisms
\[ \Gamma(U;|\sh O_U|)\rightarrow \Gamma(V,|\sh O_V|) \]
are dual to open immersions $V\hookrightarrow U$ of skeleta. This is enough to associate to every semiring $\alpha$ a quasi-compact topological space, its \emph{spectrum} $\Spec\alpha$. Skeleta can then be defined to be those semiringed spaces locally modelled by the spectra of semirings.

My main contention in this paper is that the \emph{primary} source of skeleta is the non-Archimedean geometry, and this is why I have adopted the terminology of this field. The initial concept that links non-Archimedean and piecewise-affine geometry is that of a valuation. Indeed, semirings are the natural recipients of valuations, while topological rings are the sources.

The topology of skeleta is selected so as to ensure that there is a unique functor
\[ \mathrm{sk}:\mathbf{Ad}\rightarrow\mathbf{Sk} \]
from the category $\mathbf{Ad}$ of adic spaces to the category $\mathbf{Sk}$ of skeleta, a natural homeomorphism $X\widetilde\rightarrow\mathrm{sk}X$ for $X\in\mathbf{Ad}$, and a \emph{universal valuation} $\sh O_X\rightarrow|\sh O_X|$. This \emph{universal skeleton} $\mathrm{sk}X$ of an analytic space $X$ can be thought of as the skeleton whose functions are the pointwise logarithmic norms of analytic functions on $X$. In particular, $X$ is locally the spectrum of the semiring of these functions.

The existence of this functor is the primary result of this paper. I also recover within the category of skeleta certain further examples that already existed in the literature: the \emph{dual intersection} or \emph{Clemens complex} of a degeneration (\S\ref{EGS-Clemens}), and the \emph{tropicalisation} of a subvariety of a toric variety (\S\ref{EGS-trop}).

\subsection*{Gist}

The categories of skeleta (section \ref{SKEL}) and of non-Archimedean analytic spaces may be constructed in the same way: as a category of \emph{locally representable sheaves} on some site whose underlying category is opposite to a category of algebras (\`a la \cite{Toen}). As such, to build a bridge between the two categories, it is enough to build a bridge between the categories $\frac{1}{2}\mathbf{Ring}_t$ of \emph{topological semirings} (defs. \ref{semiring}, \ref{tsemiring}) and $\mathbf{nA}$ of \emph{non-Archimedean rings} (def. \ref{ADIC-def-na}), and to check that it satisfies certain compatibility conditions.

One can associate to any non-Archimedean ring $(A,A^+)$ a \emph{free semiring} $\B^c(A;A^+)$, which, as a partially ordered set, is the set of finitely generated $A^+$-submodules of $A$. The addition on $\B^c(A;A^+)$ comes from the multiplication on $A$. It comes with a \emph{valuation} \[ A\rightarrow \B^c(A;A^+),\quad f\mapsto (f) \] universal among continuous semivaluations of $A$ into a semiring whose values on $A^+$ are negative (or zero). In other words, $\B^c(A;A^+)$ corepresents the functor
\[ \mathrm{Val}(A,A^+;-):\frac{1}{2}\mathbf{Ring}_t\rightarrow\mathbf{Set} \]
which takes a topological semiring $\alpha$ to the set of continuous semivaluations $\val:A\rightarrow\alpha$ satisfying $\val|_{A^+}\leq0$.

In particular, if $A=A^+$, then $\B^cA^+:=\B^c(A^+;A^+)$ is the set of finitely generated ideals of $A^+$, or of finitely presented subschemes of $\Spf A^+$.

Everything in the above paragraph may also be phrased in the internal logic of topoi so that, for example, it makes sense to replace $A$ and $\alpha$ with \emph{sheaves} of non-Archimedean rings and semirings on a space. Thus if $X$ is a non-Archimedean analytic space, then \[|\sh O_X|:U\mapsto\B^c(\sh O_U;\sh O_U^+)\] is a sheaf of topological semirings on $X$, universal among those receiving a continuous semivaluation from $\sh O_X$.

\begin{thrm}[\ref{SKEL-main}]Let $X$ be quasi-compact and quasi-separated. There is a natural homeomorphism \[X\widetilde\rightarrow\Spec\B^c(\sh O_X;\sh O_X^+)\] which matches the structure sheaf on the right with $|\sh O_X|$ on the left.\end{thrm}

In particular, if $X$ is a qcqs formal scheme, then the spectrum of the semiring $\B^c\sh O_X$ of ideal sheaves on $X$ is naturally homeomorphic to $X$ itself.

This skeleton $\Spec\B^c(\sh O_X;\sh O_X^+)$ is called the \emph{universal skeleton} $\mathrm{sk}X$ of $X$. It follows from the universal property of its structure sheaf that the \emph{real} points $\mathrm{sk}X(\R_\vee)$ can be identified canonically with the Berkovich analytic space associated to $X$ \cite[\S1.6]{Berketale}, provided such a thing exists; see theorem \ref{thm-berk}.

\

A natural geometric counterpart to the universality of $\B^c$ might be to say that the universal skeleton of an analytic space is universal among skeleta $B$ equipped with a continuous map $\iota:B\rightarrow X$ and valuation $\sh O_X\rightarrow \iota_*|\sh O_B|$. However, my point of view is that the very construction of the universal semiring diminishes the importance of valuation theory in getting a handle on the geometry of $X$. It tends to be easier, and perhaps more natural, to construct skeleta $B$ with a morphism $X\rightarrow B$ in the \emph{opposite} direction.

For example, let $X^+\rightarrow\Spf\sh O_K$ be a simple normal crossings degeneration over a DVR $\sh O_K$, with general fibre $j:X\rightarrow X^+$ (so $X$ is an analytic space, smooth over $K$). The irreducible components $E_i$ of the central fibre $X^+_0$ of the degeneration generate a subring $|\sh O_{\mathrm{sk}(X,X^+)}|^\circ\subseteq \B^c\sh O_{X^+}$ whose elements are the ideals \emph{monomial} with respect to normal co-ordinates $(t=\prod_{i=1}^kx_i^{n_i})$. Their supports are the strata of $X^+_0$. Base-changing over $K$ yields the \emph{dual intersection semiring}
\[ \mathrm{Cl}(X,X^+)\hookrightarrow\B^c(\sh O_{X^+}\tens K;\sh O_{X^+}) \]
and, dually, \emph{dual intersection skeleton} $\Spec\mathrm{Cl}(X,X^+)=:\Delta(X,X^+)\stackrel{\mu}{\leftarrow} X$ (definition \ref{EGS-def-Clemens}). 

That $X$ is defined over $K$ means that the universal skeleton, dual intersection skeleton, and morphism $\mu$ are defined over its value group: the \emph{semifield of integers} $\Z_\vee:=\Z\sqcup\{-\infty\}$. The real points $\mathrm{sk}(X,X^+)(\R_\vee)$ of the dual intersection skeleton are $\Z_\vee$-semialgebra homomorphisms $|\sh O_{\mathrm{sk}(X,X^+)}|\rightarrow\R_\vee$ to the \emph{real semifield} $\R_\vee=:=\R\sqcup\{-\infty\}$. They can be identified with the points of the na\"ive dual intersection complex as defined in, for example, \cite[\S A.3]{KoSo2}. Indeed, the simplices of this complex are defined by the logarithms of local equations for the intersections of $X^+_0$:
\[ \frac{K\{x_1,\ldots,x_n\}}{(t=\prod_{i=1}^kx_i^{n_i})} \rightsquigarrow \frac{\Z_\vee\{X_1,\ldots,X_n\}}{(-1=\sum_{i=1}^kn_iX_i)}, \]
where the curly braces on the right-hand side signify that $X_i\leq 0$. The latter equation $1+\sum_{i=1}^kn_iX_i=0$ cuts the dual intersection simplex
\[ \mathrm{conv}\{(0,\ldots,0,-1/n_i,0,\ldots,0)\}_{i=1}^k\subset  \R_{\leq0}^n \] out of the negative orthant in $\R^n$.
Under this identification, the elements of the dual intersection semiring correspond to integral, piecewise-affine functions whose restriction to each cell is convex.

The construction of such skeleta, perhaps \emph{partial} skeleta of $X$, is the crux of the theory. At this point I know of only a few examples (\S\ref{EGS}).

\subsection*{An elliptic curve}

Let us consider now the case that $X^+=E^+/\sh O_K$ is an elliptic curve degenerating semistably to a cycle of $n>3$ $\P^1_k$s, which I denote $\{D_i\}_{i=1}^n\in\Z_\vee\{X;X^+\}$. Its general fibre $E/K$ is a Tate elliptic curve. The dual intersection skeleton $\Delta(E,E^+)$ is, at the level of real points, a cycle of $n$ unit intervals joined at their endpoints. The vertices $\{v_i\}_{i=1}^n$ correspond to the lines $D_i$. Functions are allowed to be concave at these vertices. In particular, the function $D_i$ takes the value -1 at $v_i$ and zero at the other vertices.

Now let us collapse one of the $D_i$s \[p_i:E^+\rightarrow E^+_i\] to an $A_1$ singularity. The special fibre of $E^+_i$ is now a cycle of $(n-1)$ $\P^1_k$s meeting transversally except at the discriminant locus of the blow-up, which now has the local equation $(xy-t^2)$. With these co-ordinates, $p_i$ is the blow-up of the ideal $(x,y,t)$. 

In the semialgebraic notation, the ideal is \[(x,t,y) = D^\prime_{i-1}\vee -1\vee D^\prime_{i+1}\in \mathrm{Cl}(E,E^+) ,\] where $D^\prime_j$ denotes the divisor whose strict transform under $p_i$ is $D_j$, so $p_i^*D_{i\pm1}^\prime=D_{i\pm 1}+D_i$. The blow-up is monomial, and hence induces a pullback homomorphism $p_i^*:\mathrm{Cl}(E,E^+)\rightarrow\mathrm{Cl}(E,E^+)$, and dually, a morphism
\[ p_i:\Delta(E,E^+)\rightarrow \Delta(E,E^+_i) \]
of the dual intersection skeleta.

\



The segment of the dual intersection complex corresponding to the singular intersection $D_{i-1}\cap D_{i+1}$ is an interval $I$ of affine length two. Considered as a function on $I$, the blow-up ideal $D^\prime_{i-1}\vee -1\vee D^\prime_{i+1}$ has real values as the absolute value
\[ |-|:I\simeq [-1,1]\rightarrow\R. \]
It has a kink in the middle. Because, in $\Delta(E,E^+_i)$, there is no vertex here, the inverse of this function is not allowed; while of course the pullback $D_i$ \emph{is} invertible on $\Delta(E,E^+)$.

In fact, $\mathrm{Cl}(E,E^+)$ is a \emph{localisation} of $\mathrm{Cl}(E,E^+_i)$ at $D^\prime_{i-1}\vee -1\vee D^\prime_{i+1}$, and $\Delta(E,E^+)$ is an \emph{open subset} of $\Delta(E,E^+_i)$.

Varying $i$, we obtain therefore an atlas
\[ \coprod_{i\neq j=1}^n\Delta(E,E^+) \rightrightarrows \coprod_{i=1}^n\Delta(E,E^+_i) \]
for a skeleton $B$ which \emph{compactifies} $\Delta(E,E^+)$. Functions on $B$ are required to be convex everywhere, and $B(\R_\vee)$ is, as an affine manifold, the flat circle $\R/n\Z$.

This skeleton is a kind of \emph{Calabi-Yau skeleton} of $E$, and it depends only on the intrinsic geometry of $E$ and not on any choices of model. See also section \ref{EGS-ellipt}.

\subsection*{Mirror symmetry context}

Conjectures arising from homological mirror symmetry \cite{KoSo1} suggest that a Calabi-Yau $n$-fold $X$ approaching a so-called \emph{large complex structure limit} acquires the structure of a completely integrable system $\mu:X\rightarrow B$ with singularities in (real) codimension one, shrinking to two in the limit. The base $B$ therefore acquires the structure of a Riemannian $n$-manifold with integral affine co-ordinates $y_i$, away from the singular fibres, given by the Hamiltonians of the system. The metric is locally the Hessian, with respect to these co-ordinates, of a convex function $K$, and satisfies the \emph{Monge-Amp\`ere equation}
\[ d\det\left(\frac{\partial^2K}{\partial y_i\partial y_j}\right)=0 \]
which can be thought of as the `tropicalisation' of the complex Monge-Amp\`ere equation satisfied by the Yau metric.

The central idea of \cite{KoSo2} is that the skeleton $B$ can be constructed, with the \emph{Legendre dual} affine structure $\check y_i$, from the non-Archimedean geometry of $X^\mathrm{an}$ or, what is the same thing, the birational geometry of its formal models. Indeed, Kontsevich noted that the Gromov-Hausdorff limit of $X$ should resemble the dual intersection complex of a certain `crepant' model thereof. To be precise, the real points of $B$ should be embeddable into $X^\mathrm{an}(\R_\vee)$ as the dual intersection complex of any dlt minimal model of $X$ \cite{Nicaise}. Its structure as a dual intersection complex also endows it with the correct affine structure, away from a subset of codimension one which contains the singularities.

More subtle is to construct the correct \emph{non-Archimedean torus fibration} $\mu:X\rightarrow B$. This would also determine the affine structure of $B$ in the sense that 
\[ |\sh O_B|\cong\mathrm{Im}(\mu_*\sh O_X\rightarrow\mu_*|\sh O_X|). \]
Such a $\mu$ is determined by a choice of minimal model. Unfortunately, in dimensions greater than one, the morphisms $\mu$ coming from various models differ. The affine structures they induce are related by so-called \emph{worm deformations}, which move the singularities of the affine structure along their monodromy-invariant lines. These deformations correspond to flops in birational geometry.

This forms the basis of a dictionary, motivated by mirror symmetry, between concepts in birational geometry and the tropical geometry of affine manifolds. This dictionary has been partially developed along combinatorial lines in the Gross-Siebert programme.\footnote{In general the Gross-Siebert programme \cite{GS} bypasses the non-Archimedean geometry to give a direct construction of the affine structure of $B$, up to worm deformations, in terms of toric geometry. Using this approach, they were able to obtain many results with a combinatorial flavour, and even a reconstruction of $X$ (as an algebraic variety) from $B$ together with some cocycle data. To mimic at least their basic construction in the context of skeleta is not difficult, but beyond the scope of this paper.}
However, geometrically interesting examples present enormous combinatorial complexity, already for the case of K3 surfaces. I propose that a more geometric approach, such as outlined in this paper, will be more robust in such applications.

There is some hope that, armed with a suitably flexible language, the birational geometry of $X$ together with a polarisation can be used to construct solutions to a real Monge-Amp\'ere equation on $B$.

\subsection{How to read this paper}

The structure of the paper is as follows. In the first three sections, we establish the theory of semirings and their (semi)modules as a theory of commutative algebras in a certain closed monoidal category, the category of \emph{$\B$-modules} $(\mathrm{Mod}_\B,\oplus)$. The objects of $\mathrm{Mod}_\B$ are also known in the literature as `join-semilattices'. Since we wish to compare with non-Archimedean geometry, we actually need to work with \emph{topological} $\B$-modules (\S\ref{TOP}). At this paper's level of sophistication, this causes few complications.

Apart from establishing the formal properties of the categories of $\B$-modules and semirings, the secondary thrust of this part is to introduce various versions of the \emph{subobject} and \emph{free} functors \begin{align}\nonumber \B,\B^c:&\mathrm{Mod}_A\longrightarrow\mathrm{Mod}_\B \\
\nonumber &\mathbf{Ring}\longrightarrow\frac{1}{2}\mathbf{Ring} \\
\nonumber & etc. \end{align}
which will pave the major highway linking algebraic and `semialgebraic' geometry. I have spelled out in some detail the functoriality of these constructions, though they are mostly self-evident.

\

Section \ref{LOC} sets about defining the localisation theory of semirings, which is designed to parallel the one used for topological rings in non-Archimedean geometry. These \emph{bounded localisations} factorise into two types: cellular, and free. The latter resemble ordinary localisations of algebras, and the algebraically-minded reader will be unsurprised by their presence. The cellular localisations, on the other hand, may be less familiar: they involve the non-flat operation of setting a variable equal to zero. Thinking of a skeleton as a polyhedral or cell complex, these localisations will be dual to the inclusions of cells (of possibly lower dimension). Perhaps confusingly, these are the semiring homomorphisms that correspond, under $\B^c$, to open immersions of formal schemes. The precise statements are the Zariski-open (\ref{LOC-Zar-open}) and cellular cover (\ref{LOC-Zar-cover}) formulas.

With some understanding of the `cellular topology' we are able, as an aside, to describe the spectrum of \emph{contracting} semirings in terms of a na\"ive construction: the \emph{prime spectrum} \S\ref{LOC-prime}. This makes clear the relationship between the topological space underlying a formal scheme $X^+$ and the spectrum of the ideal (sheaf) semiring $\B^c\sh O_X^+$.

It is also easy to describe the free localisations in terms of the polyhedral complex picture. Inverting a strictly convex function has the effect of destroying the affine structure along its non-differentiability locus (or `tropical set'); we therefore think of it as further subdividing our complex into the cells on which the function is affine. We can also give an algebro-geometric interpretation of these subdivisions: it is given by the \emph{blow-up formula} (\ref{LOC-blow-up}). In the setting of a formal scheme $X^+$ over a DVR $\sh O_K$, it says that blowing up an ideal sheaf $J$ supported on the reduction has the effect of inverting $J$ in $\B^c(\sh O_{X^+}\tens K;\sh O_{X^+})$. Intuitively, the blow-up of $J$ is the universal way to make it an invertible sheaf.

\

In \S\ref{SKEL} we meet the category $\mathbf{Sk}$ of skeleta, and introduce some universal constructions of certain skeleta from adic spaces and their models. The construction of this category follows the general programme of glueing objects inside a topos, as outlined in \cite{Toen}. The main result \ref{SKEL-main} - which concerns the main skeletal invariant of an analytic space $X$, the \emph{universal skeleton} $\mathrm{sk}X$ - boils down to proving that for reasonable values of $X$, the topological space underlying $X$ can be identified with that of $\Spec\Gamma(X;|\sh O_X|)$. The technical part of the proof is based on the Zariski-open, blow-up, and cellular cover formulas, which together allow us to explicitly match the open subsets of $X$ with those of its skeleton.
 
As an artefact of the proof, we may notice that a surprisingly many skeleta - those associated to any quasi-compact, quasi-separated analytic space - are affine. As an aside in $\S\ref{SKEL}$ I was able to obtain a kind of quantification (thm. \ref{SKEL-qc=aff}) of this observation. We also glance at the relationship (thm. \ref{thm-berk}) between skeleta and the theory of Berkovich.

In the examples section \ref{EGS}, we reconstruct some well-known `tropical spaces' as skeleta: the dual intersection complexes of locally toric degenerations (\S\ref{EGS-Clemens}), and the tropicalisations of subvarieties of a toric variety \cite{Payne} (\S\ref{EGS-trop}). I have also attempted to couch the construction of an affine manifold from a Tate elliptic curve, summarised above, in more general terms (\S\ref{EGS-ellipt}). This forms the first test case of an ongoing project.

\subsection*{Acknowledgements}

I would like to thank my PhD supervisors, Alessio Corti and Richard Thomas, for their support. I thank also Mark Gross, Sam Payne, and Jeff Giansiracusa, for interesting conversations.

I also thank the Cecil King foundation for funding my visit to Mark Gross in UCSD, where some of the aforementioned conversations, as well as part of the work writing this paper, occurred.

\section{Preliminaries and conventions}

\subsection{On sites and topoi}\label{TOPOS}

Our general notational conventions on sites and topoi follow the canonical \cite{SGA4}. The central glueing construction of the paper revolves around the notion of a \emph{locally representable sheaf}, defined in \cite[def. 2.15]{Toen}. I only wish to replace the input, the authors' notion of \emph{Zariski-open immersion}, with something a bit more flexible.

\begin{defn}\label{TOPOS-def}Let $\sh U$ be a class of monomorphisms in a category $\mathbf C$ stable for composition and base change. One defines the structure of a Grothendieck site on $\mathbf C$ whose generating coverings are finite families of morphisms in $\sh U$ that form a covering for the canonical site.

An \emph{open immersion} in the associated topos $\mathbf C\topos$ is a morphism locally representable by morphisms in $\sh U$.

An object of $\mathbf C\topos$ is \emph{locally representable} if it is a union of representable open subobjects.\end{defn}

Much of the theory of \cite{Toen} is valid with this more flexible input, notably proposition 2.18. I warn the reader only that without a \emph{locality} requirement for our definition of affine open immersion, part 2 of \cite[prop. 2.17]{Toen} is false. This is the case, for example, for the category of adic spaces (\S\ref{ADIC}).

The resulting class of objects can also be characterised in terms of point-set topology via a modern analogue of Stone's construction:

\begin{enumerate}
\item By construction, $\mathbf C\topos$ is a coherent topos and so by Deligne's theorem \cite[\S IX.11.3]{Topos} it has enough points.
\item Since the morphisms in $\sh U$ were assumed to be monic, the small topos of an object $X\in\mathbf C\topos$ is localic.
\item Being localic and having enough points, the small topos of an object is therefore equivalent to a uniquely determined sober topological space \cite[\S IX.3.1-4]{Topos}. This determines a functor
\[  \mathbf C\topos\rightarrow\mathbf{Top} \]
that takes morphisms in $\sh U$ to open immersions. 
\item The topological space associated to a representable (or more generally, compact) object is quasi-compact and quasi-separated.
\item Being locally representable corresponds to having a basis of open sets coming from open immersions with representable source.
\end{enumerate}

A covering - or, more precisely, two-term hypercovering - of a space $X$ will be denoted $U_\bullet\twoheadrightarrow X$, with the bullet ranging over a partially ordered set of indices. 

\subsection{Non-Archimedean geometry}\label{ADIC}

The perspective on non-Archimedean geometry taken in this paper was influenced by the foundational works \cite{Hubook} and \cite{FujiKato}. Broadly speaking, I have adopted the categorical localisation constructions of the latter (after the approach of Raynaud), but the language and notation of the former - in particular, the nomenclature \emph{adic spaces}.

I introduce the following innovations in terminology:

\begin{defns}\label{ADIC-def-na}A \emph{marked formal scheme} is a pair $(X^+,Z)$ consisting of a formal scheme $X^+$ and a collection $Z$ of Cartier divisors. A morphism of marked formal schemes is a morphism $f:X^+_1\rightarrow X^+_2$ such that $(f^{-1}Z_2)^\mathrm{red}\subseteq Z_1$. An admissible blow-up is a finite type blow-up whose centre has underlying reduced scheme contained in $Z$.

A \emph{non-Archimedean ring} is a pair $(A,A^+)$ consisting of an adic ring $A^+$ and a localisation $A$ of $A^+$. We only consider \emph{locally convex} $(A,A^+)$-modules, that is, complete topological $A$-modules whose topology is generated by $A^+$-submodules; the category of such is denoted $\mathrm{LC}_{(A,A^+)}$, or just $\mathrm{LC}_A$ for short.\end{defns}

\begin{itemize}
\item The category $\mathbf{Ad}$ of adic spaces is defined by the same means as the category $\mathbf{Rf}$ of \cite[\S\textbf{II}.2]{FujiKato}, with the modification that the input to the localisation construction is instead the category of coherent \emph{marked formal schemes} at admissible blow-ups, as in def. \ref{ADIC-def-na}. This ensures that the notion of adic space is a generalisation of that of formal scheme.
\item The glueing construction of \cite[\S\textbf{II}.2.2(c)]{FujiKato}, although expressed in less standard language, is identical to the locally representable sheaves story of \S\ref{TOPOS}.
\item Following Huber, the sheaf of functions extending over a model is denoted $\sh O^+$ (rather than $\sh O^\mathrm{int}$ as in \cite[\S\textbf{II}.3.2(a)]{FujiKato}). The structure sheaf of the adic topos $\mathbf{Ad}\topos$ is a \emph{pair} $(\sh O, \sh O^+)$. It is a sheaf of non-Archimedean rings in the sense of def. \ref{ADIC-def-na}. 
\item Accordingly, we also adopt the notation $X^+$ for models of a adic space $X$. The category of models is denoted $\mathbf{Mdl}_{X^+}$; if $X$ is qcqs, it is cofiltered. The map $j:X\rightarrow X^+$ exhibiting the model is a morphism of adic spaces.
\item An \emph{affine adic space} $X$ is one admitting an affine formal model whose marking divisors are principal. By definition, this space is the spectrum of the non-Archimedean ring $A:=\Gamma\sh O_X$; following Huber, this spectrum is denoted $\Spa A$.\end{itemize}

The key aspect of this construction that we will use is that for quasi-compact, quasi-separated $X$, as topologically ringed sites,
\begin{align}\label{important} (X,\sh O_X^+)\simeq\lim_{X^+\in\mathbf{Mdl}(X)}X^+ \end{align}
where the limit is over all formal models of $X$.

\input{SPAN}

\input{TOP}

\input{RING}

\input{LOC}

\input{SKEL}

\input{EGS}




\bibliographystyle{alpha}
\bibliography{plan}

\end{document}

%% file: header.tex
\usepackage[dvipsnames]{xcolor}
\usepackage[all,cmtip]{xy}
\entrymodifiers={+!!<0pt,\fontdimen22\textfont2>}

\usepackage{fouriernc}
\usepackage[english]{babel}         
\usepackage{csquotes}
\usepackage{verbatim}               
\usepackage[T1]{fontenc}

\usepackage{amsmath}
\usepackage{amsfonts}
\usepackage{amssymb}
\usepackage{mathrsfs}    
\usepackage{amsthm}

\usepackage[colorlinks=true,linkcolor=MidnightBlue,citecolor=OliveGreen]{hyperref}           

\usepackage{epsfig}   

\swapnumbers

\theoremstyle{plain}
\newtheorem*{thrm}{Theorem}
\newtheorem{thm}{Theorem}[section]

\newtheorem{prop}[thm]{Proposition}
\newtheorem{cor}[thm]{Corollary}
\newtheorem{lemma}[thm]{Lemma}
\newtheorem*{conj}{Conjecture}
\theoremstyle{definition}

\newtheorem{defn}[thm]{Definition}
\newtheorem{defns}[thm]{Definitions}

\theoremstyle{remark}

\newtheorem{eg}[thm]{Example}
\newtheorem{egs}[thm]{Examples}
\newcommand{\proofof}[1]{\end{#1}\begin{proof}}

\makeatletter
\renewcommand\section{\@startsection {section}{1}{\z@}%
  {-3.5ex \@plus -1ex \@minus -.2ex}{2.3ex \@plus.2ex}%
  {\normalfont\large\bfseries}}
\renewcommand\subsection{\@startsection{subsection}{2}{\z@}%
  {-3.25ex\@plus -1ex \@minus -.2ex}{1.5ex \@plus .2ex}%
  {\normalfont\bfseries}}

\newcommand{\lie}[1]{\mathfrak{#1}}
\newcommand{\sh}[1]{\mathcal{#1}}
\newcommand{\N}{{\mathbb N}}
\newcommand{\Z}{{\mathbb Z}}
\newcommand{\Q}{{\mathbb Q}}
\newcommand{\R}{{\mathbb R}}
\newcommand{\C}{{\mathbb C}}

\newcommand{\B}{{\mathbb B}}

\renewcommand{\P}{{\mathbb P}}
\newcommand{\A}{{\mathbb A}}
\newcommand{\D}{{\mathbb D}}

\newcommand{\tens}{\mathbin{\otimes}}

\newcommand{\iden}{\mathrm{id}}

\newcommand{\topos}{\hspace{4pt}\widetilde{}\hspace{4pt}}

\newcommand{\Hom}{\mathrm{Hom}}

\DeclareMathOperator*\colim{colim}
\DeclareMathAlphabet{\mathrmsl}{OT1}{cmr}{m}{sl}

\newcommand{\rssymb}[2]{\newcommand{#1}{\mathrmsl{#2}} }
\newcommand{\oper}[3][n]{\newcommand{#2}{\mathop{\mathrm{#3}}%
\ifx n#1\nolimits\else\limits\fi} }
\newcommand{\rsoper}[3][n]{\newcommand{#2}{\mathop{\mathrmsl{#3}}%
\ifx n#1\nolimits\else\limits\fi} }

\oper\Ad{Ad}
\oper\ad{ad}

\oper\val{val}
\oper\coker{coker}
\oper\mult{mult}
\oper\Iso{Iso}
\oper\End{End}
\oper\Aut{Aut}
\oper\Sub{Sub}
\oper\Alt{Alt}
\oper\Ext{Ext}
\oper\Pic {Pic}
\oper\Sym{Sym}
\oper\Spec{Spec}
\oper\Spf{Spf}
\oper\Sp{Sp}
\oper\Spa{Spa}
\oper\Proj{Proj}
\rsoper\divg{div}
\rsoper{\sym}{sym}
\rsoper{\alt}{alt}
\rsoper\trace{tr}
\rssymb\id{id}

\newcommand{\thismonth}{\ifcase\month\or
  January\or February\or March\or April\or May\or June\or
  July\or August\or September\or October\or November\or December\fi
  \space\number\year}



%% file: SPAN.tex
\section{Subobjects and $\B$-modules}\label{SPAN}

The theory of \emph{$\B$-modules} plays the same r\^ole in tropical geometry that the theory of Abelian groups plays in algebraic geometry: while rings are commutative monoids in the category of Abelian groups, semirings are commutative monoids instead in the category of $\B$-modules. This is the fundamental point of departure of the two theories. There is therefore a temptation to try to treat $\B$-modules as "broken" Abelian groups, and to literally translate as many concepts and constructions from the category $\mathbf{Ab}$ as will survive the transition.

In this paper, I adopt a different perspective. A $\B$-module is a particular type of partially ordered set which axiomatises some properties of \emph{subobject posets} in Abelian and similar categories. In particular, there is a functor $\B:\mathbf{Ab}\rightarrow\mathrm{Mod}_\B$ which associates to an Abelian group its $\B$-module of subgroups. As such, I propose to treat $\B$-modules as though they are \emph{lower categorical shadows} of structures in the \emph{category} of Abelian groups, rather than simply as elements of a single Abelian group. The theory of $\B$-modules is a na\"ive form of category theory, rather than a weak form of group theory.

There is also a dual, or more precisely, \emph{adjoint}, perspective, which is that a $\B$-module is the natural recipient of a \emph{non-Archimedean seminorm} from an Abelian group. This fits well with traditional perspectives on non-Archimedean geometry. In keeping with the ahistorical nature of this paper, I barely touch upon this idea here (but see example \ref{eg-discs'}).

\subsection{$\B$-modules}\label{SPAN-span}

\begin{defn}A \emph{$\B$-module} is an idempotent commutative monoid. In other words, it is a commutative monoid $(\alpha,\vee,-\infty)$ in which the identity
\[ X\vee X=X \]
holds for all $X\in\alpha$, and where $-\infty$ is the identity for $\vee$. The category of $\B$-modules and their homomorphisms will be denoted $\mathrm{Mod}_\B$.\end{defn}

A $\B$-module is automatically a partially ordered set with the relation
\[ X\leq Y \quad \Leftrightarrow \quad X\vee Y=Y. \]
It has all finite joins (suprema). Conversely, any poset with finite joins is a $\B$-module under the binary join operation. They are more commonly called \emph{join semilattices} or simply \emph{semilattices}.\footnote{I abandon this terminology for a number of reasons, but one could be the inconsistency of the r\^oles of the modifier \emph{semi} in the words `semiring' and `semilattice'.}

We may therefore introduce immediately a path to category theory in the form of an essentially equivalent definition.

\begin{defn}\label{SPAN-def-preorder}A \emph{$\B$-module} is a preorder with finite colimits. A \emph{$\B$-module homomorphism} is a right exact functor.\end{defn}

\begin{egs}\label{eg-first}The \emph{null} or \emph{trivial} $\B$-module is the $\B$-module with one element $\{-\infty\}$. The \emph{Boolean semifield}
is the partial order $\B=\{-\infty,0\}\simeq\{\texttt{false,true}\}$.

The \emph{integer, rational}, and \emph{real semifields} $\Z_\vee,\Q_\vee,\R_\vee$ are obtained by disjointly affixing $-\infty$ to $\Z,\Q,\R$, respectively. More generally, we can obtain a semifield $H_\vee$ by adjoining $-\infty$ to any totally ordered Abelian group $H$. These semifields are totally ordered $\B$-modules (in fact, semirings; cf. e.g. \ref{semifields}).

If $X$ is a topological space, the set $C^0(X,\R_\vee)$ of continuous functions $X\rightarrow\R_\vee$, where $\R_\vee$ is equipped with the order topology, is a $\B$-module. So too are the subsets of bounded above functions, or of functions bounded above by some fixed constant $C\in\R$.

Suppose that $X$ is a manifold (with boundary). The subset $C^1(X,\R_\vee)$ of \emph{differentiable} functions is not a $\B$-module, since the pointwise maximum $f\vee g$ of two differentiable functions $f,g$ needn't be differentiable. One must allow \emph{piecewise} differentiable $\mathrm{P}C^1$ (or piecewise smooth $\mathrm{P}C^\infty$) functions to obtain submodules of $C^0(X,\R_\vee)$. Since convexity is preserved under $\vee$, the subsets of \emph{convex} functions $\mathrm{CP}C^r(X;\R_\vee)$ are also submodules.

We can also endow $X$ with some kind of affine structure \cite[\S2.1]{KoSo2}, which gives rise to $\B$-modules $\mathrm{CPA}_*(X,\R_\vee),*=\R,\Q,\Z$ of piecewise-affine, convex functions (with real, rational, or integer slopes, respectively). 
If $X=\Delta\subset\R^n$ is a polytope, then it has a notion of integer points, and so one can define a $\B$-module $\mathrm{CPA}_\Z(X,\Z_\vee)$ of piecewise-affine, convex functions with integer slopes and which take \emph{integer values} on lattice points $\Z^n\cap\Delta$. Note that any function in this $\B$-module that attains the value $-\infty$ must in fact be constant.\end{egs}

\begin{eg}\label{PFR-set}Let $S$ be a set. The \emph{subset $\B$-module} $\B S$ is the power set of $S$; its join operation is union. The \emph{free $\B$-module} $\B^c S\subseteq \B S$ is the set of finite subsets of $S$. Its elements may be written uniquely (up to permutation of terms) as idempotent linear expressions ``with coefficients in $\B$,'' i.e.\ as $X_1\vee\cdots\vee X_k$ for some $X_1,\ldots,X_k\in S$.

Both constructions are functorial in $S$, so we have functors $\B,\B^c:\mathbf{Set}\rightarrow\mathrm{Mod}_\B$; the latter is left adjoint to the forgetful functor.\end{eg}

The theory of $\B$-modules is a finitary algebraic theory, and so limits, filtered colimits, and quotients by groupoid relations are computed in $\mathbf{Set}$; this remains true with $\mathbf{Set}$ replaced by any topos. The following (proposition \ref{SPAN-prop-bicomp}) also remains true in that generality.

For any $\B$-modules $\alpha,\beta$, we can construct the \textit{direct join} $\alpha\vee\beta$ as the $\B$-module whose underlying set is the Cartesian product $\alpha\times\beta$ and whose join is defined by the law \[ (X_1,Y_1)\vee(X_2,Y_2):=(X_1\vee X_2,Y_1\vee Y_2). \] I simply write $X_1\vee X_2$ for $(X_1,X_2)$ where this is not likely to cause confusion.

There are natural $\B$-module homomorphisms $\alpha\rightarrow\alpha\vee\beta\rightarrow\alpha$ defined by \[X\mapsto X\vee(-\infty), \quad X\vee Y\mapsto X,\] and similarly for $\beta$, which make the direct join into a coproduct and product in $\mathrm{Mod}_\B$. In particular, there are natural homomorphisms \[\alpha\stackrel{\Delta}{\longrightarrow}\alpha\vee\alpha \stackrel{\vee}{\longrightarrow} \alpha,\] the diagonal and the map defining the $\B$-module structure, respectively. I use also the direct join notation for a pushout $\alpha\vee_\beta\gamma:=\alpha\sqcup_\beta\gamma$.

The null $\B$-module is the empty direct join, or zero object, of $\mathrm{Mod}_\B$.
The kernel and cokernel of a morphism $f:\alpha\rightarrow\beta$ of $\B$-modules are defined: $\ker f:=f^{-1}(-\infty),\coker f = \beta\vee_\alpha\{-\infty\}$.

If $f,g\in\Hom(\alpha,\beta)$, then their `sum' is given by the composition
\[\alpha\stackrel{\id\times \id}\longrightarrow\alpha\vee\alpha\longrightarrow\beta\vee\beta\stackrel{\id\sqcup\id}\longrightarrow\beta\]
which takes $X\in\alpha$ to $f(X)\vee g(X)$. This description establishes that the monoid $\Hom(\alpha,\beta)$ is in fact a $\B$-module in which $f\leq g$ if and only if $f(X)\leq g(X)\in\beta$ $\forall X\in\alpha$; moreover $\Hom(-,-)$ is a bifunctor from $\mathrm{Mod}_\B$ to itself. 

\begin{prop}\label{SPAN-prop-bicomp}The category $\mathrm{Mod}_\B$ is semiadditive.\footnote{A category is \emph{semiadditive} if it admits finite products and coproducts and the natural map $\times\rightarrow\sqcup$ is an isomorphism of bifunctors.} It is complete and cocomplete.\end{prop}

It is harder to obtain an explicit description of general coequalisers; see \S\ref{SPAN-quotient}.

\subsubsection{Subobjects}

Beyond the geometric examples \ref{eg-first}, the primary source of $\B$-modules are the \emph{subobject posets} in certain finitely cocomplete categories. One could formulate a general theory of subobjects in certain kinds of categories; however, for the purposes of this paper we only need to know the version for modules over a commutative ring (possibly in a Grothendieck topos).

\begin{defn}\label{SPAN-def-B}Let $A$ be a ring, $M$ an $A$-module. I write $\B(M;A)$ for the \emph{submodule lattice} of $M$, the partially ordered set of all $A$-submodules of $M$; its join operation is submodule sum. I abbreviate $\B(A;A)$ to $\B A$, the \emph{ideal semiring} of $A$.\end{defn}

The submodule lattice is functorial in $A$-module homomorphisms $f:M_1\rightarrow M_2$
\[\B f:\B(M_1;A)\rightarrow\B(M_2;A),\quad N\mapsto\mathrm{Im}(f|_N) \]
and ring maps $g:A\rightarrow B$
\[ \B g:\B(M;A)\rightarrow\B(M\tens_AB;B),\quad N \mapsto\mathrm{Im}(N\tens_AB\rightarrow M\tens_AB).\]
In particular, $\B A\rightarrow \B B$.

Typically, $A=:\sh O_X$ will be a sheaf of rings on some space $X$ and $M$ an $\sh O_X$-module, in which case $\B(M;\sh O_X)$ is the lattice of $\sh O_X$-subsheaves of $M$. The submodule lattice is then functorial for maps defined in the sheaf category $X\topos$, but also for morphisms $g:(Y,\sh O_Y)\rightarrow (X,\sh O_X)$ of ringed spaces. In the latter case, I will write \[ g^*=\B g:\B(M;\sh O_X)\rightarrow\B(g^*M;\sh O_Y),\quad N \mapsto\mathrm{Im}(g^*N\rightarrow g^*M) \]
 for the induced map of lattices, though this should not be confused with the functor of pullback of $\sh O_Y$-modules, which it equals only when $g$ is flat.

\begin{eg}[Discs]\label{eg-discs}Let $K$ be a complete, valued field, $V$ a $K$-vector space. I would like to be able to say that the subobjects of $V$ are the \emph{discs} \cite[\S 2.2]{Banach}. If $K$ is non-Archimedean with ring of integers $\sh O_K$, then a disc is the same thing as an $\sh O_K$-submodule, and so the set of discs is $\B(V;\sh O_K)$ (which in \emph{loc. cit.} is called $\sh D(V)$).

If $K$ is Archimedean, then we need an alternative theory of `abstract discs' or `convex sets'. Following \cite{Durov}, one can describe it as a theory of modules for a certain algebraic monad. For instance, if $K=\R=\Q_\infty$, the corresponding monad is that $\Z_\infty$ (also written $\sh O_\R$) of convex, balanced sets \cite[\S2.14]{Durov}. An object of $\mathrm{Mod}_{\Z_\infty}$ is a set $M$ equipped with a way of evaluating convex linear combinations \[\sum_{i=1}^k\lambda_ix_i,\quad x_i\in M,\lambda_i\in\R,\sum_{i=1}^k|\lambda_i|\leq 1\] of its elements. A subset of $V$ is a disc if and only if it is stable for the action of $\Z_\infty$. In other words, $\sh D(V)=\B(V;\Z_\infty)$, in a mild generalisation of definition \ref{SPAN-def-B}.\end{eg}

\subsection{Orders and lattices}

The alternative definition \ref{SPAN-def-preorder} puts $\B$-module theory in the broader context of order theory. In particular, there are \emph{sometimes} defined infinitary operations \[(X_i)_{i\in I}\mapsto \sup_{i\in I}X_i.\] I reserve the notation $\bigvee_{i=1}^kX_i$ for the (always defined) operation of finite supremum or join.

The following definitions are standard in order theory:

\begin{defns}A map of posets is \emph{monotone} if it preserves the order. A monotone map of posets is the same as a functor of preorders. The category of posets and monotone maps is denoted $\mathbf{POSet}$.

A $\B$-module is a \emph{complete lattice} if it has all suprema. A complete lattice is the same thing as a cocomplete poset. In particular, meets exist. A \emph{lattice homomorphism} is a map of complete lattices preserving all suprema, that is, a colimit-preserving functor. The category of complete lattices and homomorphisms is denoted $\mathbf{Lat}\subset\mathbf{Span}$.

Let $\alpha$ be a $\B$-module, $S,T\subseteq\alpha$. The \textit{lower slice set} 
\begin{align}\nonumber S_{\leq T} &:= \{Y\in S|\exists X\in T \text{ s.t. }X\vee Y=X\}\\
 & = \{Y\in S|\exists X\in T\text{ s.t. } Y\leq X\}\end{align} is the $\B$-module of all elements contained in $S$ that are bounded above by an element of $T$. The \emph{upper slice set} \[S_{\geq T} := \{Y\in S|\exists X\in T \text{ s.t. }X\vee Y=Y\}\] is defined dually.

A subset $S$ is said to be \emph{lower} (resp. \emph{upper}) if $S=\alpha_{\leq S}$ (resp. $\alpha_{\geq S}$). A lower submodule of $\alpha$ is called an \emph{ideal} of $\alpha$.

The subset $S$ is called \textit{coinitial} (resp. \emph{cofinal}) if all lower (resp. upper) slice sets are non-empty, that is, $\forall X\in\alpha$, $\exists Y\in S$ such that $Y\leq X$ (resp.\ $X\leq Y$).\end{defns}

\begin{eg}A quotient of a $\B$-module $\alpha$ by an ideal $\iota$, that is, the cokernel of the inclusion $\iota\hookrightarrow\alpha$, is easy to make explicit: it is simply the set-theoretic quotient $\alpha/\iota$ of $\alpha$ by the equivalence relation $\iota\sim-\infty$. If $\iota=\alpha_{\leq T}$ is a slice set, we may also write $\alpha/T$. The cokernel of a $\B$-module homomorphism $f:\alpha\rightarrow\beta$ is the quotient of $\beta$ by $\beta_{\leq f(\alpha)}$, the smallest ideal containing $f(\alpha)$. In particular, $\alpha$ is an ideal if and only if it is the kernel of its cokernel.\end{eg}

The set of all ideals of $\alpha$ can be thought of as a \emph{subobject poset} in the category of $\B$-modules. It is a complete lattice.

\begin{defn}\label{SPAN-def-lat}The lattice of ideals $\B\alpha$ of a $\B$-module $\alpha$ is called the \emph{lattice completion} of $\alpha$.\end{defn}

The lattice completion defines a left adjoint $\B:\mathrm{Mod}_\B\rightarrow\mathbf{Lat}$ to the inclusion of $\mathbf{Lat}$ into $\mathrm{Mod}_\B$.
The unit $\iden\rightarrow\B$ of the adjunction is an injective homomorphism
\[ \alpha\rightarrow\B\alpha, \quad X\mapsto \alpha_{\leq X}. \]
As a preorder, the lattice completion of $\alpha$ is its category of ind-objects \cite[\S I.8.2]{SGA4}.

\subsection{Finiteness}\label{SPAN-fin}

In ordinary category theory, the notion of \emph{finite presentation} of objects is captured by \emph{compact objects}, that is, objects whose associated co-representable functor preserves filtered colimits. One then seeks to try to understand all objects of the category in terms of its compact objects. In particular, we like to work with \emph{compactly generated} categories: those for which every object is a colimit of compact objects.

A compactly generated category $\mathbf C$ is equivalent
\[ \mathbf C \cong \mathrm{Ind}(\mathbf C^c) \]
to its category of ind-compact objects. In particular, filtered colimits are exact.

\

The order-theoretic version of compactness is \emph{finiteness}. Its basic behaviour can be derived by applying the above results directly to the special case of objects in pre-orders.

\begin{defns}An element $X$ of a complete lattice $\alpha$ is \emph{finite} if, for any formula $X\leq\sup_{i\in I}X_i$ in $\alpha$, with the $X_i$ a filtered family, there exists an index $i$ such that $X\leq X_i$.

A lattice is \emph{algebraic} if every element is a supremum of finite elements.

A homomorphism $f:\alpha\rightarrow\beta$ \emph{preserves finiteness} if $f(X)$ is finite whenever $X$ is.\end{defns}

Be warned that it is not, in general, equivalent to replace the inequalities in the above definition with equalities. An element $X\in\alpha$ can be finite as an element of $\alpha_{\leq X}$ without being finite in $\alpha$.

\begin{lemma}A finite join of finite elements is finite.\end{lemma}

Let $\alpha$ be a complete lattice. I denote by $\alpha^c$ its subset of finite elements; by the lemma, $\alpha^c$ is a $\B$-module. It is functorial for $\B$-module homomorphisms that preserve finiteness.

\begin{prop}Let $\alpha\in\mathbf{Lat}$. The following are equivalent:
\begin{enumerate}
\item $\alpha$ is algebraic;
\item $\sup:\B(\alpha^c)\rightarrow\alpha$ is an isomorphism;
\item Every element of $\alpha$ is a supremum of elements $X$ that are finite in their slice set $\alpha_{\leq X}$, and finite meets distribute over filtered suprema in $\alpha$.\end{enumerate}\end{prop}

Let $\alpha$ be any $\B$-module. A $\B$-module ideal $\iota\hookrightarrow\alpha$ is finite as an element of $\B\alpha$ if and only if it is \emph{principal}, that is, equal to some slice set $\alpha_{\leq X}$. Therefore, $\alpha\rightarrow\B\alpha$ identifies $\alpha$ with the $\B$-module of finite elements of $\B\alpha$. This sets up an equivalence of categories
\[ \B:\mathrm{Mod}_\B\leftrightarrows \mathbf{Lat}_{al}:(-)^c \]
between $\mathrm{Mod}_\B$ and the category $\mathbf{Lat}_{al}$ of algebraic lattices.

\begin{egs}Let $S$ be a set. A subset of $S$ is finite as an element of $\B S$ if and only if it has finitely many elements; $(\B S)^c\cong\B^cS$ in the notation of example \ref{eg-first}. The power set $\B S\cong\B\B^cS$ is an algebraic lattice.

A submodule of a module $M$ over a ring $A$ is finite if and only if it is finitely generated; $\B(M;A)$ is an algebraic lattice.\end{egs}

\begin{defn}\label{SPAN-def-fin}The \emph{finite submodule} or \emph{free} $\B$-module on $M$ is the $\B$-module $\B^c(M;A)\cong (\B(M;A))^c$ of finitely generated $A$-submodules of $M$; since a sum of finite submodules is finite, this is closed in $\B(M;A)$ under joins. By algebraicity, $\B\B^c(M;A)\cong\B(M;A)$. We abbreviate $\B^c(A;A)$ to $\B^cA$.\end{defn}

\begin{eg}[Seminorms]\label{eg-seminorms}Let $A$ be an Abelian group. A \emph{(logarithmic) non-Archimedean seminorm} on $A$ with values in a $\B$-module $\alpha$ is a map of sets $\val:A\rightarrow\alpha$ satisfying the \emph{ultrametric inequality} $\val(f+g)\leq\val f\vee\val g$. One can take the supremum of any (non-Archimedean) seminorm on $A$ over any finitely generated subgroup $X\subseteq A$; indeed, if $X=(x_1,\ldots,x_n)$, then \[ \sup_{f\in X}\val f = \bigvee_{i=1}^n\val x_n. \] This supremum defines a \emph{$\B$-module homomorphism} $\B^c(A;\Z)\rightarrow\alpha$.

This correspondence exhibits the natural seminorm \[A\rightarrow\B^c(A;\Z),\quad a\mapsto (a) \] as \emph{universal} among seminorms of $A$ into any $\B$-module. In other words, $\B^c(A;\Z)$ corepresents the functor \[\frac{1}{2}\mathrm{Nm}(A,-):\mathrm{Mod}_\B\rightarrow\mathbf{Set}\] of seminorms on $A$.\end{eg}

\begin{eg}\label{eg-discs-field}Let $K$ be a non-Archimedean field with ring of integers $\sh O_K$ and value group $|K|\subseteq\R$. The given valuation induces a $\B$-module isomorphism $\B^c(K;\sh O_K)\widetilde\rightarrow |K|_\vee$. In fact, the same holds if $K$ is Archimedean, cf. e.g. \ref{eg-discs}.\end{eg}

\begin{eg}[Not enough finites]Let $K$ be a complete, discrete valuation field with uniformiser $t$. Let $\overline K$ be an algebraic closure with ring of integers $\sh O_{\overline K}$. Then $\B^c\sh O_{\overline K}\cong\Q_\vee^\circ=\Q_{\leq0}\sqcup\{-\infty\}$ (cf. def. \ref{.5RING-integers}) is the set of principal ideals generated by positive rational powers $t^q$ of the uniformiser. The `traditional' way to complete $\Q_\vee\circ$ would be to embed it in its set $\R_\vee^\circ$ of Dedekind cuts. The latter is a complete lattice with no finite elements.

Of course, it is more sensible in this case to consider $\Q_\vee^\circ$ as the set of finite elements in the well-behaved lattice $\B\Q_\vee^\circ\in\mathbf{Lat}_{al}$.\end{eg}

One can show that if the above statements are interpreted in the usual semantics within the topos of sheaves on a space $X$, one obtains the following set-theoretic characterisation of the finite submodule $\B$-module (sheaf). Let $\sh O_X$ be a sheaf of rings on $X$, $M$ an $\sh O_X$-module.

\begin{defn}\label{SPAN-def-fin'}A submodule $N\hookrightarrow M$ is \emph{locally finitely generated} if there exists a covering $\{f_i:U_i\rightarrow X\}_{i\in I}$ and epimorphisms $\sh O_{U_i}^{n_i}\twoheadrightarrow f_i^*N$ for some numbers $n_i\in\N$.

The \emph{finite submodule} or \emph{free} $\B$-module on $M$ is the sheaf \[\B^c(M;\sh O_X):U\mapsto\B^c(M(U);\sh O_X(U))\] 
of locally finitely generated $\sh O_X$-submodules of $M$.\end{defn}

One may simply take this as a set-theoretic definition of $\B^c$, verifying directly that $\B^c(M;\sh O_X)$ is a sheaf.

\begin{eg}[Local seminorms]\label{eg-sheafseminorms}Let $X$ be a space, $A$ a sheaf of Abelian groups on $X$. A \emph{seminorm} on $A$ with values in a sheaf $\alpha$ of $\B$-modules is a map $A\rightarrow\alpha$ of sheaves which induces over each $U\subseteq X$ a non-Archimedean seminorm on $A(U)$ (e.g. \ref{eg-seminorms}).

One can define a \emph{universal seminorm} $A\rightarrow\B^c(A;\mathrm{Mod}_{\sh O_X})$, which, for a given $U\subseteq X$, takes $f\in A(U)$ to the subsheaf of $A|_U$ that it locally generates. Any seminorm $\val:A\rightarrow\alpha$ factors uniquely through this universal one, with the factoring arrow taking any finite subsheaf \(F\subseteq A|_V\) to 
\[ \sup_{f^\bullet\in F(U_\bullet)}\left|\val f^\bullet\right| = \left|\bigvee_{i=1}^{n^\bullet}\val f_i^\bullet\right| \in |\alpha(U_\bullet)|\cong \alpha(V) \]
In this formula, $U_\bullet\twoheadrightarrow V$ is a covering on which $F$ is defined, and $(f_1^\bullet,\ldots,f_{n^\bullet}^\bullet)$ denotes a locally finite system of generators for $F(U_\bullet)$. (Note that the numbers $n^\bullet$ need not be bounded.)\end{eg}

\subsection{Noetherian}
\begin{defn}A $\B$-module is called \textit{Noetherian} if the slice sets satisfy the ascending chain condition, that is, if every bounded, totally ordered subset has a maximum.\end{defn}

\begin{prop}Let $\{X_i\}_{i\in I}\subseteq\alpha$ be a bounded family of elements of a $\B$-module $\alpha$. If $\alpha$ is Noetherian, then $\sup_{i\in I}X_i=\bigvee_{i\in J}X_i$ for some finite $J\subseteq I$.\end{prop}
\begin{proof}We proceed by contraposition. Suppose that for all finite $J\subseteq I$, there is some $i(J)\in I\setminus J$ such that $X_{i(J)}\not\leq\bigvee_{j\in J}X_j$, that is, such that $\bigvee_{j\in J}X_j < X_{i(J)}\vee\bigvee_{j\in J}X_j$. Then $I$ is infinite, and starting from any index $0\in I$ we can inductively construct an infinite, strictly increasing sequence
\[ X_0 < \left(X_1 \vee X_0\right) < \left(X_2 \vee X_1 \vee X_0\right) < \cdots \]where $n:=i(\{0,\ldots,n-1\})\in I$. Therefore $\alpha$ is not Noetherian.\end{proof}

\begin{cor}The following are equivalent for a bounded $\B$-module $\alpha$:
\begin{enumerate}\item $\alpha$ is Noetherian;
\item $\alpha$ is a complete lattice, and $\alpha^c=\alpha$;
\item $\alpha\widetilde\rightarrow\B\alpha$.\end{enumerate}
A $\B$-module is Noetherian if and only if its every bounded ideal is Noetherian.\end{cor}

\begin{eg}Let $A$ be a ring.
The following are equivalent:
\begin{enumerate}
\item $A$ is Noetherian;
\item $\B A$ is Noetherian;
\item $\B^cA$ is Noetherian;
\item $\B^cM$ is Noetherian for all $A$-modules $M$;
\end{enumerate}
In this case, $\B^cM=\B M$ if and only if $M$ is finitely generated.\end{eg}

\subsection{Adjunction}

As in category theory, the notion of adjoint map is central to the theory of $\B$-modules.

\begin{defn}Let $f:\alpha\rightarrow\beta$ be a monotone map of $\B$-modules. We say that a monotone map $g:\beta\rightarrow\alpha$ is \emph{right adjoint} to $f$, and write $f^\dagger:=g$, if $\iden_\alpha\leq gf$ and $fg\leq\iden_\beta$. In this situation, we also say ${}^\dagger g:=f$ is \emph{left adjoint} to $g$.\end{defn}

If $\alpha$ is a complete lattice, then by the adjoint functor theorem a right adjoint exists for $f$ if and only if it preserves arbitrary suprema. We have the formula
\[ X\mapsto f^\dagger X=\sup\alpha_{\leq f^{-1}(X)}. \]
Alternatively, $\B f$ always preserves suprema, and therefore we can always find an adjoint 
\[(\B f)^\dagger:\B\beta\rightarrow\B\alpha, \quad \iota\mapsto f^{-1}\iota \]
at the level of the lattice completions. The restriction of $(\B f)^\dagger$ to $\beta$ is an ind-adjoint in the sense of \cite[\S I.8.11]{SGA4}. If an ordinary right adjoint to $f$ exists, then the ind-adjoint is the composite of this with the inclusion $\alpha\rightarrow\B\alpha$; I therefore denote the ind-adjoint also by $f^\dagger$ in general, since no confusion can arise.

In particular, any $\B$-module homomorphism gives rise to a diagram
\[\xymatrix{ \B\alpha \\ \alpha\ar[r]^f\ar[u] & \beta\ar[ul]_{f^\dagger} }\]
in $\mathbf{POSet}$, and $\iden_\alpha\leq f^\dagger f$.

\subsubsection{Pullback and pushforward}

Suppose that $A$ is a ring, $f:M_1\rightarrow M_2$ an $A$-module homomorphism. If $N\hookrightarrow M_2$ is a submodule, then so is $N\times_{M_2}M_1\rightarrow M_1$. The fibre product is a monotone map
\[ f^\dagger=f^{-1}:\B(M_2;A)\rightarrow\B(M_1;A),\quad N\mapsto N\times_{M_2}M_1, \]
\emph{right adjoint} to the image functor $\B f$. It happens to be a lattice homomorphism.

Secondly, let $g:X\rightarrow Y$ be a morphism of ringed spaces, $A=\sh O_X$. Then the pushforward functor $f_*$ is right adjoint to $f^*$ on the category $\mathrm{Mod}_{\sh O}$ of modules. Correspondingly,
\[ g_*:\B(M;\sh O_X)\rightarrow\B(g_*M;\sh O_Y),\quad N\mapsto g_*N \]
is right adjoint to the lattice homomorphism $g^*=\B g$. Since pushforward is left exact, this lattice homomorphism \emph{does} agree with the functor on modules.

\begin{eg}\label{eg-closure}If $X\hookrightarrow Y$ is an open immersion of schemes, then the right adjoint to $f^*:\B\sh O_Y\rightarrow\B\sh O_X$ sends a closed subscheme of $X$ to its scheme-theoretic closure in $Y$.\end{eg}

\subsubsection{$\B$-module quotients}\label{SPAN-quotient}

In the theory of categorical localisation, certain types of adjunction can provide a substitute for a linear calculus of quotients of categories. One can apply a similar technique to semilinear algebra in order to provide explicit descriptions of $\B$-module coequalisers and quotients.

Let $s,t:\alpha\rightrightarrows\beta$ be a pair of $\B$-module homomorphisms.

\begin{defn}An ideal $\iota\hookrightarrow\beta$ is \emph{invariant} for the pair $s,t$ if, for all $X\in\alpha$, $sX\in\iota\Leftrightarrow tX\in\iota$.\end{defn}

Since $s$ and $t$ are $\B$-module homomorphisms, the subset $\B\beta/(s\sim t)\subseteq\B\beta$ of invariant ideals is closed under arbitrary suprema. The right adjoint to the inclusion is a self-homomorphism
\[ p:=\sup_{n\in\N}\left(  (ts^\dagger)^n \vee (st^\dagger)^n \right):\B\beta\rightarrow\B\beta \]
taking an ideal to the smallest invariant ideal containing it. It coequalises $s,t$. In fact, for any $\B$-module homomorphism $f:\beta\rightarrow\gamma$ coequalising $s,t$, $\B f$ is independent of the action of $s,t$, that is, factors uniquely through $p$. Setting
\[ p:\beta\rightarrow\beta/(s=t):=p(\beta)\subseteq\B\beta/(s\sim t)\subseteq\B\beta\]
where $p(\beta)$ is the set-theoretic image, we therefore obtain:

\begin{lemma}\label{SPAN-adj-coeq}$\beta/(s= t)$ is a coequaliser for $s,t$.\end{lemma}

In the special case $\alpha=\B$, where $s,t$ are determined by some elements $S,T\in\beta$, we write also as usual $\beta/(S=T)$ for the $\B$-module quotient (by the congruence relation generated by the relation $S=T$).

Specialisations of the above construction will come into play in later sections; see, for example, \S\ref{.5RING-contract}.

%% file: TOP.tex
\section{Topological lattices}\label{TOP}

A topological space with linear structure is \emph{linearly topologised} if its topology is generated by linear subspaces. In other words, a linear topology on a space is one that can be defined in terms of a certain decoration - a \emph{principal topology} - on its subobject lattice.

Let $\alpha$ be a complete lattice, $\alpha^u\subseteq\alpha$ a non-empty, upper subset, closed under finite meets. Such an $\alpha^u$ is called a \emph{fundamental system of opens}, or just \emph{fundamental system}, for short.

\begin{lemma}The collection of slice sets $\alpha_{\leq X}$ for $X$ in a fundamental system, together with $\emptyset$, are a topology on $\alpha$ for which $\vee$ is continuous.\end{lemma}
\begin{proof}It is clear that these sets define a topology; for continuity of $\vee:\alpha\times\alpha\rightarrow\alpha$, note simply that $\vee^{-1}(\alpha_{\leq X})=\alpha_{\leq X}\times\alpha_{\leq X}$.\end{proof}

The topology in the lemma is that \emph{defined by the fundamental system}.

Any intersection of fundamental systems is a fundamental system. Therefore, for any family $f_i:\alpha\rightarrow\beta_i$ of maps of complete lattices and fundamental systems $\beta_i^u$ on $\beta_i$, there is a \emph{smallest} fundamental system on $\alpha$ such that the $f_i$ are \emph{continuous} for the induced topology. Explicitly, it is given by the closure of the upper set
\[\bigcup_{i,X\in\beta_i^u}\alpha_{\geq f_i^\dagger(X)}\] under finite meets.

Dually, any union of fundamental systems generates a new fundamental system under finite meets. This coincides with the usual notion of generation of new topologies. Hence, for any family $g_i:\alpha_i\rightarrow\beta$ of homomorphisms and fundamental systems $\alpha_i^u$, there is a \emph{largest} fundamental system 
\[\beta^u:=\bigcap_i\beta_{\geq f_i(\alpha_i^u)}\]on $\beta$ making the $g_i$ continuous, and the topology it defines is simply the strong topology on the underlying set.

\begin{defns}A complete lattice equipped with a \emph{principal topology}, that is, a topology defined by a fundamental system, is called simply a \emph{topological lattice}. A \emph{topological $\B$-module} is a $\B$-module $\alpha$ equipped with an \emph{ideal topology}, that is, is the subspace topology with respect to some principal topology on $\B\alpha\supseteq\alpha$. A \emph{fundamental system for $\alpha$} is a fundamental system for $\B\alpha$, and we write $\alpha^u:=(\B\alpha)^u$.

A homomorphism of topological $\B$-modules (resp. lattices), is a continuous $\B$-module homomorphism (resp. lattice homorphism). The category of topological lattices is denoted $\mathbf{Lat}_t$, the category of topological $\B$-modules $\mathrm{Mod}_{\B,t}$.\end{defns}

A topological $\B$-module is a $\B$-module whose inhabited open sets are ideals, and in which every neighbourhood (of $-\infty$) is open. A topological lattice is the same, except that inhabited open sets are principal ideals. There is also an obvious notion of principal topology on a possibly incomplete $\B$-module.

\begin{eg}\label{semifields-top}The semifields $H_\vee$ (e.g. \ref{eg-first}) will always come equipped with the (principal) topology $H_\vee^u=H$. (In particular, $\B$ is discrete.) A net $\{X_i\}_{i\in I}$ converges to $-\infty$ if and only if it does so with respect to the order; in other words, if $\forall\lambda\in H$ $\exists i\in I$ such that $X_j\leq\lambda$ for all $j>i$.\end{eg}

The category of topological $\B$-modules (resp. lattices) comes with a forgetful functor \[?:\mathrm{Mod}_{\B,t}\rightarrow\mathrm{Mod}_\B\] which I do not suppress from the notation. Its left adjoint is given by equipping a lattice $\alpha$ with the \emph{discrete topology} $\alpha^u=\alpha$, its right adjoint by the \emph{trivial topology} $\alpha_\mathrm{triv}^u=\{\sup\alpha\}$. Both adjoints are fully faithful. We will treat the category of $\B$-modules as the (full) subcategory of discrete objects inside $\mathrm{Mod}_{\B,t}$.

In particular, limits (resp. colimits) in $\mathrm{Mod}_{\B,t}$ are computed by equipping the limits (resp. colimits) of the underlying discrete $\B$-modules with weak (resp. strong) topologies.

\begin{eg}\label{TOP-Hausdorff}A non-trivial topological $\B$-module is never Hausdorff in the sense of point-set topology, since every open set contains $-\infty$. Let us instead say that a $\B$-module $\alpha$ is \emph{Hausdorff} if $\inf\alpha^u=-\infty$. The category $\mathrm{Mod}_{\B,\dot t}$ of Hausdorff $\B$-modules is a reflective subcategory of $\mathrm{Mod}_{\B,t}$.\footnote{The notation $\dot{t}$ follows Bourbaki \cite{Banach}.}\end{eg}

\begin{defns}Let $f_i:\alpha_i\rightarrow\beta$ be a family of continuous $\B$-module homomorphisms. We say that $\beta$ carries the \emph{strong topology} with respect to the $f_i$, or that the family $f_i$ is \emph{strong}, if its topology is the strongest ideal topology such that the $f_i$ are continuous. 

In particular, if $f$ is just a single map, $f:\alpha\rightarrow\beta$ is strong if and only if it sends $\alpha^u$ into $\beta^u\subseteq\B\beta$. In this case, we may also say that $f$ is \emph{open} - although beware that it may fail to be open in the sense of general topology.

If $g_j:\alpha\rightarrow\beta_j$ are a family of continuous $\B$-module homomorphisms, then $\alpha$ carries the \emph{weak topology} with respect to the $g_j$, or that the family $g_i$ is \emph{weak}, if its topology is the weakest ideal topology such that the $g_i$ are continuous.\end{defns}

From the definition of ideal topology, it follows:

\begin{lemma}A family $f_i$ is weak (resp. strong) if and only if the induced family $\B f_i$ on the lattice completions is weak (resp. strong).\end{lemma}

In particular, weak and strong topologies, and hence limits and colimits, always exist.

\begin{prop}\label{TOP-span-complete}The category $\mathrm{Mod}_{\B,t}$ is complete, cocomplete, and semi-additive. Filtered colimits are exact.\end{prop}
\begin{proof}
We need to check that the product and coproduct topology on the direct join agree. The explicit formulas show that
\[ (\alpha\times\beta)^u=\{ (X,\sup\beta)\wedge(\sup\alpha,Y) | X\in\alpha^u,Y\in\beta^u \}
= \alpha^u\times\beta^u = (\alpha\sqcup\beta)^u \]
which proves that $\mathrm{Mod}_{\B,t}$ is semi-additive.

Now let $\alpha_i,\beta_i\rightarrow\gamma_i$ be a filtered system, with $\alpha,\beta\rightarrow\gamma$ its colimit. We will confuse $\alpha_i,\beta_i$ with their image in $\alpha\times_\gamma\beta$. To show that the natural map $\colim_i(\alpha_i\times_{\gamma_i}\beta_i)\rightarrow\alpha\times_\gamma\beta$ is a homeomorphism, it will suffice to show that it is open. Let 
\[ U\in \left(\colim_i(\alpha_i\times_{\gamma_i}\beta_i)\right)^u = \bigcap_i \alpha\times_\gamma\beta_{\geq\alpha^u_i\wedge\beta^u_i} \]
so there exist $X_i\in\alpha^u_i,Y_i\in\beta^u_i$ such that $\sup_i(X_i\wedge Y_i)\leq U$. Since, in  $\B(\alpha\times_\gamma\beta)$, filtered suprema distribute over meets (cf. \ref{SPAN-fin}),
\[ U\geq (\sup_iX_i)\wedge(\sup_iY_i)\in \left(\bigcap_i \alpha\times_\gamma\beta_{\geq\alpha^u_i}\right) \wedge \left(\bigcap_i \alpha\times_\gamma\beta_{\geq\beta^u_i}\right) = (\alpha\times_\gamma\beta)^u \]
and is therefore open.\end{proof}


\begin{eg}\label{eg-second}There are two obvious ways to topologise the function $\B$-module $C^0(X,\R_\vee)$ on a topological space $X$ (and similarly $\mathrm{P}C^r(X;-)$, $\mathrm{CP}C^r(X;-)$, etc., cf. e.g. \ref{eg-first}): a topology of \emph{pointwise convergence}, which is the weak topology with respect to evaluation maps
\[ \mathrm{ev}_x:C^0(X,\R_\vee)\rightarrow\R_\vee, \]
and one of \emph{uniform convergence}, which is the strong topology with respect to the inclusion $\R_\vee\hookrightarrow C^0(X,\R_\vee)$ of constants. In the important case $\mathrm{CPA}_*(X,\R_\vee)$ of convex, piecewise-affine functions, when $X$ is compact with affine structure, these two topologies agree.\end{eg}

\subsection{Topological modules}

Let $A$ be a non-Archimedean ring (def. \ref{ADIC-def-na}), $M$ a (complete locally convex) $A$-module. We equip $\B(M;A^+)$ with a principal topology
\[ (\B M)^u:=\{ U\hookrightarrow M|U\text{ open} \} ,\]
which, by definition of local convexity, is enough to recover the topology on $M$. We also consider $\B^c(M;A^+)$ as a topological $\B$-module with respect to the subspace topology. This topology is natural for continuous module homomorphisms, and hence lifts $\B^{(c)}$ to a functor
\[ \B^{(c)}(-;M):\mathrm{LC}_A\rightarrow\mathrm{Mod}_{\B,t}. \]
If $g:A\rightarrow B$ is a ring homomorphism, then the base extension must be replaced with a \emph{completed} base extension $-\widehat\tens_AB:\mathrm{LC}_A\rightarrow\mathrm{LC}_B$ (that is, ordinary base extension followed by topological completion with respect to the projective tensor product topology). Correspondingly, there is a lattice (resp. $\B$-module) homomorphism
\[ \B g:\B^{(c)}(M;A^+)\rightarrow\B^{(c)}(M\widehat\tens_{A^+}B^+;B^+),\quad N \mapsto\mathrm{Im}(N\tens_{A^+}B^+\rightarrow M\widehat\tens_AB); \]
in the case $M=A$ this agrees with the homomorphism $\B^{(c)}(A;A^+)\rightarrow\B^{(c)}(B;B^+)$ defined previously without taking into account the topology. The same functoriality extends to morphisms of nA-ringed spaces.

Beware that the elements of $\B^c(M;A^+)$ correspond to not necessarily \emph{closed} submodules of $M$, and hence might not be represented by a subobject in $\mathrm{LC}_A$.


\begin{eg}[Continuous seminorms]Let $A$ be a linearly topologised Abelian group. It follows immediately from the definition of the topology on $\B(A;\Z)$ that the universal seminorm $A\rightarrow\B^c(A;\Z)$ (e.g. \ref{eg-seminorms}) on $A$ is continuous. In fact, $\B^c(A;\Z)$ carries the \emph{strong} topology with respect to this map. In other words, if $A\rightarrow\alpha$ is any continuous seminorm into some $\alpha\in\mathrm{Mod}_{\B,t}$, then the factorisation $\B^c(A;\Z)\rightarrow\alpha$ is also continuous.

The topological free $\B$-module $\B^c(A;\Z)$ corepresents the functor of continuous seminorms
\[ \frac{1}{2}\mathrm{Nm}(A,-):\mathrm{Mod}_{\B,t}\rightarrow\mathbf{Set}. \]

As we know, we may also use a seminorm $\nu:A\rightarrow\alpha$ to induce a coarser topology on $A$, the weak topology with respect to $\B^c(A\;\Z)rightarrow\alpha$. This is called the \emph{induced} topology with respect to $\nu$. It is Hausdorff if and only if the image of $\B^c(A;\Z)$ in $\alpha$ is.\end{eg}

\begin{eg}Let $K$ be a complete, rank one valuation field. The isomorphism $\B^c(K;\sh O_K)\cong|K|_\vee$ of example \ref{eg-discs-field} is a homeomorphism.\end{eg}

One can also formulate a theory of \emph{pro-discrete completion} for $\B$-modules and lattices to correspond to the completion operation for non-Archimedean rings and their modules. Followed to the conclusion of this paper, this would yield a different category of skeleta.

However, in situations typically encountered in geometry, one only has to deal with rings $A$ that have an ideal of definition $I$, and are therefore in particular \emph{first countable}. In this situation, one can use the axiom of dependent choice to show that $\B(-;A^+)$ is automatically pro-discrete.
Moreover, Nakayama's lemma ensures that in these situations, even the free $\B$-module $\B^c(-;A^+)$ is pro-discrete. Indeed, if $M\twoheadrightarrow M/I$ is a quotient of discrete $A^+$-modules, any finite system of generators for $M/I$ lifts to generators for $M$. A pro-finite, $I$-adic $A^+$-module is therefore finitely generated. Conversely, any finite topological $A$-module is $I$-adically complete. It follows that $\B^c(M;A^+)$ is pro-discrete for any complete $A$-module $M$.

The main results \ref{SKEL-main}, \ref{thm-berk} of this paper remain true, under such first-countability hypotheses, if we work instead with pro-discrete $\B$-modules.


%% file: RING.tex
\section{Semirings}\label{.5RING}

Any symmetric monoidal category $\mathbf C$ gives rise to a theory of \emph{commutative algebras} $\mathrm{Alg(C)}$ and their \emph{modules}. In this section, I describe a closed, symmetric monoidal structure on the category of $\B$-modules; the corresponding theories are those of \emph{semirings} and their \emph{semimodules}. This semialgebra will provide the algebraic underpinning of the theory of skeleta.

Let $\mathbf C$ a category equipped with a monoidal structure $\tens$ with unit $1=1_\mathbf{C}$. One has a category $\mathrm{Alg}(\mathbf C)$ of \emph{monoids} or \emph{algebras} in $\mathbf C$, which are objects $A$ of $\mathbf C$ equipped with structural morphisms
\[ A\tens A \stackrel{\mu}{\rightarrow} A \stackrel{e}{\leftarrow} 1 \]
satisfying various usual constraints, and morphisms respecting these. If $A\in\mathrm{Alg}(\mathbf C)$, there is also a category $\mathrm{Mod}_A(\mathbf C)$ of \emph{$A$-modules in $\mathbf C$}, which comes equipped with a free-forgetful adjunction
\[ -\tens A:\mathbf C \rightleftarrows \mathrm{Mod}_A(\mathbf C):?_A. \]
The (right adjoint) forgetful functor $?_A$ is conservative.

If $\tens$ is \emph{symmetric}, then there is also a category $\mathrm{CAlg}(\mathbf C)$ of \emph{commutative} algebras. The module category $\mathrm{Mod}_A(\mathbf C)$ over $A\in\mathrm{CAlg}(\mathbf C)$ acquires its own symmetric monoidal structure, the \emph{relative tensor product}
\[ -\tens_A-=\mathrm{coeq}(-\tens-\tens A\rightrightarrows -\tens -) \]
(as long as $\mathbf C$ has coequalisers).

If $\tens$ is \emph{closed}, that is, $-\tens A$ has a right adjoint $\Hom_{\mathbf C}(-,A)$, then $?_A$ also commutes with colimits and therefore $-\tens A$ preserves compactness. Limits and colimits of modules are computed in the underlying category.

\subsection{The tensor sum}       \label{.5RING-sum}
The category of $\B$-modules carries a closed symmetric monoidal structure given by the \textit{tensor sum} operation $\oplus$ which, by definition, is characterised by a natural isomorphism
\[ \Hom_\B(\alpha\oplus\beta,\gamma) \cong \Hom_\B(\alpha,\Hom_\B(\beta,\gamma)) \]
where we use the internal $\Hom$ functor defined in section $\ref{SPAN}$. Alternatively, it is characterised as universal with respect to order-preserving maps $\alpha\times\beta\rightarrow\gamma$ that are right exact in each variable, that is, such that for each $X\in\beta$ the composite $\alpha\rightarrow\alpha\times\{X\}\rightarrow\gamma$ is a $\B$-module homomorphism, and similarly the transpose of this property. There is a canonical monotone map $\alpha\times\beta\rightarrow\alpha\oplus\beta$ such that for any such map, there is a unique extension
\[\xymatrix{ \alpha\oplus\beta \ar[dr] \\ \alpha\times\beta\ar[u]\ar[r] & \gamma}\]
to a commuting diagram of sets. It identifies $\alpha\times\{-\infty\}\cup\{-\infty\}\times\beta$ with $\{-\infty\}$.

Explicitly, $\alpha\oplus\beta$ is generated by symbols $X\oplus Y$ with $X\in\alpha,Y\in\beta$ subject to the relations
\begin{align}\nonumber X\oplus(Y_1\vee Y_2) &= (X\oplus Y_1)\vee(X\oplus Y_2); \\
\nonumber X\oplus(-\infty) &=-\infty\end{align}
which ensure that the map
\[ [f:\alpha\oplus\beta\rightarrow\gamma] \quad \mapsto \quad [X \mapsto [Y\mapsto f(X\oplus Y)]] \]
is well-defined and determines the promised adjunction. 

\begin{prop}\label{.5RING-closed}The tensor sum defines a closed, symmetric monoidal structure on $\mathrm{Mod}_\B$.\end{prop}
\begin{proof}The argument is routine; I reproduce here the unit and counit of the adjunction $-\oplus\alpha\dashv\Hom(\alpha,-)$. First, we have maps
\[ \beta\rightarrow\Hom(\alpha,\alpha\oplus\beta),\quad X\mapsto [Y\mapsto Y\oplus X] \]
which is a $\B$-module homomorphism by the relations above. Second, one checks that the map
\[ \mathrm{ev}:\Hom(\alpha,\beta)\times\alpha\rightarrow\beta \]
preserves joins in each variable, and so descends to a homomorphism $\Hom(\alpha,\beta)\oplus\alpha\rightarrow\beta$.\end{proof}

\

The definitions of semirings and semimodules are those of algebras and their modules in the category $(\mathrm{Mod}_\B,\oplus)$. I spell out some of these definitions here, in order to fix notation.

\begin{defns}\label{semiring}An \textit{idempotent semiring} $\alpha$, or, more briefly, \textit{semiring}, is a commutative monoid object $(\alpha,+,0)$ in the monoidal category $(\mathrm{Mod}_\B,\oplus)$. Explicitly, it is a $\B$-module equipped with an additional commutative monoidal operation $+$, called \emph{addition}, with identity $0$, that satisfies
\[ X+(Y_1\vee Y_2)=(X+Y_1)\vee(X+Y_2) \qquad  X+(-\infty)=-\infty \quad\forall X. \]
In notation, addition takes priority over joins: $X+Y\vee Z=(X+Y)\vee Z$. A semiring homomorphism is a monoid homomorphism. The category of semirings is denoted $\frac{1}{2}\mathbf{Ring}$.

We will also have occasion to use a category $\frac{1}{2}\mathbf{Alg}:=\mathrm{Alg}(\mathrm{Mod}_\B,\oplus)$ of possibly non-commutative \emph{semialgebras}.

A \emph{right semimodule} over a semiring $\alpha$, or (right) $\alpha$-module, is a $\B$-module $\mu$ equipped with an action $\mu\oplus\alpha\rightarrow\mu$ of $\alpha$, written $X\oplus Z\mapsto X+Z$. A homomorphism of semimodules is a module homomorphism. The category of $\alpha$-modules is denoted $\mathrm{Mod}_\alpha$.

The \emph{relative} tensor sum $\oplus_\alpha$ on $\mathrm{Mod}_\alpha$ is the quotient
\[  \mu\oplus_\alpha\nu \cong \mathrm{coeq}[\mu\oplus\nu\oplus\alpha \rightrightarrows \mu\oplus\nu]\]
in $\mathrm{Mod}_\B$. A commutative monoid in $\mathrm{Mod}_\alpha$ is an \emph{$\alpha$-algebra}; it consists of the same data as a semiring $\beta$ equipped with a semiring homomorphism $\alpha\rightarrow\beta$. The category of $\alpha$-algebras is denoted $\frac{1}{2}\mathbf{Ring}_\alpha$. The tensor sum of two $\alpha$-algebras over $\alpha$ has a semiring structure.
\end{defns}

\begin{egs}\label{semifields}The Boolean semifield $\B=\{-\infty,0\}$ is a unit for the tensor sum operation. It therefore carries a unique semiring structure, of which the notation is indicative, rendering it initial in the category of semirings. That is, $\B$ plays the r\^ole in the category of semirings that $\Z$ plays in the category of rings.

Any $\B$-module is in a canonical and unique way a module over $\B$, with $0$ acting as the identity and $-\infty$ as the constant map $-\infty$; whence the terminology of $\B$-modules. 

More generally, the semifield $H_\vee=H\sqcup\{-\infty\}$ associated to a totally ordered Abelian group $H$ (e.g. \ref{eg-first}) carries an addition induced by the group operation on $H$.

If $H$ can be embedded into the additive group $\R$, $H_\vee$ is a \emph{rank one} semifield; these semifields play the r\^ole in tropical geometry that ordinary fields play in algebraic geometry. Of particular interest are $\Z_\vee,\Q_\vee,\R_\vee$, the value semifields of DVFs, their algebraic closures, and of Novikov fields, respectively. Other semifields that arise from geometry, for example in Huber's work \cite{Hubook}, include those with $H$ of the form $\Z^k_\mathrm{lex}$, that is, $\Z^k$ with the lexicographic ordering and $-\infty$ adjoined. These semifields are non-Noetherian. They fit into a tower
\[ (\Z^k_\mathrm{lex})_\vee\rightarrow(\Z^{k-1}_\mathrm{lex})_\vee\rightarrow \cdots \rightarrow \Z_\vee \]
of semiring homomorphisms which successively kill each irreducible convex subgroup. See also \cite[\S\textbf{0}.6.1.(a)]{FujiKato}.
\end{egs}

From general principles about algebra in monoidal categories, it follows:

\begin{prop}The category of semirings is complete and cocomplete. Limits and filtered colimits are computed in $\mathrm{Mod}_\B$, and the latter are exact. Pushouts are computed by the relative tensor sum.\end{prop}

\subsubsection{Free semimodules}\label{fungus}

Let $A$ be a ring, $M_1,M_2\in\mathrm{Mod}_A$. There are natural homomorphisms
\[ m:\B^{(c)}(M_1;A)\oplus\B^{(c)}(M_2;A)\rightarrow \B^{(c)}(M_1\tens_AM_2;A),\quad [N_1]\oplus [N_2]\mapsto \mathrm{Im}(N_1\tens N_2\rightarrow M_1\tens_AM_2)\]
which in the case of the subobject $\B$-module $\B$ is a lattice homomorphism. These homomorphisms upgrade $\B^{(c)}$ to \emph{lax monoidal} functors
\[ \B^{(c)}:(\mathrm{Mod}_A,\tens_A)\rightarrow(\mathrm{Mod}_\B,\oplus). \]
It is therefore compatible with algebra on both sides, in the following ways:
\begin{enumerate}\item If $B$ is an $A$-algebra, then the multiplication $\mu$ on $B$ induces a semiring structure on $\B^{(c)}(B;A)$
\[ [N_1]+[N_2]=\mu(N_1\tens N_2)\subseteq B \]
and therefore the subobject (resp. free) $\B$-module functors are upgraded to functors
\[ \B^{(c)}:\mathrm{CAlg}_A\rightarrow\frac{1}{2}\mathbf{Ring}. \]
Beware that the sum $[N_1]+[N_2]$ of elements of this submodule semiring corresponds to a \emph{product} in $B$, and should not be confused with the set of sums of elements of $N_1$ and $N_2$, which corresponds instead to $\vee$.
\item If $M$ is a $B$-module, then the $B$-action on $M$ induces a $\B^{(c)}(B;A)$-module structure on $\B^{(c)}(M;A)$.
\[ \B^{(c)}:\mathrm{Mod}_B\rightarrow \mathrm{Mod}_{\B^{(c)}(B;A)} \]
With respect to the relative tensor sum $\oplus_{\B^{(c)}(B;A)}$, these functors are lax monoidal. In particular, $\B^{(c)}(B;A)$ is a $\B^{(c)}A$-algebra.
\end{enumerate}

Be warned that $\B^{(c)}$ is not \emph{strongly} monoidal: usually
\[\B^{(c)} (M_1;A) \oplus_{\B^{(c)}A} \B^{(c)}(M_2;A) \not\cong \B^{(c)}(M_1\tens_AM_2;A).\]
Similarly, it does not commute with most base changes - but see prop. \ref{.5RING-prop-contract}.

\begin{eg}[Seminormed vector spaces]\label{eg-discs'}Let $V$ be a vector space over a complete, valued field $K$, considered as an $\sh O_K$-module as in example \ref{eg-discs}. Let us discuss seminorms on $V$ with values in $|K|_\vee=|K|\sqcup\{-\infty\}$, the value semifield of $K$. Note that $|K|_\vee$ acts on the set of discs $\B^{(c)}(V;\sh O_K)$ (cf. \S\ref{fungus}).

If $K$ is non-Archimedean, then in the same vein as the previous example \ref{eg-seminorms}, the ultrametric inequality for a seminorm can be rephrased as \[\sup_{z\in\langle x,y\rangle}\nu z = \nu x\vee\nu y, \] where $\langle x,y\rangle$ denotes the $\sh O_K$-module span of $x$ and $y$. In other words, a seminorm is the same thing as a $\B$-module homomorphism $\B^c(V;\sh O_K)\rightarrow|K|_\vee$, compatible with the actions of $|K|_\vee$ on both sides.

On the other hand, if $K$ is Archimedean, and therefore either $\R$ or $\C$, then the subobjects are the convex, balanced discs. The join of two discs is their convex hull, and a disc is finite if it is the convex hull of finitely many `vertices'. Note that this implies that, for example, the unit disc of a $K$-Banach space $V$ is infinite as soon as $\dim V>1$.

The same triangle inequality as for the non-Archimedean case works if we replace the $\sh O_K$-module span $\langle x,y\rangle$ by the \emph{convex hull} $\mathrm{conv}(x,y)$. An Archimedean seminorm is therefore once again a $|K|_\vee$-module homomorphism $\B^c(V;\sh O_K)\rightarrow |K|_\vee$.

In either case, the valuation on $K$ induces a semiring isomorphism $\B^c(K;\sh O_K)\widetilde\rightarrow |K|_\vee$ (e.g. \ref{eg-discs-field}).

The space of seminorms is the hom-space $\Hom(\B^cV,|K|_\vee)$. The \emph{unit disc} associated to a seminorm $\nu$ is $\nu^\dagger0$. Conversely, if $D\in\B(V;\sh O_K)$ is a disc, then the $|K|_\vee$-action thereon determines a homomorphism \[|K|_\vee\rightarrow\B(V;\sh O_K),\quad r\mapsto rD,\] where we interpret $r$ as the disc of radius $r$ in $K$. Since $\bigcap_{r>r_0}rD=r_0D$, this homomorphism preserves infima. If $|K|=\Z$ or $\R$, then $|K|_\vee$ has all infima, and hence this homomorphism has a left adjoint $\nu$. Its behaviour on elements of $V$ is
\[ \nu x=\inf\left\{r\in|K|_\vee|x\in r \right\}. \]
It therefore maps $\B^cV$ into $|K|_\vee$ if and only if the disc $D$ \emph{absorbs} in the sense that $KD=V$; in this case, $\nu$ is a seminorm.
This correspondence recovers the well-known dictionary between seminorms and absorbing discs in the theory of vector spaces over valued fields \cite[\S2.1.2]{Banach}.\end{eg}

\subsubsection{Free semirings}

Let $\alpha$ be a semiring. The forgetful functor $\frac{1}{2} \mathbf{Ring}_\alpha\rightarrow\mathbf{Set}$ commutes with limits and therefore has a left adjoint $\alpha[-]$. It is the set of `tropical polynomials'
\[ \alpha[S]\cong\left\{\left. \bigvee_{n\in\N^S}\sum_{X\in S}n_XX+C_n \right|C_n\in\alpha, C_n=-\infty\text{ for }n\gg0 \right\} \]
with the evident join and plus operations.

\begin{defn}Let $\alpha$ be a semiring, $S$ a set; $\alpha[S]$ is called the \emph{free semiring} on $S$.\end{defn}

The free semiring construction commutes with colimits; in particular we have the base change
\[\alpha[S]\cong\alpha\oplus\B[S]\]
and composition 
\[ \alpha[S\sqcup T]=\alpha[S]\oplus_\alpha\alpha[T]\] for any $\alpha\in\frac{1}{2}\mathbf{Ring}$.

There is similarly a free functor $T\mapsto\B[T]$ for a \emph{$\B$-module} $T$; intuitively, it is the free semiring generated by the set $T$, subject to the order relations that exist in $T$.

\subsection{Action by contraction}\label{.5RING-contract}

The concept of \emph{contracting operator} is natural in analysis, and is intimately related to the operator norm. In the context of this paper, we use this concept to control the \emph{bounds} of tropical functions, and hence the radii of convergence of analytic functions.

\begin{defn}An endomorphism $f$ of a $\B$-module $\alpha$ is \emph{contracting} if, for each ideal $\iota\hookrightarrow\alpha$, $f(\iota)\subseteq\iota$. That is, $f$ is contracting if and only if $f(X)\leq X$ for all $X\in\alpha$.\end{defn}

\begin{eg}Let $A$ be an algebra and $M$ an $A$-module. An $A$-linear endomorphism of $M$ induces a contracting endomorphism of $\B(M;A)$ if and only if it preserves all $A$-submodules; that is, if it is an element of $A$.\end{eg}

Let now $\alpha$ be a semiring, $\mu$ a semimodule. Let $\iota\hookrightarrow\alpha$ be an ideal. 

\begin{defn}We say that $\iota$ \emph{contracts $\mu$} if it acts by contracting endomorphisms, or equivalently, every ideal of $\mu$ is $\iota$-invariant.
If $\iota=\alpha$, we say that $\mu$ is a \emph{contracting} $\alpha$-module. If also $\mu=\alpha$, we say simply that $\alpha$ is \emph{contracting} (as a semiring).\end{defn}

In particular, $\alpha$ is contracted by an ideal $\iota$ if and only if $\iota\leq0$, and $\alpha$ itself is contracting if and only if $0$ is a maximal element.

Let $\mathrm{Mod}_{\alpha\{\iota\}}$ denote the full subcategory of $\mathrm{Mod}_\alpha$ on whose objects $\iota$ contracts. This subcategory is closed under limits and the tensor sum, and so its inclusion has a lax monoidal left adjoint
\[\mathrm{Mod}_\alpha\rightarrow \mathrm{Mod}_{\alpha\{\iota\}},\quad \mu\mapsto \mu\{\iota\}, \] the \emph{contraction} functor. In particular, $\alpha\{\iota\}$ is an $\alpha$-algebra, and an $\alpha$-module $\mu$ is contracting if and only if its action factors through the structure homomorphism $\alpha\rightarrow\alpha\{\iota\}$. In other words, $\mathrm{Mod}_{\alpha\{\iota\}}$ really is the category of modules over the contraction $\alpha\{\iota\}$ of $\alpha$.

The inclusion into $\frac{1}{2}\mathbf{Ring}$ of the full subcategory $\frac{1}{2}\mathbf{Ring}_{\leq0}$ of contracting semirings commutes with limits and colimits, and hence has left and right adjoints
\begin{align} \nonumber \mathrm{Left}:\alpha &\mapsto {}^\circ\hspace{-1pt}\alpha:=\alpha\{\alpha\} \\
\nonumber \mathrm{Right}:\alpha &\mapsto \alpha^\circ:=\alpha_{\leq0} \end{align}
and unit and counit $\alpha^\circ\hookrightarrow\alpha\rightarrow{}^\circ\hspace{-1pt}\alpha$. We will also write \[{}^\circ(-):=(-)\{\alpha\}\cong -\oplus_\alpha{}^\circ\alpha \] for the corresponding functor $\mathrm{Mod}_\alpha\rightarrow \mathrm{Mod}_{{}^\circ\alpha}$; but beware that this notation hides the dependence on $\alpha$.

\begin{defn}\label{.5RING-integers}The subring $\alpha^\circ$ is the \emph{semiring of integers} of $\alpha$. The \emph{(universal) contracting quotient} is ${}^\circ\alpha$.\end{defn}

\begin{prop}The semiring of integers functor commutes with limits and filtered colimits.\end{prop}

The contraction functor $\mathrm{Mod}_\alpha\rightarrow\mathrm{Mod}_{\alpha\{\iota\}}$ defined above can be described explicitly in terms of the ind-adjoint to $\mu\rightarrow\mu\{\iota\}$ (compare \S\ref{SPAN-quotient}). To be precise, the semiring homomorphism $\alpha^\circ\rightarrow\alpha^\circ[\iota]$ induces a homomorphism
\[ (-)[\iota]:\B\mu=\B(\mu;\alpha^\circ)\rightarrow\B(\mu;\alpha^\circ[\iota]), \]
where we write $\B(\mu;\alpha)$ for the set of ideals of $\mu$ that are also $\alpha$-submodules. Its right adjoint identifies the term on the right with the set of $\iota$-invariant ideals of $\mu$. Any $\alpha$-module homomorphism $\mu\rightarrow\nu$ to a semimodule $\nu$ contracted by $\iota$ factors uniquely through the image of $\mu$ in $\B(\mu,\alpha^\circ[\iota])$. Thus, $\mu\{\iota\}\subseteq\B\mu$ is the subset of $\iota$-invariant ideals that are generated as such by a single element.

\begin{lemma}\label{.5RING-contract-construction}The image of $\mu$ in $\B(\mu,\alpha^\circ[\iota])$ is uniquely isomorphic to $\mu\{\iota\}$.\end{lemma}

\begin{eg}The ideal semiring $\B^{(c)}A$ of a ring $A$ is a contracting semiring. If $B$ is any $A$-algebra, then $\B(B;A)^\circ$ is the image of $\B A\rightarrow\B(B;A)$. Indeed, the additive identity of $\B B$ is precisely the image of the unit $A\rightarrow B$ of the algebra.\end{eg}

\begin{eg}[Semivaluations]\label{eg-vals}Let $A$ be a ring. A \emph{semivaluation} on $A$ is a map $\val:A\rightarrow\alpha$ into a semiring $\alpha$ which is a seminorm of the underlying Abelian group, and for which
\[ \val(fg)=\val f + \val g.\]
It is said to be \emph{contracting} or \emph{integral} if $\alpha$ is a contracting semiring.

Let $A$ now be a non-Archimedean ring. A (non-Archimedean) semivaluation of $A$ is a continuous valuation on $A$ whose restriction to $A^+$ is integral. Any such valuation factors uniquely through the adic semiring $\B^c(A;A^+)$ (def. \ref{def-adic}). That is, this semiring corepresents the functor 
\[ \Hom(\B^c(A;A^+),-)\cong\frac{1}{2}\mathrm{Val}(A,A^+,-):\frac{1}{2}\mathbf{Ring}_t\rightarrow\mathbf{Set} \]
of continuous semivaluations on $A$.\end{eg}

\begin{prop}\label{.5RING-prop-contract}Let $f:A\rightarrow B$ be a ring homomorphism. The extension of scalars transformation $\B(-;A)\rightarrow\B(-;B)$ induces an isomorphism
\[ \B(-;A)\oplus_{\B(B;A)}\B B \cong {}^\circ\B(-;A) \cong \B(-;B) \]
of functors $\mathrm{Mod}_B\rightarrow \mathrm{Mod}_{\B(B;A)}$, and similarly
\[ {}^\circ\B^c(-;A)\cong \B^c(-;B)\]
as functors $\mathrm{Mod}_B\rightarrow \mathrm{Mod}_{\B^c(B;A)}$.\end{prop}
\begin{proof}Let $M\in\mathrm{Mod}_B$. We will see that the morphism $\B(M;A)\rightarrow\B(M;B)$ satisifies the universal property of ${}^\circ\B(M;A)$.

Let $p:\B(M;A)\rightarrow\alpha$ be a $\B(B;A)$-module map. Precomposing with the forgetful map $\B f^\dagger:\B(-;B)\rightarrow\B(-;A)$ gives a map 
\[p\B f^\dagger:\B(M;B)\rightarrow\alpha.\]
Now $\B f^\dagger\B f$ is not the identity on $\B(-;A)$, but the endomorphism $\iden+A\geq \iden$. However, since $\alpha$ is contracting, the diagram
\[\xymatrix{ \B(M;B)\ar[dr]^{p\B f^\dagger} \\ \B(M;A) \ar[u]^{\B f}\ar[r]_-p & \alpha }\]
nonetheless commutes. In other words, $p\B f^\dagger$ exhibits $\B(-;B)$ as ${}^\circ\B(-;A)$.

As for the finite version, since $\B\B^c(-;A)\cong\B(-;A)$, applying $\B$ across the board embeds the picture into the one above.\end{proof}

\subsubsection{Freely contracting semirings}

Let $\alpha$ be a contracting semiring. The forgetful functor $\left(\frac{1}{2} \mathbf{Ring}_{\leq0}\right)_{\alpha}\rightarrow \frac{1}{2} \mathbf{Ring}_\alpha\rightarrow\mathbf{Set}$ commutes with limits and therefore has a left adjoint $\alpha\{-\}$. It is the composite of left adjoints $\alpha\mapsto\alpha[-]\mapsto{}^\circ\alpha[S]$.

\begin{defn}\label{.5RING-def-freec}Let $\alpha$ be a contracting semiring, $S$ a set (or $\alpha$-module); $\alpha\{S\}$ is called the \emph{freely contracting semiring} on $S$. If $\alpha$ is any semiring, we may also write $\alpha\{S\}:=\alpha\oplus_{\alpha^\circ}\alpha^\circ\{S\}$.\end{defn}

Note $\alpha\{S\}\cong{}^\circ(\alpha^\circ[S])\oplus_{\alpha^\circ}\alpha \cong \alpha[S]/(S\leq 0)=\alpha[S]/(S\vee0=0)$ (semiring quotient).

The freely contracting functor commutes with colimits; in particular we have the base change
\[\alpha\{S\}\cong\alpha\oplus\B\{S\}\]
and composition 
\[ \alpha\{S\sqcup T\}=\alpha\{S\}\oplus_\alpha\alpha\{T\}\] for any $\alpha\in\frac{1}{2}\mathbf{Ring}$.

\begin{eg}If $A$ is a complete DVR with maximal ideal $\lie m$, then its ideal semiring $\B^cA$ is freely contracting on the element $\lie m$.

This can be understood as an explicit construction of a freely contracting semiring on one element. More generally, $\B\{S\}$ for arbitrary $S$ can be described as the semiring of \emph{monomial} ideals in a polynomial ring $k[S]$ on the same set of variables.\end{eg}

\begin{eg}\label{non-eg}Let $\Delta=[-\infty,0]$ denote the infinite half-line, and consider the semiring $\mathrm{CPA}_\Z(\Delta,\R_\vee)$ of its convex, piecewise-affine functions with integer slopes (e.g. \ref{eg-first}). It is generated over $\R_\vee$ by a single, contracting element $X$. However, this generation is \emph{not} free: it satisfies additional relations, such as
\[ n(Y_1\vee Y_2)=nY_1\vee nY_2 \]
for all $n\in\N$ and $Y_i\in\mathrm{CPA}_\Z(\Delta,\R_\vee)$. We can see that these relations are not satisfied in $\R_\vee\{X\}$ by thinking of it as the set of monomial $\sh O_K\{x\}$-submodules of $K\{x\}$, where $K$ is any non-Archimedean field with value group $|K|=\R$.

The key difference between free semirings and function semirings is that the latter are \emph{cancellative}, while the former are not. In the present example, cancellativity can be enforced by imposing the above list of relations in $\R_\vee\{X\}$. The resulting \emph{universal cancellative quotient} $\R_\vee\{X\}\rightarrow\mathrm{CPA}_\Z(\Delta,\R_\vee)$ is infinitely presented. In particular, $\mathrm{CPA}_\Z(\Delta,\R_\vee)$ is not a finitely presented $\R_\vee$-algebra.\end{eg}

\subsection{Projective tensor sum}

The join of two continuous $\B$-module homomorphisms is continuous. The category of topological $\B$-modules is therefore enriched over $\mathrm{Mod}_\B$. We extend this to an \emph{internal} Hom functor by equipping the continuous homomorphism $\B$-module $\Hom_{\mathrm{Mod}_{\B,t}}(\alpha,\beta)$ with the weak topology with respect to the evaluation maps \[\mathrm{ev}_X:f\mapsto f(X)\] for $X\in\alpha$. In other words, it carries the topology of \emph{pointwise convergence}. A fundamental system for this topology is given \[ \Hom_{\mathrm{Mod}_{\B,t}}(\alpha,\beta)^u:=\{U_{X,Y}:=\{f|f(X)\subseteq Y\}| X\in\alpha, Y\in\beta^u\}, \] a formula that should evoke the compact-open topology of mapping spaces in general topology. 

\begin{eg}This is not the only reasonable way of topologising the continuous Hom $\B$-module, though it is of course the weakest. For instance, one could also define a topology of \emph{uniform convergence} as the weak topology with respect to the natural embedding
\[ \Hom(\alpha,\beta)\rightarrow \Hom(\B\alpha,\B\beta), \] where the right-hand term is equipped with the usual topology. These topologies are in general inequivalent; in fact, this embedding is not always continuous in the topology of pointwise convergence.

For example, a net $\{f_n\}_{n\in\N}$ in $\Hom(\Z_\vee,\Z_\vee)$ tends to $-\infty$ as $n\rightarrow\infty$ if and only if $f_n(x)\rightarrow-\infty$ for all $x\in\Z$. For the same net to die away in $\Hom(\B\Z_\vee,\B\Z_\vee)$, in addition $\{\sup_{x\in\Z}f_n(x)\}_{n\in\N}$ must tend to $-\infty$ (and in particular, be finite for cofinal $n\in\N$).\end{eg}

We can also extend the monoidal structure to $\mathrm{Mod}_{\B,t}$. The \emph{projective topological tensor sum} of topological $\B$-modules $\alpha,\beta$ is tensor sum $?\alpha\oplus ?\beta$ equipped with the strong topology with respect to the maps \[e_Y:\alpha\rightarrow\alpha\oplus\beta,\quad X\mapsto X\oplus Y\]
for $Y\in\beta$ and $e_X$ for $X\in\alpha$. If $\alpha,\beta$ are lattices, a fundamental system is generated by elements \[X\oplus \beta \vee \alpha\oplus Y,\quad X\in\alpha^u,Y\in\beta^u.\] It is more difficult to give a fundamental system for general $\alpha$ and $\beta$.

\begin{eg}\label{TOP-injective}The ideal $\B$-module functor $\B$ is not lax monoidal for the projective topology. For instance, the $\B$-module $\Z_\vee\oplus\Z_\vee$ is topologised so that a net $X_n\oplus Y_n$ dies away if and only if either $X_n$ dies and $Y_n$ is bounded, or vice versa. However, from the description of the fundamental system it follows that for the same net to die away in $\B\Z_\vee\oplus\B\Z_\vee$ it is enough that either $X_n$ or $Y_n$ does. The natural lattice homomorphism
\[ \B\Z_\vee\oplus\B\Z_\vee \rightarrow \B(\Z_\vee\oplus\Z_\vee) \]
is discontinuous.

It is, however, lax monoidal on bounded $\B$-modules, and in particular, lattices.\end{eg}

\begin{prop}The topological tensor sum and continuous internal Hom define a closed, symmetric monoidal structure on $\mathrm{Mod}_{\B,t}$ extending that of $\mathrm{Mod}_\B$.\end{prop}
\begin{proof}We only need to check that the unit and counit maps of proposition \ref{.5RING-closed} are continuous. For the unit $\alpha\rightarrow\Hom(\beta,\alpha\oplus\beta)$, which by the definition of the projective topology factors through the continuous Hom module, it is enough that the compositions $e_X:\alpha\rightarrow\alpha\oplus\beta$ with the evaluations at $X\in\beta$ are continuous. Continuity of the counit is similarly tautological.\end{proof}

\begin{prop}\label{RING-open}Let $\alpha\rightarrow\beta$ be strong. Then for any topological $\B$-module $\gamma$, $\alpha\oplus\gamma\rightarrow\beta\oplus\gamma$ is strong.\end{prop}
\begin{proof}This follows from the fact that if $fg$ and $g$ are strong (families of) maps, then $f$ is strong.\end{proof}

\begin{defns}\label{tsemiring}A \emph{topological semiring} is a commutative algebra in $(\mathrm{Mod}_{\B,t},\oplus)$. A topological semiring $\alpha$ is \emph{adic} if $\alpha^u$ is stable in $\B\alpha$ under addition, that is, if addition by an open element is an open map (def.\ \ref{TOP-open}). The category of adic semirings and continuous homomorphisms is denoted $\frac{1}{2}\mathbf{Ring}_t$. By proposition \ref{RING-open}, it is stable in the category of all topological semirings under tensor sum.

\label{def-adic}\end{defns}

In the sequel, all semirings will be assumed adically topologised, and so we will typically omit the adjectives `topological' and `adic'. A non-Archimedean ring $A$ (def. \ref{ADIC-def-na}), resp. homomorphism $f:A\rightarrow B$, is adic if and only if $\B^c(A;A^+)$ is adic, resp. $\B f$ is strong.

\begin{eg}\label{semifields'}The semifields $H_\vee$ associated to totally ordered Abelian groups (e.g. \ref{semifields}) are adic with respect to the topology of e.g. \ref{semifields-top}. All our examples of adic semirings will be adic over some $H_\vee$. The convergence condition for such semirings will therefore be that a net $X_n\in\alpha$ converges to $-\infty$ if and only if for each `constant' $r\in H_\vee$, cofinally many $X_n\leq r$ in $\alpha$.

For instance, the semirings $\R_\vee\rightarrow\mathrm{CPA}_*(X,\R_\vee)$ (e.g. \ref{eg-second}) are of this form.

Any continuous semiring homomorphism $H_\vee\rightarrow\B$ (where $\B$ is as always discrete) is an isomorphism. On the other hand, if $H\subseteq\R$ has rank one, then there is always a unique homomorphism $H_\vee^\circ\rightarrow\B$, the \emph{reduction} map. One can still define this map for general totally ordered semifields, but it is no longer unique.\end{eg}

\begin{eg}\label{eg-adic-Noetherian}An \emph{element of definition} of an adic semiring $\alpha$ is a principal open $I\in\alpha^u\cap\alpha$ such that $\alpha$ is $\Z_\vee^\circ$-adic with respect to the induced homomorphism
\[ \Z_\vee^\circ\rightarrow\alpha,\quad -1\mapsto I. \]
The join of two elements of definition is an element of definition. If $\alpha$ is Noetherian and has an element of definition, there is therefore also a \emph{largest} element of definition, and hence a canonical \emph{largest} $\Z_\vee^\circ$-algebra structure on $\alpha$. It thereby attains also a canonical reduction $\overline\alpha=\alpha^\circ\oplus_{\Z_\vee^\circ}\B$. Note that this $\Z_\vee^\circ$-algebra structure need not be unique or functorial, even for adic semiring homomorphisms.

If $X$ is any Noetherian formal scheme, $\B^c\sh O_X$ attains a canonical $\Z_\vee^\circ$-algebra structure, and the reduction $\overline{\B^c\sh O_X}\cong\B^c\sh O_{\overline X}$. Again, this is not to say that $\B^c$ defines a functor with values in $\mathrm{Alg}_{\Z_\vee^\circ}$.\end{eg}

\begin{eg}\label{eg-convergent-series}The free and freely contracting semirings $\alpha[X],\alpha\{X\}$ over an adic semiring $\alpha$ are topologised adically over $\alpha$.

Let $A$ be a non-Archimedean ring. The convergent power series ring $A\{x\}$ may be constructed as a certain completion of $A[x]$; in terms of semirings, it is the completion with respect to the topology induced by \[A[x]\rightarrow\B^cA[X]\rightarrow\B^cA\{X\},\] where the left-hand map is the unique valuation sending $x$ to $X$.\end{eg}

\begin{eg}[Discrete valuations]Let $X$ be an irreducible variety over a field $k$. A classic result of birational geometry states that `algebraic' discrete valuations $\val:K\rightarrow\Z_\vee$ on the function field $K$ of $X$, integral on $\sh O_X$, are in one-to-one correspondence with prime Cartier divisors on blow-ups of $X$. 

More specifically, let $\widetilde X\rightarrow X$ be a blow-up, $D\subset\widetilde X$ a prime Cartier divisor, and consider the formal completion $i:\widehat D\rightarrow\widetilde X$. Then the order of vanishing against $D$ is a continuous discrete valuation on the sheaf $i^*K$ of $\sh O_{\widehat D}$-modules. Conversely, given any discrete valuation $v$ on $K$, then provided that the associated residue field is of the correct dimension over $k$ (the algebraicity condition), one can construct the generic point of a $\widehat D$ giving rise to $v$ in this way as the formal spectrum of the completed ring of integers.

We can couch this correspondence in terms of semiring theory as follows. Let $U:=\widetilde X\setminus D$, and consider $(\widehat{\sh O}_U;\widehat{\sh O}_{\widetilde X})$ as a sheaf of non-Archimedean $\sh O$-algebras on the completion $\widehat D$. The reduction $D$ corresponds to an invertible element $I\in\B^c(\widehat{\sh O}_U;\widehat{\sh O}_{\widetilde X})$, and induces an adic homomorphism \[\nu^\dagger:\Z_\vee\hookrightarrow \B^c(\widehat{\sh O}_U;\widehat{\sh O}_{\widetilde X})\] of semirings over $\widehat D$; here $\Z_\vee$ denote the locally constant sheaf.

By Krull's intersection theorem, $\bigcap_{n\in\N}I^n=0$, that is, $\nu^\dagger$ preserves infima. It therefore has a left adjoint 
\[\nu:\B^c(\widehat{\sh O}_U;\widehat{\sh O}_{\widetilde X})\rightarrow \B\Z_\vee=\Z_\vee\sqcup\{\infty\}, \quad J \mapsto \inf\{n\in\N|J\leq nI\}. \]
In fact, this adjoint is finite (i.e. does not achieve the value $\infty$), since every section of $\sh O_U$ becomes a section of $\sh O_{\widetilde X}$ after multiplication by a power of $I$; moreover $\nu^{-1}(-\infty)=\{-\infty\}$. We have therefore defined a \emph{complete, discrete norm}
\[ \nu:\widehat{\sh O}_U\rightarrow\Z_\vee\]
over $\widehat D$.

For this norm to define a \emph{valuation}, the left adjoint $\nu$ must commute with addition. In general this property is much more delicate than the existence and finiteness of $\nu$. In our setting, a study of the local algebra shows directly that this happens exactly when $D$ is prime.
In this case, if $\Spec A=V\subseteq X$ is an affine subset meeting $D$, then localisation induces an extension $K\rightarrow\Z_\vee$ of the induced discrete valuation on $A$. This extension is \emph{not} left adjoint to the obvious map $\Z_\vee\rightarrow\B^c(K;\sh O_X)$, which is typically infinite (not to mention discontinuous).

For the converse statement, note only that discrete valuations on $K$, integral on some model $X=\Spec A$, are the same thing as homomorphisms \[ v:\B^c(K;A)\rightarrow\Z_\vee. \] This homomorphism has a (discontinuous) right ind-adjoint $v^\dagger$; the algebraicity condition is equivalent to this ind-adjoint being the extension of an ordinary adjoint, in which case $v^\dagger(-1)$ is a finitely generated ideal on $\Spec A$ which may be blown up to obtain $D$.\end{eg}

\subsubsection{Projective tensor product}\label{.5RING-proj}

Let $A$ be a non-Archimedean ring. The projective tensor product $M\tens_AN$ of locally convex $A$-modules $M$ and $N$ is strongly topologised with respect to the map
\[ \B^c(M;A^+)\oplus\B^c(N;A^+) \rightarrow \B(M\tens_AN;A^+).\]
We can describe this topology in terms of linear algebra alone: it is the strong topology with respect to the maps
\[ e_y:M\rightarrow M\tens_AN,\quad x\mapsto x\tens y \]
for $y\in N$, and similarly $e_x$ for $x\in M$. A sequence converges to zero in $M\tens_AN$ if and only if it is a sum of sequences of the form $x_n\tens y$ and $x\tens y_n$, where $x_n$ and $y_n$ converge to zero in $M$ and $N$, respectively.

With this definition, the monoidal functoriality of the \emph{free} $\B$-module $\B^c$ spelled out in \S\ref{fungus} lifts to the topological setting; for example, $\B^c(M;A^+)$ is a topological $\B^c(A;A^+)$-module. The corresponding statements for $\B$ are false unless $A=A^+$.

Similarly, we topologise $\Hom(M_1,M_2)$ weakly with respect to
\[ \Hom_A(M_1,M_2)\rightarrow\Hom_{\mathrm{Mod}_{\B,t}}(\B^cM_1;\B^cM_2),\quad f\mapsto \B^c f. \] A sequence of maps $\{f_n\}_{n\in\N}$ converges to zero if and only if for every finitely generated submodule $N\subseteq M_1$, every sequence $x_n\in f_n(N)$ converges to zero. This `finite-open' topology is the weak topology with respect to the evaluation maps
\[ \mathrm{ev}_x:\Hom_A(M_1,M_2)\rightarrow M_2,\quad f\mapsto f(x)\]
for $x\in M_1$.

%% file: LOC.tex
\section{Localisation}\label{LOC}

Let $\mathbf C$ be a category with filtered colimits, $M$ an object. In this setting, we can define the (free) \emph{localisation} of $M$ at an endomorphism $s\in\End_{\mathbf C}(M)$ as the sequential colimit
\[ M[s^{-1}]:=\colim\left[ M \stackrel{s}{\rightarrow} M \stackrel{s}{\rightarrow} \cdots \right] \]
It is universal among objects under $M$ for which $s$ extends to an automorphism. More generally, by composing colimits the localisation with respect to any set $S$ of commuting endomorphisms is defined.

If $\mathbf C$ is a category of modules over some algebra $A$, then in particular we can localise modules with respect to an element $s\in A$. If $A$ is commutative, and $M$ carries its own $A$-algebra structure, then the localisation $M[s^{-1}]$ is also an ($M$-)algebra.

The general theory specialises to the case of topological semirings; we write $\alpha[-S]$ for the localisation of $\alpha$ at an element $S$.

\begin{eg}Let $A$ be a domain, $\B^cA$ the finite ideal semiring. If $s\in A$, then $\B^cA[-(s)]\cong\B^c(A[s^{-1}];A)$. In order to obtain the ideal semiring of $A[s^{-1}]$, we need to enforce a contraction $(s)\leq0$.\end{eg}

\begin{eg}Suppose that $S\in\alpha$ is open. Then $\alpha[-S]$ is an adic $\alpha$-algebra (def. \ref{def-adic}).

This corresponds to the fact that if $A$ is an adic, linearly topologised ring, and $f\in A$ generates an open ideal, then $A[f^{-1}]$ is an adic $A$-algebra.\end{eg}

\begin{defn}\label{LOC-Tate}A topological semiring is \emph{Tate} if $\alpha^\circ$ is adic, and $\alpha$ is a free localisation of $\alpha^\circ$ at an additive family of open elements. The full subcategory of $\frac{1}{2}\mathbf{Ring}_t$ whose objects are Tate is denoted $\frac{1}{2}\mathbf{Ring}_T$.\end{defn}

In particular, any contracting semiring is Tate. A non-Archimedean ring $A$ is Tate (def. \ref{ADIC-def-na}) if and only if $\B^c (A;A^+)$ is.

\subsection{Bounded localisation}

In non-Archimedean geometry, localisations must be supplemented by certain completions, which control the radii of convergence of the inverted functions. For the geometry of skeleta to reflect analytic geometry, there must therefore be a corresponding concept for semirings.

\begin{defn}\label{LOC-def-radius}Let $\alpha$ be an adic semiring. An element $T\in\alpha^\circ$ that is invertible in $\alpha$ is called an \emph{admissible bound}, or simply a \emph{bound}.\end{defn}

Invertible elements $S=(+(-S))^{-1}(0)$ in adic semirings - in particular, bounds - are always open.

A localisation $\mu\rightarrow\mu[-S]$ is adic if and only if $S$ is open.

If $T\in\alpha^u$ is an open ideal, then $S$ is open as an endomorphism of the semiring quotient
\[ \mu/(T\leq S) = \mu/(T\vee S=S), \]
since $T\leq S$ forces $S$ to be open. The \emph{bounded} localisation $\mu\rightarrow\mu/(T\leq S)[-S]$ is therefore adic.

\begin{defn}Let $S\in\alpha^\circ$ and $T\in\alpha$ a bound. Let $\mu$ be an $\alpha$-module. A \emph{bounded localisation of $\mu$ at $S$ with bound $T$} is an $\alpha$-module homomorphism
\[\mu\rightarrow\mu\{T-S\}=\mu[-S]\{T-S\}, \]
universal among those under which $S$ becomes invertible with inverse bounded (above) by $-T$.

It is called a \emph{cellular} localisation if $T=0$. 

If $T\leq S$ in $\alpha^\circ$, then the bounded localisation is isomorphic to an ordinary, or \emph{free} localisation. In this case, we will often call it a \emph{subdivision}. Note that only free localisations at elements that are bounded below by an admissible bound are allowed.

More generally, the above definition makes sense if we replace $S$ with an arbitrary additive subset of $\alpha^\circ$ and $T$ with an additive set of bounds in bijection with $S$.\end{defn}

\begin{lemma}Any bounded localisation can be factored as a cellular localisation followed by a subdivision.\end{lemma}
\begin{proof}Factor $\alpha\rightarrow\alpha\{T-S\}$ as \[\alpha\rightarrow \alpha\{T-(S\vee T)\}\{-(S-(S\vee T))\}. \] In fact, this factorisation is natural in $\alpha,S$, and $T$.\end{proof}

\begin{eg}[Intervals]\label{eg-int}Consider the semiring $\mathrm{CPA}_\Z(\Delta,\R_\vee)$ (e.g. \ref{eg-first}), and for simplicity, specialise to the case that $\Delta=[a,b]$ is an interval with $a,b\in\Z$ (but see also \S\ref{EGS-poly}).

The admissible bounds of $\mathrm{CPA}_\Z([a,b],\R_\vee)$ are the affine functions $mX+c$, $m\in\Z,c\in\R$. Since every convex function on $[a,b]$ is bounded below by an affine function, any element of $\mathrm{CPA}_\Z(\Delta,\R_\vee)$ may be freely inverted by a bounded localisation.

Let us invert the function $X\vee r$ for some $r\in[a,b]$. The resulting semiring, which we denote $\mathrm{CPA}_\Z([a,r,b],\R_\vee)$, now consists of integer-sloped, piecewise-affine functions on $[a,b]$ which are convex \emph{except possibly at $r$}. I would like to think of this as a ring of functions on the polyhedral complex obtained by joining the intervals $[a,r]$ and $[r,b]$ at their endpoints, or alternatively, by \emph{subdividing} $[a,b]$ into two subintervals meeting at $r$. The affine structure does not extend over the join point. This is the motivation for the terminology `subdivision'.

More generally, the free bounded localisations of $\mathrm{CPA}_\Z([a,b],\R_\vee)$ are in one-to-one correspondence with finite sequences of rationals $r_1,\ldots,r_k\in(a,b)$:
\[ \mathrm{CPA}_\Z([a,b],\R_\vee)\rightarrow\mathrm{CPA}_\Z([a,r_1,\ldots,r_k,b],\R_\vee),\]
that is with subdivisions of $[a,b]$ in the sense of rational polyhedral complexes.

Now let's compose this with the \emph{cellular} localisation at $S=-(0\vee(X-r))$. This has the effect of imposing the relation $X\leq r$. In other words, the localisation is naturally $\mathrm{CPA}_\Z([a,r],\R_\vee)$, the semiring of functions on the lower \emph{cell} $[a,r]$. In particular, when $r=a$, the subdivision has no effect (since in that case $X\vee-a=X$ is already invertible), and the cellular localisation is just the evaluation $\mathrm{CPA}_\Z([a,b],\R_\vee)\rightarrow\R_\vee$ at $a$.

The composite of both localisations can be expressed more succinctly as $\mathrm{CPA}_\Z(\Delta,\R_\vee)\{X-r\}$, from which we can read that $r$ is the upper bound for the interval they cut out.

More generally, every cellular localisation of $\mathrm{CPA}_\Z([a,r_1,\ldots,r_k,b],\R_\vee)$ is determined by the union of cells on which a defining function vanishes.
\end{eg}

\begin{eg}\label{eg-int'}In the limiting case of the above example $\Delta=\R$, the only functions bounded by zero are the constants $\R_\vee^\circ=\mathrm{CPA}_*(\R,\R_\vee)^\circ$. The semiring $\mathrm{CPA}_*(\R,\R_\vee)$ therefore has no completed localisations; it is a poor semialgebraic model for the real line.\end{eg}

\begin{eg}\label{non-eg'}We have seen (e.g. \ref{non-eg}) that the semirings $\mathrm{CPA}$ are not finitely presented over $\R_\vee$. It may therefore be easier to work instead with finitely presented models of them; for example, $\R_\vee\{X\}$ instead of $\mathrm{CPA}_\Z(\R_\vee^\circ,\R_\vee)$.

However, the free localisation theory of these semirings is much more complicated than their cancellative counterparts - it depends on more than just the `kink set' of the function being inverted. For example, inverting $X\vee(-1)$ and $nX\vee(-n)$ define non-isomorphic localisations for $n>1$ (though the former factors through the latter). 

This could be regarded as a problem with the theory as I have set it up. I will not make any serious attempt to address it in this paper, as it does not directly affect the main results - but see e.g. \ref{circnorm}.\end{eg}

\begin{eg}Let $K$ be a non-Archimedean field with uniformiser $t$, $K\{x\}$ the Tate algebra in one variable. It is complete with respect to the valuation $K\{x\}\rightarrow |K|_\vee\{X\}$ of example \ref{eg-convergent-series}.

A completed localisation of the Tate algebra at $x$ has the form $K\{x,t^{-k}x^{-1}\}$ for some $k\in|K|$. This $k$ is a bound in the sense of definition \ref{LOC-def-radius}. The completed localisation is a completion of $K\{x\}[x^{-1}]$ with respect to the topology induced by its natural valuation into $|K|_\vee\{X,k-X\}$. 

The number $e^k$ (or $p^k$ when the residue characteristic is $p>0$) is conventionally called the \emph{inner radius} of the annulus $\Spa K\{x,t^kx^{-1}\}$. In other words, bounds in semiring theory arise intuitively as the `logarithms' of radii of convergence in analytic geometry.\end{eg}

\begin{eg}[Admissible blow-ups]Let $X$ be a quasi-compact adic space, $T\in\B^c(\sh O_X;\sh O_X^+)$ an admissible bound. Let $j:X\rightarrow X^+$ be a formal model on which $T$ is defined. Then $T\leq 0$ corresponds to a subscheme of $X^+$ whose pullback to $X$ is empty. In other words, the admissible bounds of $\B^c(\sh O_X;\sh O_X^+)$ that are defined on $X^+$ are exactly the centres for admissible blow-ups of $X^+$ (cf. \ref{ADIC}).\end{eg}

The following elementary properties of bounded localisation are a consequence of the universal properties.

\begin{lemma}Let $\alpha$ be a topological semiring, $\mu$ an $\alpha$-module.\label{LOC-properties}
\begin{enumerate}
\item $\alpha\{T-S\}$ is a semiring, and $\mu\{T-S\}\cong\mu\oplus_\alpha\alpha\{T-S\}$ as an $\alpha\{T-S\}$-module.
\item Localisation commutes with contraction. That is, $\mu\{\iota\}[-S]\cong\mu[-S]\{\iota\}$.
\item Let $S_1,S_2\in\alpha^\circ$, $T_1,T_2$ two bounds. Then $\mu\{T_1-S_1,T_2-S_2\}\cong\mu\{T_1-S_1\}\{T_2-S_2\}$.
\end{enumerate}\end{lemma}

It follows also from the discussion above:

\begin{lemma}Let $\alpha$ be adic. Then $\mu\rightarrow\mu\{T-S\}$ is adic.\label{LOC-strong}\end{lemma}

\subsection{Cellular localisation}

Let $\alpha$ be a contracting semiring. Then the only invertible element, and hence only admissible bound, is $0$. All localisations of a contracting semiring are therefore cellular: $\alpha\rightarrow\alpha/(S=0)$.

\begin{eg}\label{eg-topol}Let $X$ be a coherent topological space \cite[def. \textbf{0}.2.2.1]{FujiKato}, so that the $\B$-module $|\sh O_X|$ of quasi-compact open subsets of $X$ has finite meets that distribute over joins. Its lattice completion $\B|\sh O_X|$ is the lattice of all open subsets of $X$ (or the opposite to the lattice of all closed subsets of $X$).

If $X$ is quasi-compact, then it is an identity for the meet operation on $|\sh O_X|$; in other words, intersection of open subsets is a \emph{contracting semiring} operation on $|\sh O_X|$, and $X=0$. Note that this addition is idempotent. Let us describe the localisations of $|\sh O_X|$.

Let $S\in|\sh O_X|$. The inclusion $\iota:S\hookrightarrow X$ induces adjoint \emph{semiring} homomorphisms
\[  \iota_!:|\sh O_S|\leftrightarrows|\sh O_X|:\iota^* \] by composition with and pullback along $\iota$, respectively. They satisfy the identities $\iota^*\iota_!=\iden$ and $\iota_!\iota^*=(-)+S$. The right adjoint $\iota^*$ identifies $S$ with $0$. Moreover, any semiring homomorphism $f:|\sh O_X|\rightarrow\alpha$ with this effect admits a factorisation $f=f+S=f\iota_!\iota^*$ through $|\sh O_S|$, necessarily unique since $\iota^*$ is surjective. In other words, $|\sh O_S|$ is a cellular localisation of $|\sh O_X|$ at $S$.

Alternatively, and more in the spirit of what follows, one can argue this using the \emph{right ind-adjoint} \[f^\dagger:|\sh O_X|\{-S\}\rightarrow\B|\sh O_X|\] to the localisation $f$. This map is easier to describe in terms of closed subsets: if $Z\in|\sh O_X|\{-S\}$, then $f^\dagger Z$ is the smallest closed subset of $X$ whose image in $|\sh O_X|\{-S\}$ is $Z$. It identifies the localised semiring with the image of the composite $f^\dagger f$, which is the set of subsets $K\subseteq X$ equal to the closure of their intersections with $S$, $K=\overline{K\cap S}$. Closure puts $|\sh O_S|$ in one-to-one correspondence with this set.\end{eg}

The latter method of this example can be abstracted, in line with the methods of \S\ref{SPAN-quotient} and \S\ref{.5RING-contract}. Let $\alpha\in\frac{1}{2}\mathbf{Ring}_t$, $\mu$ an $\alpha$-module, $S\in\alpha^\circ$.

\begin{defns}An ideal $\iota\hookrightarrow\mu$ is \emph{$-S$-invariant} if $X+S\in\iota\Rightarrow X\in\iota$. The \emph{$-S$-span} of an ideal $\iota$ is \[ \bigcup_{n\in\N}(+S)^{-n}(\iota), \] that is, the smallest $-S$-invariant ideal containing $\iota$.\end{defns}

If $S$ is invertible, then being $-S$-invariant is the same as being invariant under the action of $-S$. In particular, the set of $-S$-invariant ideals of $\mu[-S]$ is the lattice $\B(\mu[-S];\alpha^\circ[-S])$ of $\alpha^\circ[-S]$-submodule ideals of $\mu[-S]$. Moreover,

\begin{lemma}The right adjoint to the localisation map
\[ \B\mu\stackrel{f}{\rightarrow}\B(\mu[-S];\alpha^\circ[-S]) \]
identifies the latter with the set of $-S$-invariant ideals of $\mu$.\end{lemma}
\begin{proof}Let $\iota\hookrightarrow\mu$ be $-S$-invariant. Every element of $\iota[-S]\hookrightarrow\mu[-S]$ is of the form $X-nS$ with $X\in\iota$. If $X-nS=f(Y)$ for some $Y\in\mu$, then $f(Y+nS)=X\in\iota$ and hence $Y\in\iota$. This proves that $f^\dagger f\iota=\iota$.
\end{proof} 

Since in the cellular localisation, $-S\leq 0$, every ideal is automatically $-S$-invariant. By lemma \ref{.5RING-contract-construction}, the contraction $(-)\{-S\}$ induces isomorphisms
\[ \B(\mu[-S];\alpha^\circ[-S])\widetilde\rightarrow\B(\mu\{-S\};\alpha^\circ\{-S\})\cong\B\mu\{-S\}. \]
This identifies the cellular localisation $\mu\{-S\}$ with the image of $\mu$ in $\B(\mu[-S];\alpha^\circ[-S])$.

We have obtained a characterisation of cellular localisations in terms of ideals:

\begin{lemma}\label{LOC-Zar-construction}A homomorphism $f:\mu\rightarrow\nu$ of $\alpha$-modules is a cellular localisation of $\mu$ at $S\in\alpha^\circ$ if and only if $f^\dagger$ identifies $\B\nu$ with the $-S$-invariants of $\B\mu$.\end{lemma}

Note only that the `if' part of the statement follows from the fidelity of $\B$.

\begin{eg}[Zariski-open formula]\label{LOC-Zar-open}
Let $X$ be a quasi-compact formal scheme, $i:U\hookrightarrow X$ a quasi-compact open subset. Let $I$ be a finite ideal sheaf cosupported inside $X\setminus U$. The restriction $\rho:\B^c\sh O_X\rightarrow i_*\B^c\sh O_U$ evidently factors through $\B^c\sh O_X\{-I\}$. 

Now suppose that $U=X\setminus Z(I)$ is exactly the complement of the zeroes of $I$. Then $\rho^\dagger$ identifies $i_*\B^c\sh O_U$ with the sheaf of subschemes $Z\hookrightarrow X$ equal to the scheme-theoretic closure of their intersection with $U$ (cf. e.g. \ref{eg-closure}). These subschemes are the $-I$-invariants of $\B^c\sh O_X$. Indeed, suppose that $f$ is some local function on $X$ such that $fI$ vanishes on $Z$. Then over $U$, $fI=(f)$, that is, $f$ vanishes on $Z\cap U$ and therefore on $Z$.

By lemma \ref{LOC-Zar-construction}, the natural semiring homomorphism
\[ \B^c\sh O_X\{-I\}\widetilde\rightarrow i_*\B^c\sh O_U\]
is an isomorphism.

\end{eg}

\subsection{Prime spectrum}\label{LOC-prime}

The purpose of this section is to discuss a special case of the general theory of the following section \ref{SKEL}, in which constructions can be made particularly explicit. It therefore perhaps would logically have its place after that section. For this reason, the discussion here is relatively informal.

In algebraic geometry, the underlying space of a formal scheme can be described in terms of open primes. A strong analogy holds in the setting of contracting semirings.

\begin{defn}Let $\alpha$ be a semiring. A \emph{semiring ideal} $\iota\hookrightarrow\alpha$ is an ideal and an $\alpha$-submodule. It is further a \emph{prime ideal} if $\alpha\setminus\iota$ is closed under addition.\end{defn}

Let $\alpha$ be a contracting semiring, $p:\alpha\rightarrow\B$ a (continuous) semiring homomorphism. The kernel $p^{-1}(-\infty)$ is an open prime ideal. Conversely, given an open prime ideal $\lie p\trianglelefteq\alpha$, one can define a semiring homomorphism
\[ \alpha\rightarrow\B,\quad X\mapsto\left\{
\begin{array}{cc}
	-\infty, & X\in\lie p \\
	0, & X\notin\lie p
\end{array}\right.
 \]
This sets up an order-reversing, bijective correspondence between the poset $\Spec_{\lie p}\alpha:=\Hom(\alpha,\B)$ and that of open prime ideals $\lie p\triangleleft\alpha$. In other words, every point in the \emph{prime spectrum} of a contracting semiring is represented by a $\B$-point.

Let us write $\D_\B^1$ for the Sierpinski space, whose underlying set is the Boolean semifield, but equipped with the topology is generated instead by the open set $\{0\}$ instead of the semiring topology. The Sierpinski space underlies the \emph{unit disc over $\B$}.

We now topologise the prime spectrum of a contracting semiring $\alpha$ weakly with respect to the evaluation maps $\Spec^{\lie p}\alpha\rightarrow\D_\B^1$, defined by identifying the underlying set of $\B$ with that of $\D_\B^1$. In other words, a sub-base for the topology is given by the open sets \[U_X:=\{f:\alpha\rightarrow\B|f(X)=0\}, \]
and $U_{X\vee Y}=U_X\cup U_Y$. This upgrades the prime spectrum to a contravariant functor
\[ \Spec^{\lie p}:\frac{1}{2}\mathbf{Ring}_{\leq0}\rightarrow\mathbf{Top}. \]
The continuous map of prime spectra induced by a homomorphism $f:\alpha\rightarrow\beta$ can be described in terms of prime ideals as
\[ \Spec^{\lie p}f:\beta\triangleright\lie p \mapsto f^{-1}(\lie p)\triangleleft \alpha, \]
just as in the case of formal schemes.

By construction the localisation morphism $\Spec^{\lie p}\alpha\{-S\}\rightarrow\Spec^{\lie p}\alpha$ induces an identification
\[ \Spec^{\lie p}\alpha\{-S\}\cong U_S\subseteq\Spec^{\lie p}\alpha \]
as topological spaces. This allows us to define a presheaf $|\sh O|$ of semirings on the site $\sh U_{/\Spec^{\lie p}\alpha}$ of affine subsets of the prime spectrum. By proposition \ref{LOC-Zar-cover}, below, it is actually a sheaf.

In summary, the \emph{prime spectrum} construction allows us to contravariantly associate to each \emph{contracting} topological semiring $\alpha$ a topological space $\Spec^\lie{p}\alpha$ equipped with a sheaf of semirings whose global sections are naturally $\alpha$.

\begin{egs}
First, it is of course easy to describe the spectrum of a freely contracting semiring: by the adjoint property, $\Hom(\B\{X_1,\ldots,X_k\},\B)=\D^k_\B:=\prod_{i=1}^k\D^1_\B$ is the \emph{polydisc of dimension $k$} over $\B$. The open subset defined by $\bigvee_{i=1}^kX_i=0$ is a kind of combinatorial simplex, in the sense that its poset of irreducible closed subsets is isomorphic to that of the faces of a $k$-simplex. See also \S\ref{EGS-poly}.

Similar statements hold for free contracting $H_\vee^\circ$-algebras, where $H_\vee$ is a rank one semifield. Indeed, the unique continuous homomorphism $H_\vee^\circ\rightarrow\B$ induces a homeomorphism
\[ \Spec^{\lie p}\alpha\oplus_{H_\vee^\circ}\B \rightarrow\Spec^{\lie p}\alpha \]
for any $\alpha$ over $H_\vee^\circ$. If $\alpha$ is of finite type, then in particular the set underlying the spectrum is finite.\end{egs}

\begin{eg}\label{LOC-prime-Noetherian}
The prime spectrum of a Noetherian semiring is a Noetherian topological space. As such, it has well-behaved notions of dimension and decomposition into irreducible components, cf. \cite[\S\textbf{0}.2]{EGA}. In particular, it is quasi-compact.\footnote{In fact, one can conclude from Zorn's lemma that \emph{any} prime spectrum is quasi-compact. I omit an argument, since anyway the definitions of this section will ultimately be superseded.}\end{eg}

\subsection{Blow-up formula}\label{blowup}

Let $X$ be a formal scheme, $I$ a finite ideal sheaf. The blow-up $p:\widetilde X\rightarrow X$ of $X$ along $I$ is constructed as $\Proj_X R_I$, where $R_I$ is the Rees algebra
\[ R_I:=\bigoplus_{n\in\N}I^nt^n\subseteq \sh O_X[t]. \] One associates in the usual fashion \cite[\S\textbf{II}.2.5]{EGA} a quasi-coherent sheaf on $\widetilde X$ to any quasi-coherent, graded $R_I$-module on $X$; in particular, if $M$ is quasi-coherent over $\sh O_X$, then $p^*M$ is associated to $M\tens_{\sh O_X}R_I$. If we write $\B^{(c)}(M;R_I)$ for the set of (finitely generated) \emph{homogeneous} $R_I$-submodules of $M$, the associated module functor is induces a natural transformation
\[ \B^{(c)}(-;R_I)\rightarrow\B^{(c)}(-;\sh O_{\widetilde X}) \]
of functors $\mathrm{Mod}_{R_I}\rightarrow\mathrm{Mod}_{\B,t}$ over $X$.

By following the algebra through, we can obtain an explicit formula relating the subobjects of quasi-coherent sheaves on $X$ to those of their pullbacks to $\widetilde X$.

The dependence of the associated sheaf to a graded module is only `up to' the irrelevant ideal $R_I^+=\bigoplus_{n>0}I^nt^n$. For example, let $M$ be quasi-coherent and homogeneous over $R_I$, and let $N_1,N_2\hookrightarrow M$ be finite, homogeneous submodules. Then $N_1=N_2$ as sections of $\B^c(M;\sh O_{\widetilde X})$ if and only if
\[ N_i+kR_I^+\leq N_j\text{ for all }i,j \text{ and }k\gg0 \]
in $\B^c(M;R_I)$.
It is equivalent that the high degree graded pieces $(N_i)_k,k\gg0$ agree.
In other words, $\B^c(M;R_I)\rightarrow\B^c(M;\sh O_{\widetilde X})$ descends to an isomorphism
\[ \B^c(M;R_I)/(R_I^+=0)\widetilde\rightarrow \B^c(M;\sh O_{\widetilde X}). \]

Now suppose that $M$ is quasi-coherent on $X$. The $\B$-module $\B^c(M\tens R_I;\sh O_X)$ of finite, homogeneous $\sh O_X$-submodules of $M\tens R_I$ is itself graded
\[ \B^c(M\tens R_I;\sh O_X)\cong\bigvee_{n\in\N}\B^c(M\tens I^n;\sh O_X)+nT, \]
where $T=(t)$ is a formal variable to keep track of the grading. It is a module over the graded semiring 
\[\B^c(R_I;\sh O_X)\cong\bigvee_{n\in\N}\B^c(I^n;\sh O_X)+nT \cong \bigvee_{n\in\N}(\B^c\sh O_X)_{\leq nI}+nT\]
in which the irrelevant ideal is written $R_I^+=\bigvee_{n\in\Z_{>0}}n(I+T)$.

By proposition \ref{.5RING-prop-contract},
\[ \B^c(M\tens R_I;\sh O_X)\{R_I^+\}\cong {}^\circ\B^c(M\tens R_I;\sh O_X)\widetilde\rightarrow\B^c(M\tens R_I;R_I) \]
in the category of $\B^c(R_I;\sh O_X)$-modules (cf. def. \ref{.5RING-integers} for notation).

Composing these identifications, we therefore have for any $M$ a factorisation
\[ \B^c(M;\sh O_X)\rightarrow \left(\bigvee_{n\in\N}\B^c(M\tens I^n)+nT\right)\left\{\pm\bigvee_{n\in\Z_{>0}}n(I+T)\right\}
\widetilde\rightarrow\B^c\left(p^*M;\sh O_{\widetilde X}\right) \]
of $\B^c\sh O_X$-module homomorphisms. The isomorphism on the right is the general blow-up formula.

In the context of adic spaces and their models, a more elegant form is available.

\begin{prop}[Blow-up formula]\label{LOC-blow-up}Let $X$ be an adic space, $j:X\rightarrow X^+$ a quasi-compact formal model. Let $I\in\B^cj_*\sh O_X^+$ be an ideal sheaf cosupported away from $X$, i.e.\ such that $j^*I=\sh O_X$. Let $\tilde j:X\rightarrow \widetilde X^+\rightarrow X^+$ be the blow-up of $X^+$ along $I$. Then the pullback homomorphism \[\B^c(j_*\sh O_X;\sh O_{X^+})\rightarrow\B^c(\tilde j_*\sh O_X;\sh O_{\widetilde X^+})\] is a free localisation at $I$.\end{prop}
\begin{proof}
First, the preimage of $I$ on $\widetilde X^+$ is an invertible sheaf, and therefore invertible in $\B^c(\tilde j_*\sh O_X;\sh O_{\widetilde X^+})$; hence this semiring homomorphism at least factors though the localisation \[\varphi:\B^c(j_*\sh O_X;\sh O_{X^+})[-I]\rightarrow\B^c(\tilde j_*\sh O_X;\sh O_{\widetilde X^+}).\] Moreover, $\varphi$ is injective; if two sections $J_i-n_iI,i=1,2$ become identical on $\widetilde X^+$, then by the general blow-up formula, $J_i+kI$ are already equal on $X$ for $k\gg0$.

Surjectivity, on the other hand, follows from this

\begin{lemma}\label{lemma}If $I$ is finite, then for any $N\in\B^cp^*M$, $N+kI$ is in the image of $p^*$ for $k\gg0$.\end{lemma}

Suppose that $N$ is generated in degrees less than $k$. Then \[N^\prime=\bigvee_{i=0}^kN_i+(k-i)I\] is finite, and satisfies the inequalities
\[  p^*N^\prime \leq N+kI \leq p^*N^\prime + kR^+.  \] It is therefore a lift for $N+kI$.
\end{proof}

In fact, the proof of this lemma shows more: it gives a recipe for exactly which modules on $X$ pull back to which modules on $\mathrm{Bl}_IX$. Following this recipe yields a generalisation.

First, observe that $\widetilde j_*\sh O_X$ is the $\sh O_{\widetilde X^+}$-algebra associated to the graded $R_I$ algebra \[K_I:=j_*\sh O_X[t]\simeq\bigoplus_{n\gg0}j_*\sh O_Xt^n\] on 
$X^+$. We therefore obtain surjective homomorphisms
\[ \B^c(j_*\sh O_X;\sh O_{X^+})[T] \twoheadrightarrow\B^c(K_I;R_I)\twoheadrightarrow \B^c(\widetilde j_*\sh O_X;\sh O_{\widetilde X^+}) \]
in which the left-hand arrow associates to a polynomial $\bigvee_{i=0}^kiT+J_i$ the $R_I$-submodule of $K_I$ that the $J_it^i$ generate.

\begin{defn}\label{LOC-def-sta}Let $\alpha\subseteq\B^c(j_*\sh O_X;\sh O_{X^+})$ be a subring containing $I$. The \emph{strict transform semiring} $\widetilde\alpha$ of $\alpha$ is subring of $\B^c(\widetilde j_*\sh O_X;\sh O_{\widetilde X^+})$ whose objects can be written as graded $R_I$-submodules of $K_I$ in the form
\[ \bigoplus_{n\in\N}J_nt^n\subseteq K_I \]
with $J_n\in\alpha$. It is the image of $\alpha[T]\rightarrow\B^c(\widetilde j_*\sh O_X;\sh O_{\widetilde X^+})$.\end{defn}

The strict transform semiring contains the inverse of $I$: it is defined by the formula
\[-(p^*I) \simeq \bigoplus_{n\gg0}I^{n-1}t^n. \]
The argument of lemma \ref{lemma} therefore establishes:

\begin{cor}\label{LOC-strict-transform}The strict transform semiring $\widetilde\alpha$ is a free localisation of $\alpha$ at $I$.\end{cor}

%% file: SKEL.tex
\section{Skeleta}\label{SKEL}
\subsection{Spectrum of a semiring}\label{SKEL-spec}

Let $\frac{1}{2}\mathbf{Ring}$ denote the category of Tate semirings (def. \ref{LOC-Tate}; the subscript $T$ is held to be implicit from hereon in), $\mathbf{Sk}^\mathrm{aff}$ its opposite. We say that a morphism $f:X\rightarrow Y$ in $\mathbf{Sk}^\mathrm{aff}$ is an \emph{open immersion} if it is dual to a bounded localisation
\[ f^\sharp:|\sh O_Y|\rightarrow |\sh O_Y|\{T_i-S_i\}_{i=1}^k \]
of the semiring $|\sh O_Y|$ dual to $Y$ at finitely many variables $S_i,T_i\in|\sh O_Y|$.

Paraphrasing lemma \ref{LOC-properties} above:

\begin{lemma}The class of open immersions is closed under composition and base change.\end{lemma}

Following the general principles outlined in the preliminaries \S\ref{TOPOS}, and in more detail in \cite{Toen}, we obtain the structure of a Grothendieck site on $\mathbf{Sk}^\mathrm{aff}$ generated by those finite canonical covers of the form
\[ \{U_i\rightarrow X\}_{i=1}^k \]
where $U_i\rightarrow X$ is an open immersion for each $i\in[k]$. The tautological presheaf $|\sh O|$ of Tate semirings on $\mathbf{Sk}^\mathrm{aff}$ is a sheaf, by the definition of canonical coverings.

\begin{defns}\label{SKEL-skel}The category $\mathbf{Sk}^\mathrm{aff}$, considered equipped with this topology, is called the \emph{skeletal site}. Its sheaf category is denoted $\mathbf{Sk}\topos$.

An \emph{affine skeleton}, resp. \emph{skeleton}, is a representable, resp. locally representable sheaf on the skeletal site (cf. def. \ref{TOPOS-def}, \cite[def. 2.15]{Toen}). If $\alpha$ is a semiring, the dual affine skeleton is called its \emph{spectrum} and denoted $\Spec\alpha$ The category of skeleta is denoted $\mathbf{Sk}$.\end{defns}

Of course, the Yoneda embedding identifies $\mathbf{Sk}^\mathrm{aff}$ with the category of affine skeleta.

More general arguments (cf. \S\ref{TOPOS}) equip each skeleton $X$ with a small topos $X\topos$, equivalent to the category of sheaves on a uniquely determined sober topological space with lattice of open sets $\sh U_{/X}$. I will abuse notation and denote this topological space also by $X$.

This fact allows us to alternatively interpret the Grothendieck site structure on $\mathbf{Sk}^\mathrm{aff}$ in terms  of a contravariant functor
\[ \mathbf{Sk}^\mathrm{aff}\rightarrow\mathbf{Top} \]
into the category of sober, quasi-compact, and quasi-separated topological spaces equipped with a sheaf of Tate semirings.
A skeleton is then a topological space $X$ equipped with a sheaf $|\sh O_X|$ of Tate semirings, locally isomorphic to an affine skeleton. The sections of $|\sh O_X|$ may be called \emph{convex functions} on $X$. 

\begin{prop}[\cite{Toen}]An affine skeleton is qcqs and sober, and affine open subsets form a basis for the topology. 

The category of skeleta has all fibre products.\end{prop}

\begin{eg}\label{eg-cpa}The spectrum $\Delta_{[a,b]}=\Spec\mathrm{CPA}_\Z([a,b],\R_\vee)$ of the semiring of convex, piecewise-affine functions on an interval $[a,b]\subseteq\R$ with rational endpoints (cf. e.g. \ref{eg-int}) is homeomorphic to a certain Grothendieck site structure on the poset of closed subintervals of $\Delta\subset\R$. 

Indeed, we already observed in example \ref{eg-int} that every subdivision of $\Delta_{[a,b]}$ is determined by a subdivision of $[a,b]$ as a rational polyhedron; meanwhile, by the cellular cover formula of the next section (proposition \ref{LOC-Zar-cover}), the cellular topology of $\Delta_{[a,r_1,\ldots,r_k,b]}$ is generated by the inclusions $[r_i,r_{i+1}]\rightarrow[r_0,r_k]$.

It remains to say when a collection of affine subsets $U_j=\{[a_{ji},b_{ji}]\}_{i=1}^{k_j}$ covers $X$.
\begin{conj}The $U_j$ cover $X$ if and only if $[a,b]=\bigcup_{i,j}[a_{ji},b_{ji}]$ and $(a,b)=\bigcup_{i,j}(a_{ji},b_{ji})$.\end{conj}
With the cellular cover fomula, the only part in question is the condition for a family of subdivisions to cover $\Delta$. The proposed criterion says that a collection of subdivisions covers if and only if there are no common `kink' points, that is, if \[\bigcap_j\bigcup_{i=1}^{k_j}\{a_{ji},b_{ji}\}=\emptyset.\] Indeed, in that case, the intersection over $j$ (in, say, the set of continuous functions) of the semirings of piecewise-affine functions convex on $U_j$ is exactly the set of such functions convex on $\Delta$. In other words,
\[ \mathrm{CPA}_\Z(\Delta,\R_\vee) \rightarrow\prod_i|\sh O_{U_i}| \rightrightarrows\prod_{i,j}|\sh O_{U_i\cap U_j}| \]
is an equaliser of semirings. I do not know how to show that this equaliser is universal.

As was pointed out in e.g. \ref{eg-int'}, the spectrum of $\mathrm{CPA}_*(\R,\R_\vee)$ consists of a single point. One can obtain a better model for the affine real line $\R$ as the increasing union
\[ \mathrm{sk}\R:=\bigcup_{a\rightarrow\infty}[-a,a] \]
in $\mathbf{Sk}$. Like the analytic torus over a non-Archimedean field, it is not quasi-compact.\end{eg}

\begin{eg}[Dichotomy]\label{circnorm}The skeleton constructed in the above example \ref{eg-cpa}, although relatively easy to describe, is not finitely presented over $\Spec\R_\vee$ (cf. \ref{non-eg}). In the vein of example \ref{non-eg'}, we can replace $\mathrm{CPA}_\Z([a,b],\R_\vee)$ with its finitely presented cousin
\[ \R_\vee\{[a,b]\}:=\R_\vee\{X-b,a-X\}. \]
They are related by an (infinitely presented) morphism $\Spec\mathrm{CPA}_\Z(\Delta,\R_\vee)\rightarrow\Spec\R_\vee\{[a,b]\}$.
We also saw in \ref{non-eg'} that this morphism is not a homeomorphism, and that in fact the topology of the target is difficult to describe.

It seems to be possible to modify the definition of skeleton, by introducing another condition into our definition of semiring, so as to make this morphism a homeomorphism. This condition is the semiring version of the algebraic notion of relative normality (integral closure of $A^+$ in $A$), which is used in non-Archimedean geometry to define a good Spec functor. However, I wish to defer a serious pursuit of this approach to a later paper, since this issue does not directly affect any of the results here.

The examples in \S\ref{EGS} all more closely resemble finitely presented skeleta like $\Spec\R_\vee\{\Delta\}$, but I will often only describe open subsets in a way that depends only on their pullbacks to a `geometric' counterpart $\Spec\mathrm{CPA}_\Z(\Delta,\R_\vee)$.
\end{eg}

\subsection{Integral skeleta and cells}

\begin{defn}A skeleton that admits a covering by spectra of contracting semirings is said to be \emph{integral}. The full subcategory of $\mathbf{Sk}$ whose objects are integral is denoted $\mathbf{Sk}^\mathrm{int}$.\end{defn}

The tropical site $\mathbf{Sk}^\mathrm{aff}$ carries a tautological sheaf $|\sh O|^\circ$ of contracting semirings, whose sections over $\Spec\alpha$ are $\alpha^\circ$. Taking the spectrum defines a functor \[ \Spec|\sh O|^\circ: \mathbf{Sk}^\mathrm{aff}\rightarrow\mathbf{Sk}^\mathrm{int}. \] Moreover, any covering of an object $\Spec|\sh O|^\circ(X)$ lifts, by base extension $|\sh O|^\circ\rightarrow|\sh O|$, to a covering of $X$, that is, $\Spec|\sh O|^\circ$ is cocontinuous. It therefore extends to the pushforward functor of a morphism
\[ (-)^\circ:\mathbf{Sk}\topos\rightarrow(\mathbf{Sk}^\mathrm{int})\topos \]
of the corresponding topoi.

This functor takes a skeleton $X$ to an integral skeleton $X^\circ$ if and only if there exists an affine open cover $X=\bigcup_iU_i$ such that $(U_i\cap U_j)^\circ\hookrightarrow U_i^\circ$ is an open immersion, in which case $U_\bullet^\circ$ provides an atlas for $X^\circ$. In algebraic terms, we need that $X$ admit an affine atlas each of whose structure maps is dual to a localisation $\alpha\rightarrow\alpha\{T-S\}$ that restricts to a localisation $\alpha^\circ\rightarrow\alpha\{T-S\}^\circ\cong\alpha^\circ\{T-S\}$ of the semiring of integers. This occurs if and only if the localisation is cellular, that is, if (up to isomorphism) $T$ is invertible in $\alpha^\circ$ and therefore zero.

\begin{defns}\label{SKEL-def-cel}An open immersion of skeleta is \emph{cellular} if it is locally dual to a cellular localisation of semirings.

A skeleton that admits a cover by affine, cellular-open subsets is said to be a \emph{cell complex}. In particular, any affine skeleton is a cell complex.

If \emph{every} open subset is cellular, it is a \emph{spine}. By the discussion above, any integral skeleton is a spine. 

The categories of spines, resp. cell complexes are denoted $\mathbf{Sk}^\mathrm{sp}\hookrightarrow\mathbf{Sk}^\mathrm{cel}$.

There is a functor
\[ (-)^\circ:\mathbf{Sk}^\mathrm{cel}\rightarrow\mathbf{Sk}^\mathrm{int}, \]
left adjoint to the inclusion, which associates to a cell complex $X$ its \emph{integral model} $X^\circ$. The unit of the adjunction is a morphism $j:X\rightarrow X^\circ$. The cellular open subsets of $X$ are those pulled back along $j$.\end{defns}

We have access to a reasonably concrete description of the `cellular topology'.

\begin{lemma}Let $\alpha$ be a semiring, $\{S_i\}_{i=1}^k\subseteq\alpha^\circ$ a finite list of contracting elements. Write $S=\bigvee_{i=1}^kS_i$. Then
\[ \alpha\{-S\}\rightarrow\prod_i\alpha\{-S_i\}\rightrightarrows\prod_{i,j}\alpha\{-S_i,-S_j\}\]
is a universal equaliser of semirings.\end{lemma}
\begin{proof}The lemma \ref{LOC-Zar-construction} yields an embedding of forks
\[ \xymatrix{ \alpha \ar[r]\ar[d] & \prod_i\alpha\{-S_i\} \ar@<2pt>[r]\ar@<-2pt>[r]\ar[d] & \prod_{i,j}\alpha\{-S_i,-S_j\}\ar[d] \\
 \B\alpha \ar[r] & \prod_i\B\alpha \ar@<2pt>[r]\ar@<-2pt>[r] & \prod_{i,j}\B\alpha } \]
in which the $i$th arrow in the lower row takes an ideal to its $-S_i$-span.

Let $\mathrm{eq}\subseteq\prod_i\B\alpha$ denote the equaliser of the second row, $f:\B\alpha\rightarrow\mathrm{eq}$ the natural $\B$-module homomorphism. An element of $\mathrm{eq}$ is a finite list $\{\iota_i\}_{i=1}^k$ of $-S_i$-invariant ideals, such that for each $i$ and $j$ the $-S_j$-span of $\iota_i$ is equal to the $-S_i$-span of $\iota_j$. The right adjoint $f^\dagger$ to $f$ sends such a list to their intersection in $\alpha$.

Since localisation commutes with base change, the fork in the statement is a universal equaliser as soon as it is an equaliser. By \ref{LOC-Zar-construction}, it is equivalent to show that $f^\dagger$ identifies $\mathrm{eq}$ with the set of $-S$-invariant ideals of $\alpha$.

On the one hand, the elements $f^\dagger(\mathrm{eq})$ are certainly $-S$-invariant. Indeed, suppose $X+nS\in\iota=\cap_i\iota_i$. Then $X+nS_i\leq X+nS\in\iota_i$, and so $X\in\iota_i$ for each $i$. Furthermore, since the $-S_i$-span of $\iota_j$ contains $\iota_i$, if $X\in\iota_i$, then $X+nS_i\in\iota_j$ for some $n$. Therefore, for $n\gg0$, $X+nS_i\in\iota$, and $\iota_i$ is the $-S_i$-span of $\iota$.

Conversely, suppose that $\iota$ is $-S$-invariant. Let $X\in f^\dagger f\iota\supseteq\iota$. Then for $n\gg0$, $X+nS_i\in\iota$ for all $i$, and therefore $X+nS\in\iota$, so $X\in\iota$. Therefore, $f$ and $f^\dagger$ are inverse.\end{proof}

In geometric terms:

\begin{prop}[Cellular cover formula]\label{LOC-Zar-cover}Let $\alpha$ be a semiring, $\{S_i\}_{i=1}^k\subseteq\alpha^\circ$ a finite list of contracting elements. Write $S=\bigvee_{i=1}^kS_i$. Then
\[ \Spec\alpha\{-S\}=\bigcup_{i=1}^k\Spec\alpha\{-S_i\} \]
as subsets of $\Spec\alpha$.\end{prop}

\begin{cor}\label{SKEL-qco=aff}Let $U$ be a quasi-compact cell complex. If $U$ can be embedded as an open subset of an affine skeleton, then $U$ is affine.\end{cor}

In fact, this result can be greatly improved.

\begin{thm}\label{SKEL-qc=aff}Let $X$ be a quasi-separated cell complex, $j:X\rightarrow X^\circ$ its integral model. Let us confuse $X^\circ$ with its site $\sh U^\mathrm{qc}_{/X^\circ}$ of quasi-compact open subsets. Then: 
\begin{enumerate}\item $\B\left(j_*|\sh O_X|\right)$ is flabby;
\item $X$ is affine if and only if it is quasi-compact and $j_*|\sh O_X|$ is flabby.\end{enumerate}\end{thm}

\begin{cor}Any quasi-compact, integral skeleton with Noetherian structure sheaf is affine.\end{cor}
\begin{proof}In this case $X=X^\circ$ and $|\sh O_X|=\B\sh O_X|$.\end{proof}

\begin{cor}Let $H_\vee$ be a rank one semifield. Any integral skeleton finitely presented over $\Spec H_\vee^\circ$ is affine.\end{cor}
\begin{proof}Such admits a model over some finitely generated subring of $H_\vee$.\end{proof}

If we make the assumption that all $H_\vee$-algebras $\alpha$ satisfy $\alpha\cong\alpha^\circ\oplus_{H_\vee^\circ}H_\vee$, then this last corollary applies also to any cell complex finitely presented over $\Spec H_\vee$ (which is, in this case, simply the base change of its integral model).

\begin{proof}[Proof of \ref{SKEL-qc=aff}]
Let $f:U_\bullet\twoheadrightarrow V$ be a finite, affine, cellular hypercover of some quasi-compact $V\subseteq X$. By corollary \ref{SKEL-qco=aff}, we may in fact assume that $U_\bullet$ is the nerve of an ordinary cover $\sqcup_if_i:\coprod_{i=1}^kU_i\twoheadrightarrow X$. Let $\alpha=\Gamma(V,|\sh O_X|)$. We have an equaliser
\[ \alpha\rightarrow \prod_{i=1}^k\alpha_i\rightrightarrows \prod_{i,j=1}^k\alpha_i\{-S_{ij}\} \]
commuting with isomorphisms $\alpha_i\{-S_{ij}\}\cong\alpha_j\{-S_{ji}\}$. There are unique elements $S_j\in\alpha$ whose images in $\alpha_i$ are $S_{ij}$.

Since $\alpha_i\rightarrow \alpha_i\{-S_{ij}\}$ is surjective, its right ind-adjoint is injective. The compositions
\[ \alpha_i \stackrel{\rho}{\rightarrow} \alpha_i\{-S_{ij}\} \widetilde\rightarrow \alpha_j\{-S_{ji}\} \stackrel{\rho^\dagger}{\rightarrow} \B\alpha_j  \]
therefore together yield a section $\B\alpha_i\rightarrow|\B\alpha_\bullet|$ of the projection. Since this holds for any quasi-compact $V$, $\B\left(j_*|\sh O_X|\right)$ is flabby.

Now set $V=X$. For the second part, it will be enough to show that each $f_i$ is a cellular localisation of $\alpha$ at $S_i$, since in this case the equaliser will be a covering, and hence induce an isomorphism $\Spec\alpha\widetilde\rightarrow X$. We will show this using the characterisation \ref{LOC-Zar-construction}.

Certainly, $f_i^\dagger:\B\alpha_i\rightarrow\B\alpha$ has image in the set of $-S_i$-invariant ideals. Since $|\sh O_X|$ is flabby, $f_i^\dagger$ is also injective. We need only show that it is surjective. The argument is based on two lemmata.

\begin{lemma}\label{SKEL-Zar-inv}Let $f:\alpha\rightarrow\beta$, $S\in\alpha^\circ$. If $f$ is surjective, $\B f$ preserves $-S$-invariance.\end{lemma}

\begin{lemma}\label{SKEL-push-pull}Let $f:\mu\rightarrow\mu\{-S\}$ be a cellular localisation of $\alpha$-modules, $g:\mu\rightarrow\nu$ a surjective homomorphism. The diagram
\[\xymatrix{ \mu \ar[d]^g & \mu\{-S\}\ar[d]^g\ar[l]_{f^\dagger} \\ \nu & \nu\{-S\}\ar[l]_{f^\dagger} }\] commutes.\end{lemma}
\begin{proof}The right adjoints embed $\mu\{-S\},\nu\{-S\}$ into $\B\mu,\B\nu$ as the set of $-S$-invariant ideals, which are preserved under $g$ by \ref{SKEL-Zar-inv}.
\end{proof}

Let $\iota\hookrightarrow\alpha$ be a $-S_i$-invariant ideal, $\iota_\bullet$ its image in $\alpha_\bullet$. By \ref{SKEL-Zar-inv}, $\iota_\bullet$ is $-S_{\bullet i}$-invariant, and hence 
\[ \iota_\bullet=f_i^\dagger f_i\iota_\bullet=f_i^\dagger f_\bullet\iota_i=f_\bullet f_i^\dagger\iota_i\]
where the last equality follows from \ref{SKEL-push-pull}. Therefore $\iota=|\iota_\bullet|=f_i^\dagger\iota_i$.
\end{proof}

\

Let $\alpha$ be a contracting semiring. The arguments of \S\ref{LOC-prime} show that we have a natural continuous functor
\[ \sh U_{/\Spec\alpha} \rightarrow \sh U_{/\Spec^\lie{p}\alpha} \]
and hence a morphism of semiringed spaces $\Spec^\lie{p}\alpha\rightarrow\Spec\alpha$.

\emph{If} $\Spec^\lie{p}\alpha$ is quasi-compact for all $\alpha\in\frac{1}{2}\mathbf{Ring}_{\leq0}$, then this is an isomorphism by the unicity of canonical topologies.
This is true for Noetherian semirings by \ref{LOC-prime-Noetherian}. The general case is implied by Zorn's lemma.

\begin{prop}\label{contentious}Let $\alpha$ be a (Noetherian) contracting semiring. Then $\Spec^\lie{p}\alpha\widetilde\rightarrow\Spec\alpha$, as semiringed spaces.\end{prop}

Topologising as before the set $X(\B)$ weakly with the respect to evaluations $X(\B)\rightarrow\D^1_\B$, we obtain:

\begin{cor}\label{cool}Let $X$ be an integral skeleton. Then $X(\B)\rightarrow X$ is a homeomorphism.\end{cor}

\subsection{Universal skeleton of a formal scheme}

\begin{eg}Let us return to the quasi-compact, coherent space $X$ of example \ref{eg-topol} and its semiring $|\sh O_X|$ of quasi-compact open subsets. We saw there that the inclusion
$S\hookrightarrow X$ of a quasi-compact subspace induces a localisation $|\sh O_X|\rightarrow|\sh O_S|$ at $S$. The cellular cover formula \ref{LOC-Zar-cover} implies that if $S=\bigcup_{i=1}^kS_i$ is a finite union of open subsets, then
\[ \Spec|\sh O_S|=\bigcup_{i=1}^k\Spec|\sh O_{S_i}|. \]
It follows that the functor
\[ \Spec|\sh O|:\sh U_{/X}^\mathrm{qc}\rightarrow\sh U_{/\mathrm{sk}X}^\mathrm{qc} \hookrightarrow \mathbf{Sk}_{/\mathrm{sk}X}^\mathrm{aff} \]
preserves coverings, and hence induces a homeomorphism of $X$ with $\mathrm{sk}X:=\Spec|\sh O_X|$. 

As we have seen (cor. \ref{cool}), every point of $\mathrm{sk}X$ is represented uniquely by a $\B$-point. In fact, the stalk of the structure sheaf $|\sh O_X|$ at any point $p\in\mathrm{sk}X$ is canonically isomorphic to $\B$, with $0$ (resp. $-\infty$) represented by an open subset containing (resp. not containing) $p$.

Under the homeomorphism $\mathrm{sk}X\widetilde\rightarrow X$, $|\sh O_X|$ can be identified with the semiring $C^0(X,\D_\B^1)$ of continuous maps from $X$ to the Sierpinski space $\D^1_\B$, that is, with the set of \emph{indicator functions} of open subsets.\end{eg}

It follows from the functoriality of the sheaves $\B^c\sh O_X$ associated on formal schemes $X$, as outlined in sections \ref{SPAN}, \ref{TOP}, \ref{.5RING}, that they assemble to a sheaf $|\sh O|$ of contracting semirings on the large formal site (in fact, with the fpqc topology). Its sections over a quasi-compact, quasi-separated formal scheme $X$ are the semiring of finite type ideal sheaves on $X$. This can be thought of as a \emph{geometric} version of the sheaf $|\sh O|$ of the above example, which is simply an avatar of the correspondence between (certain) frames and locales.

Let $X$ be any formal scheme, $\sh U_{/X}^\mathrm{qc}$ its corresponding small site, $|\sh O_X|$ the restriction of $|\sh O|$. The Zariski-open formula \ref{LOC-Zar-open} implies that if $V\hookrightarrow X$ is a quasi-compact open subset, then $|\sh O_X|$ puts the (necessarily cellular) bounded localisations of $|\sh O_X|(V)$ into one-to-one correspondence with quasi-compact Zariski-open subsets of $V$. The cellular cover formula implies that if $I_i\in|\sh O_X|(V)$ is a finite family of finite-type ideal sheaves on $X$, $U_i=V\setminus Z(I_i)$ the complementary quasi-compact opens, and
\[ U=V\setminus Z\left(\bigvee_{i=1}^kI_i\right)=\bigcup_{i=1}^kU_i, \]
then
\[ \Spec|\sh O_X|(U)=\bigcup_{i=1}^k\Spec|\sh O_X|(U_i)  \]
as subsets of $\Spec|\sh O_X|(V)$. In other words, $U\mapsto\Spec|\sh O_X|(U)$ defines a cover-preserving equivalence of categories between $\sh U_{/\Spec|\sh O_X|(U)}$ and $\sh U_{/X}$. This proves:

\begin{lemma}Let $X$ be a quasi-compact formal scheme. Then $\Spec|\sh O(X)|\rightarrow X$ is a homeomorphism.\end{lemma}

\begin{thm}\label{SKEL-min}Let $X$ be any formal scheme. Then $\mathrm{sk}X:=(X,|\sh O_X|)$ is a skeleton.\end{thm}

Of course, $\mathrm{sk}X$ is actually an \emph{integral} skeleton.

\subsection{Universal skeleton of an adic space}

Let $\mathbf{Ad}^\mathrm{qcqs}$ denote the quasi-compact, quasi-separated adic site. The sheaves $\B^c(\sh O_X;\sh O_X^+)$ on each adic space $X$ assemble to a sheaf $|\sh O|$ of Tate semirings on $\mathbf{Ad}^\mathrm{qcqs}$, extending the one with the same name introduced in the previous section.

Note that, unlike the case of formal schemes, this sheaf does not restrict to the presheaf
\[ |\sh O|^\mathrm{pre}=\B^c:\mathbf{nA}\rightarrow\frac{1}{2}\mathbf{Ring} \]
defined in terms of the section spaces of $\sh O$, since the $\sh O^+$-submodules it parametrises are, on the whole, not quasi-coherent. Naturally, $|\sh O|$ is the sheafification of $|\sh O|^\mathrm{pre}$.

In this section, we will derive the following generalisation of theorem \ref{SKEL-min}:

\begin{thm}\label{SKEL-main}Let $X$ be any adic space. Then $\mathrm{sk}X:=(X,|\sh O_X|)$ is a skeleton.\end{thm}

\begin{defn}The skeleton $\mathrm{sk}X$ is called the \emph{universal skeleton} of $X$.\end{defn}

The proof rests on a limit formula, following from the fundamental limit \ref{important} of \S\ref{ADIC}.

\begin{lemma}\label{SKEL-lim}Let $X$ be a qcqs adic space. Then
\[ \B^c(\sh O_X;\sh O_X^+)\cong\colim_{j\in\mathbf{Mdl}(X)}\B^c(j_*\sh O_X;j_*\sh O_X^+) \]
in $\frac{1}{2}\mathbf{Ring}_t$.\end{lemma}
\begin{proof}Indeed, the limit formula states explicitly that $j:X\widetilde\rightarrow\lim X^+$ as locales, and that $\sh O_X^+=\colim j^*j_*\sh O_X^+$ as sheaves on $X$. Any finitely generated ideal of $\sh O_X^+$ is therefore pulled back from some level $j_*\sh O_X^+$.

Since, by ${}^+$normality, the morphisms $j^*j_*\sh O_X^+\rightarrow\sh O_X$ are injective, then any two such ideals have the same image in $\sh O_X$ if and only if they agree on any cover, that is, on any model on which they are both defined.\end{proof}

\begin{proof}[Proof of \ref{SKEL-main}]Let $X\in\mathbf{Ad}^\mathrm{qcqs}$. We need to show that the localisations of $|\sh O|(X)$ are in one-to-one correspondence with the quasi-compact subsets of $X$.

Let $S\hookrightarrow X$ be a quasi-compact subset. There exists a formal model $j:X\rightarrow X^+$ and open subset $S^+\hookrightarrow X^+$ such that $S\cong X\times_{X^+}S^+$, and
\[ \B^c(j_*\sh O_X;j_*\sh O_X^+)\rightarrow \B^c(j_*\sh O_S;j_*\sh O_S^+) \]
is a cellular localisation at some (any) finite ideal $I$ cosupported on $X^+\setminus S^+$. This remains true when we modify $X^+$. Since $S$ is quasi-compact, every formal model $j_S:S\rightarrow S^+$ can be extended to a model $j$ of $X$, and so the colimit formula \ref{SKEL-lim} implies that
\[ |\sh O|(X)\cong\colim_{j\in\mathbf{Mdl}(X)_{/X^+}}\B^c(j_*\sh O_X;j_*\sh O_X^+)\rightarrow \colim_{j\in\mathbf{Mdl}(X)_{/X^+}}\B^c(j_{S*}\sh O_S;j_{S*}\sh O_S^+)\cong|\sh O|(S) \]
is a localisation at $I$.

Conversely, any $I\in|\sh O|(X)$ is representable by some finite ideal sheaf on a qcqs formal model $j:X\rightarrow X^+$ of $X$, whence $|\sh O|(X)\{-I\}\cong|\sh O|(U)$, where $U\cong X\times_{X^+}(X^+\setminus Z(I))$.

The cellular cover formula \ref{LOC-Zar-cover} shows that this correspondence preserves coverings, and hence induces a homeomorphism of $X$ with $\Spec|\sh O|(X)$.\end{proof}

The argument also shows:

\begin{cor}The universal skeleton of an adic space is a spine (def. \ref{SKEL-def-cel}).\end{cor}

\subsubsection{Real points of the universal skeleton}

Let $X$ be any skeleton. We can topologise the set $X(\R_\vee)$ of \emph{real points} of $X$ with respect to the evaluation maps $f:X(\R_\vee)\rightarrow\R_\vee$ associated to functions $f\in|\sh O_X|$, where on the right-hand side $\R_\vee$ is equipped with the usual \emph{order} topology (rather than the semiring topology). If $X$ is defined over some rank one semifield $H_\vee\subseteq\R_\vee$, then we may rigidify by considering $\R_\vee$-points over $H_\vee$; the subset $X_{H_\vee}(\R_\vee)\subseteq X(\R_\vee)$ similarly acquires a topology. 

The natural map $X(\R_\vee)\rightarrow X$ is often discontinuous with respect to this topology.

If $X$ is now an adic space, we can consider (following e.g. \ref{eg-vals}) the space $\mathrm{sk}X(\R_\vee)$ of real points of the universal skeleton as a space of \emph{real valuations} of $\sh O_X$. For this to be geometrically interesting, we usually want to consider this equipped with some $H_\vee$-structure. For instance, if $X$ is Noetherian, then $\mathrm{sk}X$ carries a canonical `maximal' morphism to $\Spec\Z_\vee$ (e.g. \ref{eg-adic-Noetherian}). The corresponding valuations send irreducible topological nilpotents to $-1$. Alternatively, if $X$ is defined over a rank one non-Archimedean field $K\rightarrow H_\vee$, then $\mathrm{sk}X$ is defined over $H_\vee$, and the real points are valuations extending the valuation of the ground field.

Where there is no possibility of confusion, I will abbreviate $\mathrm{sk}X_{H_\vee}(\R_\vee)$ to $X(\R_\vee)$.

To the reader familiar with analytic geometry in the sense of Berkovich \cite{Berketale} the following theorem will come as no surprise:

\begin{thm}\label{thm-berk}Let $X^\mathrm{Berk}$ be a Hausdorff Berkovich analytic space over a non-Archimedean field $K$, $X$ the corresponding quasi-separated adic space \emph{\cite[thm. 1.6.1]{Berketale}}. The composition 
\[ X(\R_\vee)\rightarrow X\rightarrow X^\mathrm{Berk} \]
is a homeomorphism.\end{thm}
\begin{proof}It is enough to show that the restriction of this map to every affinoid subdomain is a homeomorphism. This follows from the definitions and the identity
\[ \Hom(\B^c(\sh O_X;\sh O_X^+),\R_\vee)\widetilde\rightarrow\Hom(\B^c(A;A^+),\R_\vee) \]
for affine $X=\Spa A$, which holds because $\B^c(\sh O_X;\sh O_X^+)$ is a localisation of $\B^c(A;A^+)$.\end{proof}

\

\begin{prop}\label{SKEL-rat}Let $X$ be integral and adic over an adic space with a Noetherian formal model. Every function on $\mathrm{sk}X$ is determined by its rational values.\end{prop}

In classical terms this means the following: let $j:X\rightarrow X^+$ be a formal model of $X$, $Z_1,Z_2\hookrightarrow X^+$ two finitely presented subschemes; then if for all continuous rational valuations $\val:\sh O_X\rightarrow\Q_\vee$, $\val(I_1)=\val(I_2)$, then $Z_1=Z_2$ after some further blow-up of $X^+$.

\begin{proof}The statement is clear when $X^+$ is Noetherian and the $Z_i$ are both supported away from $X$; in this case, we may blow-up each $Z_i$ to obtain Cartier divisors, which by the Noetherian hypothesis factorise into prime divisors. Knowing that the $Z_i$ are Cartier divisors, they are therefore determined by the multiplicities of each prime divisor therein, that is, the values of local functions for the $Z_i$ under the corresponding discrete valuations.

Moreover, any formal subschemes $Z_i$ are, by definition, formal inductive limits of subschemes supported away from $X$, and so determined by a (possibly infinite) set of valuations.

Finally, for the general case we may assume that $Z_i$ are pulled back from some Noetherian formal scheme $X^+\rightarrow Y^+$ over which $X^+$ is integral. Since rational valuations admit unique extensions along integral ring maps, the discrete valuations on $\sh O_Y$ determining the $Z_i$ extend to rational valuations on $\sh O_X$.\end{proof}

\begin{cor}The universal skeleton of an adic space is cancellative.\end{cor}

\begin{cor}Let $X$ be as in \ref{SKEL-rat}. Then $X(\R_\vee)$ satisfies the conclusion of Urysohn's lemma.\end{cor}
\begin{proof}The proposition implies that $|\sh O_X|$ injects into the the set $C^0(X(\R_\vee),\R_\vee)$ of continuous, real-valued functions. By definition, two points of $X(\R_\vee)$ agree only every element of $|\sh O_X|$ takes the same value at both points. In other words, distinct points are separated by continuous functions.\end{proof}

This last result can be understood as a cute proof of the corresponding property for Hausdorff Berkovich spaces, that is, that they are \emph{completely} Hausdorff.

\subsection{Shells}

Let $X$ be an adic space. The universal skeleton of $X$ is a spine, so that any function with an admissible lower bound is invertible. If, for example, the skeleton is adic over $\Z_\vee$, then this is the same as every bounded function being invertible. Intuitively, this means that we have not defined a good notion of convexity for functions on $\mathrm{sk}X$.

We obtain a more restrictive notion of convexity by embedding $\mathrm{sk}X$ into a \emph{shell}, that is, a skeleton $B$ inside which $\mathrm{sk}X$ is a subdivision - in fact, the intersection of all subdivisions. At the level of the Berkovich spectrum $X(\R_\vee)$, this is akin to choosing a kind of `pro-affine structure' (a concept that I do not define here).

Suppose that $X$ is qcqs, and let $j:X\rightarrow X^+$ be a formal model of $X$. Write
\[  \mathrm{sk}(X;X^+):=\Spec\B^c(j_*\sh O_X;j_*\sh O_X^+), \] for the \emph{$X^+$-shell} of $X$. It is an affine skeleton whose integral model is the universal skeleton $\mathrm{sk}X^+$ of $X^+$.

More generally, if $X$ is any adic space admitting a formal model $X^+$, then a qcqs cover $U_\bullet^+\twoheadrightarrow X^+$ with generisation $U_\bullet=X\times_{X^+}U^+_\bullet$ gives rise to an \emph{$X^+$-shell}
\[ \mathrm{sk}(X;X^+):=|\mathrm{sk}(U_\bullet;U^+_\bullet)|, \]
which is a cell complex whose integral model, again, is $\mathrm{sk}X^+$. 

The blow-up formula \ref{LOC-blow-up} shows that the colimit \ref{SKEL-lim}, for each qcqs $U\hookrightarrow X$, is in fact over all possible \emph{free localisations} of $\B^c(j_*\sh O_U;j_*\sh O_U^+)$. In other words,

\begin{prop}Let $X$ be an adic space, $j:X\rightarrow X^+$ a formal model; $\mathrm{sk}X\subseteq\mathrm{sk}(X;X^+)$ is the intersection of all subdivisions of $\mathrm{sk}(X;X^+)$.\end{prop}

Any open subset of the $X^+$-shell is induced by a blow-up $X^+_i\rightarrow X^+$ followed by a Zariski-open immersion $U^+\hookrightarrow X^+_i$. I do not know of any easily-checked \emph{necessary} criterion to determine when a family of blow-ups $\{X^+_i\rightarrow X^+\}_{i=1}^k$ gives rise to a cover $\mathrm{sk}(X;X^+_\bullet)\twoheadrightarrow\mathrm{sk}(X;X^+)$ of the corresponding shells; it is certainly \emph{sufficient} that the blow-up centres have no common point.

Note that the formal model $X^+$ can be recovered from the data of $X$ and the shell $\mathrm{sk}X\hookrightarrow\mathrm{sk}(X;X^+)$. Indeed, one obtains from these data the continuous map $j:X\widetilde\rightarrow\mathrm{sk}X\rightarrow\mathrm{sk}X^+$ to the integral model of $\mathrm{sk}(X;X^+)$, $X^+$ is the formal scheme with the same underlying space as $\mathrm{sk}X^+$ and structure sheaf $j_*\sh O_X^+$.

\

Finally, the fact that any two models of $X$ are dominated by a third means that any two shells of $\mathrm{sk}X$ have a common open subshell; the shells can therefore be glued together to create a \emph{universal shell} $\mathbf{sk}X$. In abstract terms, the functor $\mathbf{sk}$ is obtained by left Kan extension along the inclusion $\mathbf{Ad}^\mathrm{aff}\hookrightarrow\mathbf{Ad}$ of
\[ \Spec|\sh O|^\mathrm{pre}:\mathbf{Ad}^\mathrm{aff}\rightarrow\mathbf{Sk}, \]
where $|\sh O|^\mathrm{pre}$, as before, denotes the presheaf $\Spa A\mapsto\B^c(A;A^+)$. Again, the universal shell $\mathbf{sk}X$ contains the spine $\mathrm{sk}X$ as the intersection of all subdivisions.

The universal shell is a universal way of defining a `pro-affine structure' on $X(\R_\vee)$ with respect to which the valuations of sections of $\sh O_X$ are convex. It also supports convex potentials for semipositive metrics on $X$.

%% file: EGS.tex
\section{Examples \& applications}\label{EGS}

I conclude this paper with some abstract constructions of skeleta which are already well-known via combinatorial means in their respective fields.

\subsection{Polytopes and fans}\label{EGS-poly}

Let $N$ be a lattice with dual $M$, and let $\Delta\subset N\tens\R$ be a rational polytope with supporting half-spaces 
$\{\langle-,f_i\rangle\leq\lambda_i\}_{i=1}^k,\lambda_i\in\Q$. We will allow $\Delta$ to be non-compact, as long as it has at least one vertex; this means that the submonoid $M_\Delta\subseteq M$ of functions bounded above on $\Delta$ separates its points. In this case, we can compactify $\Delta\subseteq\overline\Delta$ in, for example, the real projective space $\R\P(N\oplus\Z)$.

The semiring of `tropical functions' on $\Delta$ is presented
\[ \Z_\vee\{\Delta\}:=\Z_\vee[M_\Delta]/(f_i\leq\lambda_i)_{i=1}^k; \]
its elements have the form $\bigvee_{j=1}^dX_i+n_i$, with $X_i\in M_\Delta$ and $n_i\in\Z$.

\begin{defn}The semiring $\Z_\vee\{\Delta\}$ is the \emph{polytope semiring} associated to $\Delta$. Its spectrum $\mathrm{sk}\Delta$ is the corresponding \emph{polytope skeleton}, or just \emph{polytope} if the skeletal structure is implied by the context.\end{defn}

The construction $\mathrm{sk}$ is functorial for morphisms $\phi:M_1\rightarrow M_2,\phi(\Delta_1)\subseteq\Delta_2$ of polytopes. In particular, every sub-polytope $\Delta^\prime\subseteq\Delta$ (with $N$ fixed) induces an open immersion of skeleta $\mathrm{sk}\Delta^\prime\hookrightarrow\mathrm{sk}\Delta$. The morphism induced by a refinement $N\rightarrow \frac{1}{d}N$ can be thought of as a degree $d$ `base extension' $\mathrm{sk}\Delta\rightarrow\mathrm{sk}d\Delta$.

The polytope skeleton $\mathrm{sk}\Delta$ is a skeletal enhancement of $\overline\Delta$, in the sense that there is a canonical homeomorphism
\[ \mathrm{sk}\Delta(\R_\vee)\simeq\overline\Delta \]
of the real points, and surjective homomorphism $\Z_\vee\{\Delta\}\rightarrow\mathrm{CPA}_\Z(\overline\Delta,\Z_\vee)$ onto the semiring of integral, convex piecewise-affine functions on $\overline\Delta$ (that is, the semiring of integral, convex piecewise-affine, and bounded above functions on $\Delta$). 
We can produce a continuous map $\mathrm{sk}\Delta\rightarrow\overline\Delta$, right inverse to the natural inclusion, whose inverse image functor sends an open $U\subseteq\overline\Delta$ to the union
\[ \bigcup_{\sigma\subseteq U}\mathrm{sk}\sigma \hookrightarrow \mathrm{sk}\Delta, \]
ranging over all polytopes $\sigma$ contained in $U$. This map presents $\overline\Delta$ as a \emph{Hausdorff quotient} of $\mathrm{sk}\Delta$ (cf. thm. \ref{thm-berk}).

\

Polytope semirings admit an alternate presentation, related to the theory of toric degenerations. Let $N^\prime=N\oplus\Z$, with dual $M^\prime\cong M\oplus\Z$, and take the closed cone \[\sigma:=\overline{\bigcup_{\lambda>0}\lambda\Delta\times\{\lambda\}}\subset N^\prime\tens\R\] over the polytope placed at height one. The inclusion of the factor $\Z$ induces a homomorphism $i:\N\rightarrow\sigma^\vee\cap M^\prime$ of monoids; we topologise $\N$ linearly with ideal of definition $1$, and the cone monoid adically with respect to $i$. In other words, a fundamental system of open ideals of $\sigma^\vee\cap M^\prime$ is given by the subsets $\sigma^\vee\cap M^\prime + i(n)$ for $n\in\N$.


We find that \[\Z_\vee^\circ\{\Delta\}=\B\{\sigma^\vee\cap M^\prime\}=\B^c(\sigma^\vee\cap M^\prime)\] (see definitions \ref{.5RING-def-freec} and \ref{SPAN-def-fin} for notation) is the semiring of integers (def. \ref{.5RING-integers}) in $\Z_\vee\{\Delta\}$. Its elements are idempotent expressions $\bigvee_{i=1}^kX_i$ with $X_i\in\sigma^\vee\cap M^\prime$, subject to $X_i\leq 0$. Note that under this notation $-1\in\Z_\vee^\circ$ corresponds, perhaps somewhat confusingly, to $(0,1)\in M^\prime$.

An element $S=\bigvee_{i=1}^kX_i\in\Z_\vee^\circ\{\Delta\}$ corresponds to a finite union of subcones $\sigma_S=\bigcup_{i=1}^k(X=0)\subseteq\sigma$ and hence of faces $\Delta_S$ of $\Delta$, and the induced restriction
\[ \Z_\vee^\circ\{\Delta\}\rightarrow\Z_\vee^\circ\{\Delta_S\} \]
is a localisation at $S$. The topology of the integral model $\mathrm{sk}\Delta^\circ=\Spec\Z_\vee^\circ\{\Delta\}			$ of $\mathrm{sk}\Delta$ is therefore equal, as a partially ordered set, to the set of unions of faces of $\Delta$. In particular, $\mathrm{sk}\Delta^\circ$ is a finite topological space.

A refinement of the lattice $N\mapsto\frac{1}{k}N$ commutes with base extension $\Spec\frac{1}{k}\Z_\vee\rightarrow\Spec\Z_\vee$:
\[ \frac{1}{k}\Z_\vee\{\Delta\}:=\frac{1}{k}\Z_\vee\oplus_{\Z_\vee}\Z_\vee\{\Delta\}\widetilde\rightarrow \Z_\vee\{k\Delta\}. \]
The dual morphism \[\mathrm{sk}k\Delta^\circ\cong\frac{1}{k}\Z_\vee^\circ\times_{\Z_\vee^\circ}\mathrm{sk}\Delta^\circ\rightarrow \mathrm{sk}\Delta^\circ\] of integral skeleta is a homeomorphism.

\

Let $k[\![\sigma^\vee\cap M^\prime]\!]$ denote the completed monoid algebra of $\sigma^\vee\cap M^\prime$, and write $z^m$ for the monomial corresponding to an element $m\in\sigma^\vee\cap M^\prime$. I introduce the special notation $t:=z^{(0,1)}$ for the uniformiser; the completion is with respect to the $t$-adic topology. The monoid inclusion $\sigma^\vee\cap M^\prime\subset k[\![\sigma^\vee\cap M^\prime]\!]$ induces a continuous embedding
\[ \Z_\vee^\circ\{\Delta\} \hookrightarrow \B^c\left(k[\![\sigma^\vee\cap M^\prime]\!]\right) \]
into the ideal semiring of $k[\![\sigma^\vee\cap M^\prime]\!]$, matching $-1\in\Z_\vee^\circ$ with the ideal of definition $(t)$.

The formal spectrum $\D^+_{k[\![t]\!]}\Delta$ of $k[\![\sigma^\vee\cap M^\prime]\!]$ is an affine \emph{toric degeneration} in the sense of Mumford. That is, it is a flat degeneration
\[\xymatrix{ \D^+_{k[\![t]\!]}\Delta\ar[d] \\ \Spf k[\![t]\!] }\]
of varieties over the formal disc arising as the formal completion of a toric morphism of toric varieties. 

\begin{defn}Let $k$ be a ring. The \emph{polyhedral algebra} of functions \emph{convergent over $\Delta$} is the finitely presented $k(\!(t)\!)$-algebra 
\[k(\!(t)\!)\{\Delta\} := k[\![\sigma^\vee\cap M^\prime]\!][t^{-1}]. \] 
Its (analytic) spectrum $\D_{k(\!(t)\!)}\Delta$ is called the \emph{polyhedral domain over $k(\!(t)\!)$ associated to $\Delta$}.\end{defn}

For example, if $\Delta$ is the negative orthant in $\R^n$, then $\D_{k(\!(t)\!)}\Delta$ is just the ordinary unit polydisc $\D^n_{k(\!(t)\!)}$ over $k(\!(t)\!)$. Note that the polyhedral algebra has relative dimension equal to the rank of $N$, while the polyhedral semiring depends only on the lattice points of $\Delta$ and not on the ambient lattice.

A similar construction is possible in mixed characteristic.

In light of the main result \ref{SKEL-main}, there is a commuting diagram
\[\xymatrix{ \D_{k(\!(t)\!)}\Delta \ar[r]\ar[d]_{\mu_\Delta} &  \D^+_{k[\![t]\!]}\Delta \ar[d] \\
              \mathrm{sk}\Delta    \ar[r]       &  \mathrm{sk}\Delta^\circ }\]
in which the top and bottom horizontal arrows are morphisms of adic spaces and of skeleta, respectively, and the vertical arrows are continuous maps. If $\Delta$ spans $M$, then I would like to call the leftmost arrow $\mu_\Delta$ a \emph{standard non-Archimedean torus fibration over $\Delta$}.

\

This construction can be globalised to obtain torus fibrations on toric varieties and on certain possibly non-compact analytic subsets, in analogy with (and, more precisely, mirror to) the symplectic theory. Let $\Sigma$ be a fan in a lattice $N$, and let $X=X_\Sigma$ be the associated toric variety over a non-Archimedean field $K$, considered as an analytic space. Each cone $\Delta$ of $\Sigma$ corresponds to a Zariski-affine subset $U_\Delta\subseteq X$. Considering the cone as a polytope
\[\Delta=\bigcap_{i=1}^k\{\langle-,f_i\rangle\leq\lambda_i\}_{i=1}^k, \] embed it in a filtered family of \emph{expansions}
\[\Delta_{\mathbf r}=\bigcap_{i=1}^k\{\langle-,f_i\rangle\leq\lambda_i+r_i\}_{i=1}^k \]
for $\mathbf r=(r_i)\in\R_{\geq0}^k$; the analytic version of the subset $U_\Delta$ fits into the increasing union
\[ \xymatrix{ \D_{k(\!(t)\!)}\Delta_{\mathbf r} \ar[r]\ar[d] & U_\sigma \ar[d] \\
\mathrm{sk}\Delta_{\mathbf r} \ar[r] &  \bigcup_{\mathbf r\rightarrow\infty}\mathrm{sk}\Delta_{\mathbf r} } \]
of standard non-Archimedean torus fibrations. Note that it is not quasi-compact unless $N=0$. By glueing, we obtain a skeleton $\mathrm{sk}\Sigma$ and torus fibration
\[\xymatrix{ X_\Sigma\ar[d]_{\mu_\Sigma} \\ \mathrm{sk}\Sigma }\]
which is covered by the standard fibrations over affine polyhedral domains $\mathrm{sk}\Delta_{\mathbf r}\subseteq\mathrm{sk}\Sigma$, where $\Delta_{\mathbf r}$ ranges over all expansions of cones of $\Sigma$.

\subsection{Dual intersection skeleta}\label{EGS-Clemens}

Let $X^+$ be a reduced, Noetherian formal scheme, and let $X$ be the analytic space obtained by puncturing $X^+$ along its reduction $X^+_0$.

\begin{defns}\label{EGS-def-Clemens}The \emph{dual intersection} or \emph{Clemens semiring} of $X^+$ is the subring 
\[\mathrm{Cl}(X;X^+)\hookrightarrow\B^c(j_*\sh O_X;j_*\sh O_X^+) \]
generated by the additive units of $\B^c(j_*\sh O_X,j_*\sh O_X^+)$, that is, the invertible fractional ideals of $j_*\sh O_X^+$ in $j_*\sh O_X$. It is a sheaf of semirings on $X^+$, and it is functorial in both $X$ and $X^+$. The elements of the semiring of integers $\mathrm{Cl}(X;X^+)^\circ$ correspond to \emph{monomial} subschemes of $X^+$.

The \emph{dual intersection} or \emph{Clemens skeleton} of $X^+$ is \[ \mathrm{sk}\Delta(X,X^+):= \Spec \Gamma(X^+;\mathrm{Cl}(X;X^+)). \] 
It comes equipped with a \emph{collapse map} $X\rightarrow\mathrm{sk}X\rightarrow\mathrm{sk}\Delta(X;X^+)$.\end{defns}

It is possible, where confusion cannot occur, to drop $X$ and/or $\mathrm{sk}$ from the notation. I also write $\mathrm{Cl}^\circ(X;X^+)$ and $\mathrm{sk}\Delta^\circ(X;X^+)$ for the semiring of integers and integral model, respectively.

The flattening stratification decomposes $X^+=\coprod_{i\in I}E_i$ into locally closed, irreducible subsets such that the restriction of the normalisation $\nu:\widetilde X^+_0\rightarrow X^+_0$ to each $E_i$ is flat. In particular, the set underlying each monomial subscheme appears in this stratification.

\begin{lemma}\label{i can't be bothered to think of good labels any more}If $E\hookrightarrow X^+$ is a monomial subscheme with complement $V^+$, then $\mathrm{Cl}X^+\rightarrow\mathrm{Cl}V^+$ is a cellular localisation at $E$.\end{lemma}
\begin{proof}For this we may repeat the argument of \ref{LOC-Zar-open} with `ideal' replaced by `monomial ideal' throughout.\end{proof}

\begin{eg}
If $X^+$ is an affine toric degeneration associated to some polytope $\Delta$, with general fibre $X=\D_{\sh O_K}\Delta$, then the Clemens skeleton is $\mathrm{sk}\Delta(X,X^+)=\mathrm{sk}\Delta$ and the collapse map $\mu$ is a standard torus fibration $\mu_\Delta$. Its real points $\mathrm{sk}\Delta(\R_\vee)$ are the dual intersection complex of $X^+$ in the classical sense: its $n$-dimensional faces correspond to codimension $n$ toric strata of $X^+$.

The commuting diagram \[\xymatrix{ \D^+_{\sh O_K}\Delta^\prime \ar[r]\ar[d] & \D^+_{\sh O_K}\Delta  \ar[d] \\
 \Delta^\prime \ar[r] & \Delta }\] coming from the open inclusion of a face $\Delta^\prime$ of $\Delta$ can be seen as an instance of lemma \ref{i can't be bothered to think of good labels any more}.
\end{eg}

\begin{defn}\label{EGS-toric}A normal formal scheme $X^+$ over a field $k$ is said to have \emph{toroidal crossings} if it admits a cover $\{f_i:U_i^+\rightarrow X^+\}_{i=1}^k$ by open strata such that each $U_i^+$ is isomorphic to an open subset of some affine toric degeneration $\D^+_{k[\![t]\!]}\Delta_i$.\end{defn}

One can choose whether to consider \'etale or Zariski-open subsets for the covering, with the former being the usual choice. Zariski-local toroidal crossings is a very restrictive notion - for instance, it forces the irreducible components of $X^+_0$ to be rational. For simplicity, I will nonetheless work with this latter notion in this section, though the arguments may be generalised with some additional work.

Let $X^+$ be a formal scheme with Zariski-local toroidal crossings, and select model data as in the definition. Write $U_i=U_i^+\times_{X^+}X$. Assume, without loss of generality, that the given inclusions $U_i^+\hookrightarrow\D^+_{k[\![t]\!]}\Delta_i$ induce a bijection on the sets of strata, and hence isomorphisms $\Z_\vee\{\Delta_i\}\widetilde\rightarrow\Gamma(U_i^+;\mathrm{Cl}(U_i;U_i^+))$. They identify $\Delta_i$ with the dual intersection complex of $U^+_i$. By lemma \ref{i can't be bothered to think of good labels any more}, the inclusion of the open substratum $U_{ij}^+:=U_i^+\cap U_j^+$ identifies its dual intersection complex with a face $\Delta_{ij}$ common to $\Delta_i$ and $\Delta_j$.

It follows from this and the cellular cover formula \ref{LOC-Zar-cover} that $\{\Delta(U_i;U_i^+)\rightarrow\Delta(X;X^+)\}_{i=1}^k$ is a (cellular) open cover. 
\[\xymatrix{
X\ar[d]_\mu                        & \coprod_{i=1}^k U_i\ar@{->>}[l]\ar@{^{(}->}[r]\ar[d]        & \coprod_{i=1}^k\D_i\Delta_i\ar[d]^{\mu_{\Delta_i}} \\
\Delta (X;X^+)  & \coprod_{i=1}^k \Delta(U_i;U_i^+)\ar@{->>}[l]\ar@{=}[r] & \coprod_{i=1}^k\Delta_i }\]
Intuitively, $\Delta(X;X^+)$ is constructed by glueing together the dual intersection polytopes $\Delta_i$ of the affine pieces $U_i^+$ along their faces $\Delta_{ij}$ corresponding to the intersections $U_{ij}^+$.

\begin{prop}Let $X^+$ be a locally toric formal scheme. Then
\[ \coprod_{i,j=1}^k\Delta_{ij}\rightrightarrows \coprod_{i=1}^k\Delta_i\rightarrow\Delta(X;X^+) \]
is a cellular-open cover.  In particular, $\Delta(X;X^+)$ is a cell complex (def. \ref{SKEL-def-cel}).
\end{prop}

The collapse map $\mu$ is affine in the sense that $\Delta X^+$ admits an open cover that pulls back to an affine open cover of $X^+$. It follows that $X=\Spa\mu_*\sh O_X$ and $X^+=\Spf\mu^\circ_*\sh O_X^+$, where $\mu^\circ:X^+\rightarrow\Delta^\circ X^+$ is the integral model of $\mu$. It is locally isomorphic to standard torus fibrations $\mu_{\Delta_i}$.

\

Suppose that $X^+=\D^+\Delta$ is an affine toric degeneration, and let $p:\widetilde X^+\rightarrow X^+$ be a toric blow-up with monomial centre $Z\subseteq X^+_0$. The toric affine open cover of $\widetilde X^+$ induces a decomposition
\[ \coprod_{i,j=1}^k\Delta_{ij} \rightrightarrows \coprod_{i=1}^k \rightarrow \Delta(X;\widetilde X^+)\]
of the dual intersection skeleton of $\widetilde X^+$ into polyhedral cells. The map
\[ \Delta(X;\widetilde X^+) \rightarrow \Delta(X;X^+) \]
induced by the blow-up is a subdivision at the function $Z\in\mathrm{Cl}(X;X^+)$.

More geometrically, the Clemens semiring of a monomial blow-up is the \emph{strict transform semiring} of the Clemens semiring of $X^+$ (cf. \ref{LOC-strict-transform}). It follows that monomial blow-ups induce subdivisions of the dual intersection skeleta.

\subsection{Tropicalisation}\label{EGS-trop}

Let $X$ be a toric variety, so that following \S\ref{EGS-poly} it comes with a canonical `tropicalisation' $X\rightarrow\mathrm{sk}\Sigma$. Let $f:C\hookrightarrow X$ be a closed subspace of $X$. We would like to complete the composite $\mathrm{sk}C\hookrightarrow\mathrm{sk}X\rightarrow\mathrm{sk}\Sigma$ to a commuting square
\[\xymatrix{ C \ar[r]\ar[d]_{\mathrm{trop}} & X\ar[d] \\ \mathrm{Trop}(C/X/\Sigma) \ar@{^{(}->}[r] & \mathrm{sk}\Sigma }\]
and to call $C\rightarrow\mathrm{Trop}(C/X/\Sigma)$ the \emph{amoeba} or \emph{tropicalisation} of $C$ in $\mathrm{sk}\Sigma$, after (in chronological order) \cite{Kap} and \cite{Payne}.

Let us begin in the affine setting: let $X=\D_K\Delta$ be a polyhedral domain, and let $I_C$ be the ideal defining $C$ in $\sh O_X$. There is an associated toric degeneration $j:X\rightarrow X^+$ over $\sh O_K$, and we may close the subspace $C$ to obtain an integral model $C^+$ with ideal $I_C\cap j_*\sh O_X^+$. Let us set $\alpha_\Delta$ to be the image in $\B^c(j_*\sh O_C;j_*\sh O_C^+)$ of $\Z_\vee\{\Delta\}$, so that
\[\xymatrix{ \B^c(j_*\sh O_C;j_*\sh O_C^+) & \B^c(j_*\sh O_X;j_*\sh O_X^+)\ar[l] \\
 \alpha_\Delta \ar[u] & \Z_\vee\{\Delta\}\ar[l]\ar[u] }\]
commutes (here we confuse the sheaves $\B^c(\sh O;\sh O^+)$ with their global sections). The elements of $\alpha_\Delta^\circ$ are subschemes of $C^+$ \emph{monomial} in the sense that they are defined by monomials from $\sh O_X$. We set $\mathrm{Trop}(C/X/\Delta):=\Spec\alpha_\Delta$.

\begin{eg}[Plane tropical curves]Let $\Delta$ be the lower quadrant
\[ \{\mathbf r=(r_1,r_2)\in\R^2|r_1,r_2\leq\lambda \} \]
with $0\ll\lambda\in\Z$. The polyhedral domain $\D_K\Delta=\Spa K\{t^{r_1}x,t^{r_2}y\}$ is an arbitrarily large polydisc in the affine plane $X=\A^2_K$. Let $X^+$ be the corresponding formal model (which is isomorphic to $\A^2_{\sh O_K}$).

Let $f=\sum_{i,j,k}c_{ijk}t^kx^iy^j\in K\{t^{r_1}x,t^{r_2}y\}$ be some series, where $(c_{ijk})$ is a matrix of constants in $k$. The `tropicalisation' $F$ of the function $f$ in the polytope semiring $\Z_\vee\{\Delta\}=\Z_\vee\{X-r_1,Y-r_2\}$ is \[\iota^\dagger(f)=\bigvee_{i,j,k}iX+jY-k\]
where $\iota:\Z_\vee\{\Delta\}\rightarrow\B^c(j_*\sh O_X;j_*\sh O_X^+)$ is the inclusion. Note that $\iota^\dagger$ is a norm, but not a valuation.

Suppose that $C\hookrightarrow\D_K\Delta$ is a plane curve. Let $J$ be a monomial ideal of $C^+$, $\{t^kx^iy^j\}$ a finite list of generators. A generator $t^kx^iy^j$ may be removed from the list if and only if it is expressible in terms of the other generators, which occurs exactly when the coefficient $c_{ijk}$ of that monomial in some $f\in I_C$ is non-zero. In other words, the relations of the quotient $\Z_\vee\{\Delta\}\rightarrow\alpha_\Delta$ are generated by those of the form
\[ F=\iota^\dagger(f)=\bigvee_{(i,j,k)\neq(i_0,j_0,k_0)}iX+jY-k \]
where $F$ is the tropicalisation of $f$ and $c_{i_0j_0k_0}\neq0$. There are in general infinitely many such relations. The image of $\mathrm{Trop}(C/X/\Delta)$ in the Hausdorff quotient $\mathrm{sk}\Delta\rightarrow\overline\Delta\subset\overline{\R^2}$ is the \emph{non-differentiability locus} of the convex piecewise-affine function on $\overline\Delta$ defined by $F$.

These relations were also obtained by different means in \cite{Giansiracusa}.\end{eg}

In order to globalise this procedure, we need to check the functoriality of the amoeba under inclusion $\Delta^\prime\subseteq\Delta$ of lattice polytopes. The corresponding open immersion $\mathrm{sk}\Delta^\prime\hookrightarrow\mathrm{sk}\Delta$ may be factored into a subdivision at some element $Z\in\Z_\vee^\circ\{\Delta\}$ followed by a cell inclusion. This is the combinatorial shadow of the operation of taking the toric blow-up $\widetilde X^+\rightarrow X^+$ along $Z$, and then restricting to an affine subset.

\begin{lemma}Let $C\hookrightarrow X$ be a closed embedding of adic spaces, $j:X\rightarrow X^+$ a formal model of $X$, $C^+\subseteq X^+$ the closure of $C$ in $X^+$. Let $\widetilde X^+\rightarrow X^+$ be an admissible blow-up with ideal $J$. Then the closure $\widetilde C^+$ of $C$ in $\widetilde X^+$ is the blow-up of $C^+$ along $\sh O_{C^+}J$.\end{lemma}
\begin{proof}The definitions directly imply the following identity
\[ \frac{R_J}{I_C\cap R_J}\cong\frac{\bigoplus_{n\in\N}J^nt^n}{\bigoplus_{n\in\N}I_C\cap J^nt^n}\cong\bigoplus_{n\in\N}\frac{J^n}{I_C\cap J^n}t^n \cong R_{\sh O_{C^+}J} \]
of the Rees algebras on $C^+$.\end{proof}

As we observed in the previous section, the Clemens semiring $\mathrm{Cl}(X;\widetilde X^+)$ is the strict transform semiring of $\Z_\vee\{\Delta\}$ under the monomial blow-up $\widetilde X^+\rightarrow X^+$ (def. \ref{LOC-def-sta}). Furthermore, the formation of the strict transform semiring commutes with the tropicalisation of ideals on $C$: 
\[ j_*\sh O_C^+\bigoplus_{n\in\N}J_nt^n\cong \bigoplus_{n\in\N}\frac{J_n}{I_C\cap J_n}t^n\cong \bigoplus_{n\in\N}j_*\sh O_C^+J_n. \]
By corollary \ref{LOC-strict-transform}, the image of $\B^c(j_*\sh O_C;j_*\sh O_C^+)$ in $\mathrm{Cl}(X;\widetilde X^+)$ is a free localisation of $\alpha_\Delta$.

Now writing $X^\prime=\D_K\Delta^\prime$ and $C^\prime=C\times_XX^\prime$, we obtain a natural morphism of skeleta $\mathrm{Trop}(C^\prime/X^\prime/\Delta^\prime)\rightarrow\mathrm{Trop}(C/X/\Delta)$. The above arguments, together with lemma \ref{i can't be bothered to think of good labels any more} show:

\begin{prop}$\mathrm{Trop}(C^\prime/X^\prime/\Delta^\prime)\rightarrow\mathrm{Trop}(C/X/\Delta)$ is an open immersion.\end{prop}

We can therefore glue tropicalisations as we glue polytopes. In particular, we can construct the amoeba
\[ \mathrm{Trop}(C/X/\Sigma)=\bigcup_{\sigma_{\mathbf r}\subset\Sigma}\mathrm{Trop}\left(C\times_X\D_k\sigma_{\mathbf r}/\D_K\sigma_{\mathbf r}/\sigma_{\mathbf r}\right) \]
of any subscheme of a toric variety, as promised above.

\subsection{Circle}\label{EGS-ellipt}

Returning to the situation of \ref{EGS-Clemens}, let us specialise to the case of an elliptic curve. Let $K$ be a DVF with residue field $k$, $E/K$ an elliptic curve; write $\overline E/\overline K$ for the base change to the algebraic closure. Let $\Omega=\Omega^{1,0}\in\Gamma(E;\omega_{E/K})\setminus\{0\}$ be a holomorphic volume form.

\begin{defn}A formal model $E^+$ of $E$ is \emph{crepant} if it is $\Q$-Gorenstein and one of the following equivalent conditions are true:
\begin{enumerate}\item there exists a log resolution $f:(E^+)^\prime\rightarrow E^+$ on which $\overline{(E^+)^\prime}+f^*\Omega=t^k$ as $\Q$-divisors on $(E^+)^\prime$, where $k\in\Z$ and $\overline{(E^+)^\prime}$ denotes the reduction of $(E^+)^\prime$;
\item The log canonical threshold is equal to a constant $k$ on $E^+$ (in equal characteristic zero);
\item The canonical bundle $\omega_{\overline E^+/\sh O_{\overline K}}$ over the algebraic closure is trivial.\end{enumerate}
A formal model of $\overline E$ is \emph{crepant} if it is finitely presented with trivial canonical bundle over $\sh O_{\overline K}$, or equivalently, it is obtained by flat base extension from a crepant formal model over some finite extension of $\sh O_K$.\end{defn}

A simple normal crossings model $E^+$ of $E$ is crepant if and only if its reduction is a cycle of projective lines. The multiplicity of a line in the central fibre $E^+_0:=k\times_{\sh O_K}E^+$ is one more than its multiplicity in the canonical divisor. One can make the multiplicities all one, and hence trivialise the canonical bundle, by effecting a finite base change followed by a normalisation. In particular, $E^+$ is semistable if and only if $\omega_{E^+/\sh O_K}$ is trivial, that is, if and only if it is a minimal model.

On the other hand, a formal model of $E$ is locally toric if and only if it has at worst monomial cyclic quotient singularities and the components of its central fibre are smooth rational curves. It is automatically $\Q$-Gorenstein. Such a model exists only if $E$ has bad, but semistable reduction; let us assume this.

Let $\mathbf{Mdl}^\mathrm{clt}(E)$ denote the category of crepant, locally toric models of $E$. Let $E^+$ be an object of this category. Its singularities occur at the intersections of components, and they have the form
\[ \D^+_{k[\![t]\!]}\Delta\rightarrow\Delta^\circ \]
where $\Delta=[a,b]$ is an interval with rational endpoints. They may be resolved explicitly, and crepantly, by subdividing the interval at all its integer points.

\



Since, by assumption, a crepant resolution of $E^+$ exists, $(E^+)^\prime$ must be a crepant model. Its reduction is therefore a cycle of $\P^1_k$s. Since $\Delta(E^+)^\prime\rightarrow\Delta E^+$ is a subdivision, it follows that both are cycles of intervals; $\Delta E^+(\R_\vee)$ has the topology of a circle.

The Clemens functor
\[ \Delta:\mathbf{Mdl}^\mathrm{clt}(E) \rightarrow \mathbf{Sk}^\mathrm{aff} \]
is defined, and its image is a diagram of subdivisions. It may therefore be \emph{glued} to obtain the \emph{Kontsevich-Soibelman} (or \emph{KS}) skeleton $\mathrm{sk}(E;\Omega):=\colim\Delta$. It is a \emph{shell} of any crepant Clemens skeleton, and hence comes with a collapse map 
\[ \mu:E\rightarrow\mathrm{sk}(E;\Omega), \]
which is a \emph{torus fibration}: every point of $\mathrm{sk}(E;\Omega)$ has an overconvergent neighbourhood over which $\mu$ is isomorphic over $\Z_\vee$ to a standard torus fibration on an interval. In the introduction we introduced an explicit `atlas' for $\mathrm{sk}(E;\Omega)$ under the assumption (which may be lifted) that the minimal model of $E$ consist of at least three reduced lines.\footnote{I insert the word `atlas' between inverted commas because I do not prove that these open sets really cover $\mathrm{sk}(E;\Omega)$. That statement would be equivalent to the conjecture raised in example \ref{eg-cpa}.}

If, more generally, $E$ has only bad reduction, we can still define $\mathrm{sk}(\overline E;\Omega)$ as the colimit of the Clemens functor on $\mathbf{Mdl}^\mathrm{clt}(\overline E)$. If $L\supseteq K$ is a finite extension over which $E_L:=L\times_KE$ has semistable reduction, then
\[ \mathrm{sk}(\overline E;\Omega)\cong \Q_\vee\times_{\Z_\vee}\mathrm{sk}(E_L;\Omega) \]
by the base change property for polytopes. In particular, the collapse map $\overline E\rightarrow\mathrm{sk}(\overline E;\Omega)$ is again a torus fibration. The KS skeleton is of finite presentation over $\Q_\vee$.

There is a continuous projection $\pi:\mathrm{sk}(\overline E;\Omega)\rightarrow B:=\mathrm{sk}(\overline E;\Omega)(\R_\vee)\simeq S^1$. The local models for $\mu$ induce a canonical smooth structure on $B$ with respect to which the \emph{affine functions}, that is, invertible sections $\mathrm{Aff}_\Z(B,\Q)$ of $\mathrm{CPA}_\Z(B,\Q_\vee):=\pi_*|\sh O|^\mathrm{canc}$, are smooth. It therefore attains an \emph{affine structure} in the sense of \cite[\S2.1]{KoSo2} defined by the exact sequence
\[ 0\longrightarrow\Q\longrightarrow \mathrm{Aff}_\Z(B,\Q) \longrightarrow \Lambda^\vee \longrightarrow 0 \]
of Abelian sheaves on $B$ and the induced embedding $\Lambda^\vee\hookrightarrow T^\vee B$.

Let $E^+_L$ be a semistable minimal model of $E_L$. By writing $\Omega$ locally in the form $\lambda d\log x$ for some monomial $x\in\sh O_E^\times$ and $\lambda\in L^\times$, we can think of $\Omega$ as a non-zero section of $L\tens_\Z\Lambda^\vee$. It induces a $\overline K$-orientation of $B$ that does not depend on the choice of $L$ or $x$.
